\documentclass[review,onefignum,onetabnum,reqno]{siamart252011}
\usepackage{amsmath, amssymb, graphicx, hyperref, bm, multirow, tabularx, booktabs, float,  esint, relsize}
\usepackage{subcaption}

\newcommand{\medint}{\mathop{\mathlarger{\smallint}}\nolimits}

\usepackage[numbers,sort&compress]{natbib}
\newcommand{\ben}{\begin{eqnarray}}
\newcommand{\een}{\end{eqnarray}}
\newcommand{\bea}{\begin{array}}
\newcommand{\eea}{\end{array}}
\newcommand{\bes}{\begin{subequations}}
\newcommand{\ees}{\end{subequations}}
\newcommand{\bef}{\begin{figure}[H]}
\newcommand{\eef}{\end{figure}}
\newcommand{\bet}{\begin{tikzpicture}}
\newcommand{\eet}{\end{tikzpicture}}
\newcommand{\beq}{\begin{equation}}
\newcommand{\eeq}{\end{equation}}

\def\bR{\mathbf{R}}
\def\bA{\mathbf{A}}

\def\bL{\mathbf{L}}

\def\b0{\mathbf{0}}

\def\bx{\mathbf{x}}

\def\ba{\mathbf{a}}
\def\bb{\mathbf{b}}

\def\bJ{\mathbf{J}}
\def\bS{\mathbf{S}}
\def\bI{\mathbf{I}}

\def\bM{\mathbf{M}}

\title{Skew Gradient Embedding for Thermodynamically Consistent Systems}

\author{Xuelong Gu\thanks{Department of Mathematics, University of South Carolina (xgu@mailbox.sc.edu)}
\and Qi Wang\thanks{Corresponding author. Department of Mathematics, University of South Carolina (qwang@math.sc.edu).}}

\ifpdf
\hypersetup{
  pdftitle={Numerical Schemes for General Gradient Reformulations},
  pdfauthor={Xuelong Gu, Qi Wang}
}
\fi

\graphicspath{{FIGS/}}

\nolinenumbers

\begin{document}

\maketitle

\begin{abstract}
	We propose a novel Skew Gradient Embedding (SGE) framework for systematically reformulating thermodynamically consistent partial differential equation (PDE) models-capturing both reversible and irreversible processes-as generalized gradient flows. These models include a wide spectrum of models in classical electrodynamics, fluid mechanics, quantum mechanics, rheology of complex fluids, solid mechanics, and statistical physics. Exploiting the intrinsic structure of generalized gradient flow models, especially, the skew symmetric component expressed by the exterior 2-form, we develop a unified stabilization strategy for constructing numerical schemes that either preserve the energy dissipation rate or ensure discrete energy stability.	This stabilization strategy enables the design of both first- and second-order schemes, highlighting the flexibility and generality of the SGE approach in algorithm development. A key strength of SGE is its flexible treatment of skew-gradient (zero-energy-contribution) terms arising from reversible dynamics either implicitly or explicitly. While treated explicitly, it often leads to a natural decoupling of the governing equations in multiphysics systems, thereby improving computational efficiency without compromising stability or accuracy. Numerical experiments confirm the robustness, accuracy, and performance advantages of the proposed schemes.

\end{abstract}

\section{Introduction}
In classical electrodynamics, fluid and solid mechanics, quantum mechanics, rheology of complex fluids, statical physics, the conservation laws together with constitutive equations consist of a thermodynamically consistent partial differential equation system of the following form:
\begin{equation}\label{eq:intro-general1}
	\mathbf{R}\cdot \partial_t \Phi = \mathbf{M} (\Phi)\cdot \nabla F(\Phi) + \mathbf{J}(\Phi), \ \Phi(0)=\Phi_0,
\end{equation}
where
\begin{itemize}
	\item $\Phi$ represents the state variable in a Hilbert space $H$ equipped with inner product $(\cdot, \cdot)_H$ and $\nabla$ denotes the spatial gradient operator with respect to $\Phi$ (a variational operator);
	\item  $F(\Phi)$ is the energy of the system;
	\item  $\mathbf{M}(\Phi)$ is the mobility operator assumed symmetric and negative semi-definite;
	\item $\mathbf{R}$ is the identity operator in nondegenerate cases, otherwise, it satisfies
	      \ben
	      (\nabla F(\Phi),  \partial_t{\Phi} )_H = (\nabla F(\Phi),  \mathbf{R} \cdot \partial_t{\Phi} )_H;
	      \een
	\item $\mathbf{J}(\Phi)$, termed the zero-energy-contribution (ZEC) term (see \cite{Yang-2021-CPC, Yang-2021-JCP}), satisfies the orthogonality condition $( \mathbf{J}(\Phi), \nabla F(\Phi) )_H = 0$.
\end{itemize}
It is straightforward to verify that system \eqref{eq:intro-general1} adheres to the following energy dissipation law
\begin{equation}\label{eq:intro-energy-law}
	\tfrac{d F}{dt} = (\nabla F(\Phi), \mathbf{R}\cdot \dot{\Phi})_H = (\nabla F(\Phi), \mathbf{M}(\Phi)\nabla F(\Phi))_H \leq 0.
\end{equation}
many models are
In this paper,we study  this class of thermodynamically consistent systems  and their structure-preserving
numerical approximations.

Near-equilibrium nonequilibrium phenomena are typically described by constitutive models derived from the Onsager's linear response theory \cite{OnsagerI-1931, OnsagerII-1931, Wang-2020} together with various conservation laws, like mass, momentum, energy, etc. The constitutive equation derived from the generalized Onsager principle \cite{Wang-2020} normally consists of a gradient flow system \cite{Lu-2025-SISC}:
\begin{equation}\label{eq:intro-gradient-system}
	\partial_t \tilde{\Phi} = \mathbf{L}(\Phi) \nabla F(\tilde{\Phi}),
\end{equation}
where $\tilde{\Phi}$ is a subset of $\Phi$ and $\mathbf{L}(\tilde \Phi)$ the mobility operator.  Collectively, the full governing system of equations emerges in the form of \eqref{eq:intro-general1}. There are numerous examples in the literature in which the governing system of equations exhibits the structure of \eqref{eq:intro-general1}, such as the Navier-Stokes equations for viscous fluid flows \cite{Rorger-2001}, thermodynamically consistent hydrodynamic models for complex fluids \cite{Jiang-2024-yield, Zhou-2018-Onsager},  multiphase fluid flow models \cite{Chen-2020-CHNS, Kay-2007-CHNS, Shen-2015-CHNS, Gong-2018-CHNS, Gong-2016-Binary, Zhao-2021-CHNS, Yang-2025-JCP, Yang-2021-Multi, Hong-2024-JCP, Hong-2023-JCP, Yang-2023-JCP}, and Smoluchowski systems for complex fluid flows \cite{wang2002hydrodynamic,wang2002kinetic,wang2002hydrodynamic2}, Maxwell equations \cite{jackson1998classical}, Schr\"odinger equations \cite{wang2021central}, MHD systems, etc.


Note  that if $\bJ$ can be written into the following skew gradient form
\ben
\bJ=\bS(\Phi)\cdot \nabla F,
\een
where $\mathbf{S}(\Phi)$ is a skew-symmetric or antisymmetric operator,
the entire governing system of equation in \eqref{eq:intro-general1} can be recast into a generalized gradient flow with a symmetric and skew-symmetric component in the mobility operator as follows, reflecting irreversible and reversible dynamics, respectively,
\begin{equation}\label{eq:intro-gradient-system2}
	\bR\cdot \partial_t \Phi = \mathbf{M}(\Phi) \nabla F(\Phi) + \mathbf{S}(\Phi) \nabla F(\Phi),
\end{equation}
where $\mathbf{M}(\Phi) = \bM^\top \leq 0$, $\mathbf{S}(\Phi) =-\bS^\top$. We refer to this reformulation process as the skew gradient embedding.

The zero energy contribution (ZEC) term arises naturally since 1) the generalized Onsager principle in \eqref{eq:intro-gradient-system} often includes transport mechanisms for reversible processes in the form of an antisymmetric mobility operator encompassing material derivatives to accurately capture transport, convection and rotation of the material in Eulerian coordinates \cite{Wang-2020,Jiang-2024-yield}; 2) the conservation laws comprise convective derivatives contributing zero energy dissipation to the system. Thus, the general form in \eqref{eq:intro-gradient-system2} encompasses a wider range  of physically significant models beyond the classical gradient flow model. This class of models is thus named the generalized gradient flow.

When $\bM=0$, model \eqref{eq:intro-gradient-system2} is energy conservative, making it particularly suitable for the numerical approximations developed for conservative systems \cite{Hairer-2006-SP}. In fact, for a subclass of conservative systems—namely Hamiltonian systems—numerical approximations have been extensively investigated in the literature \cite{Hairer-2006-SP}.  For dissipative systems ($\bM\neq 0$), however, it is often sufficient to construct numerical schemes that merely guarantee discrete energy decay rather than following the exact energy dissipation rate \eqref{eq:intro-energy-law}. These methods, referred to as energy-stable schemes, ensure the monotonic decrease of a discrete energy functional over time, although the exact dissipation rate may differ from that of the continuous system. Classical examples of energy-stable schemes include the convex splitting methods \cite{CS1,CS2}, stabilization methods \cite{Hou-2020,Xiao-2017,Tang-2016,Lu-Stabilization,Wang-CH-Stabilization}, and exponential time difference methods \cite{Ju-2018,Du-2019,Du-2021}. A comprehensive review of the numerical methods for energy-stable schemes is available in \cite{Tang-2020}.

Most of the structure-preserving schemes developed so far are fully implicit. It is challenging to extend these schemes to general models with nonlinear energies or difficult to extend them to high orders. Inspired by ideas from \cite{EQ1}, the energy quadratization (EQ) approach was proposed in \cite{EQ2,EQ3,EQ4} to construct linearly implicit schemes applicable to gradient flows with general nonlinear terms, preserving a modified quadratic energy. Subsequently, the scalar auxiliary variable (SAV) approach \cite{SAV-Shen,SAV-Li,SAV-NLSW,ESAV-High,ESAV-KG} introduced a simplified implementation of energy quadratization, requiring the solution of only a linear system with a constant-coefficient matrix and a scalar nonlinear equation at each time step.

However, both EQ and SAV methods preserve only a modified quadratic energy and their stability hinges upon the accuracy of the augmented ODE for the new EQ variable, potentially leading to accuracy and stability issues. To improve consistency between the original and quadratic energies, relaxation techniques were introduced in \cite{Jiang-2022,Zhang-2022-GSAV}. Recently, new frameworks capable of constructing linearly implicit schemes, while preserving the original energy law were developed, including supplementary variable methods \cite{Gong-2021} and relaxation Runge-Kutta methods \cite{Li-2023,Li-2023-JCP,Li-2024-SISC}.

So far, all aforementioned approaches have been developed specifically for gradient flow systems with occasional extensions to some hydrodynamical models. Consequently, little effort has been dedicated to devising general thermodynamically consistent algorithms for more general systems like \eqref{eq:intro-general1}. Occasionally, existing methods are tailored to specific systems of form \eqref{eq:intro-general1}. For instance, the Navier-Stokes equation commonly requires rewriting convection terms into skew-symmetric forms to construct energy-stable discretization (see \cite{Gong-2016-Binary,Gong-2018-CHNS}). Similarly, numerical schemes for multi-phase flow models typically exploit particular structure of $\mathbf{J}(\Phi)$. When energy functionals become more complex, such as in liquid crystal models, some temporal discrete schemes were designed to preserve discrete zero energy contribution properties; however, these structure-preserving properties may fail to be preserved after the spatial discretization (see \cite{Zhao-2016-JSC}).

Moreover, implicit discretization of $\mathbf{J}(\Phi)$ is necessary to ensure thermodynamic consistency in some cases \cite{Gong-2016-Binary,Gong-2018-CHNS,Hong-2023-JCP,Hong-2024-JCP}. Since $\mathbf{J}(\Phi)$ is generally neither symmetric nor positive definite, such implicit implementation typically results in challenging linear or nonlinear systems, significantly eroding computational efficiency. Although recent SAV-like frameworks allow explicit implementation of $\mathbf{J}(\Phi)$ \cite{Shen-2015-CHNS,Yang-2019-JCP,Zhang-2022-GSAV,Lu-2025-SISC,SAV-NS}, these methods typically preserve only the modified energy rather than the original energy functional.

In this paper, we propose a general framework to systematically construct thermodynamically consistent numerical algorithms for the general thermodynamically consistent model in \eqref{eq:intro-general1}. The first strategy in our approach is to recast the zero energy contribution term, $\mathbf{J}(\Phi)$, into an equivalent skew-gradient form with respect to chemical potential $F(\Phi)$ using an exterior 2-form.  Through this reformulation, \eqref{eq:intro-general1} is transformed into an equivalent generalized gradient flow system. Then, we devise another systematic approach to devise structure-preserving numerical algorithms for the generalized gradient flow system using the   discrete gradient method and stabilization strategy guided by the convex-splitting idea. This suite of theoretical and computational approaches is termed skew gradient embedding. The major advantages of our proposed SGE framework are summarized as follows:
\begin{itemize}
	\item SGE provides a universal and systematic strategy for transforming the general ZEC term $\mathbf{J}(\Phi)$ into a skew-gradient form with respect to the given energy $F(\Phi)$ via an exterior 2-form, resulting in a skew gradient component in the model. The skew-gradient reformulation naturally preserves the discrete ZEC property independently of the spatial discretization.
	\item The proposed framework facilitates the construction of thermodynamically consistent algorithms with fully explicit implementation of $\mathbf{J}(\Phi)$. In addition, if the energy gradient is discretized using a linearly implicit scheme, the resulting numerical scheme achieves efficiency comparable to the SAV-based approaches
	\item The discrete energy dissipation law depends solely on the discretization of the energy gradient $\nabla F$. Specifically, if a discrete-gradient approach is adopted for $\nabla F$, SGE rigorously preserves the original energy-dissipation law instead of a modified version.
	\item Guided by the idea of convex-splitting, a set of stabilization schemes are devised to provide a set of first and second order energy stable schemes for system \eqref{eq:intro-general1} with respect to the original energy functionals.
\end{itemize}

The remainder of the paper is structured as follows. In \S 2, we introduce essential notations and present SGE. In addition, we present multiple illustrative examples from complex fluid dynamics, demonstrating how one use SGE to systematically reformulate these systems into equivalent generalized gradient flow systems. In \S 3, we develop thermodynamically consistent schemes for the generalized gradient flow system obtained via the SGE reformulation. In particular, we detail how stabilization can be incorporated into the thermodynamically consistent schemes guided by convex-splitting. \S 4 is dedicated to applying some developed numerical schemes  to several representative model systems; we detail the construction and numerical implementation procedures therein. Numerical experiments and comparative analyses are performed to highlight the efficiency, accuracy, and robustness of the proposed methods. Finally, we conclude the study in \S5.

\section{Skew gradient embedding (SGE)}
In this section, we introduce the skew gradient embedding (SGE) method for thermodynamically consistent systems in detail. We then  demonstrate the implementation of the SGE framework by reformulating various systems into equivalent generalized gradient flow models using exterior 2-forms via several examples drawn from the field of complex fluid flows systematically.

\subsection{Preliminary}
We begin by introducing some useful notations. Let \( (H, (\cdot,\cdot)_H) \) denote a Hilbert space for functions and \( \nabla \) represent the corresponding gradient operator defined via the Riesz representation theorem, allowing identification of \( H \) with its dual $H^*$. For elements \( \Phi, \Psi, \Gamma, \Phi_i \in H \) (\( i=1,2 \)), we write inner product \( (\Phi, \Psi)_H \) as \( \Phi^\top \Psi \) and define wedge product $\Phi \wedge \Psi : H \times H \to \mathbb{R}$ through
\ben
(\Phi_1 \wedge \Phi_2)(\Psi_1, \Psi_2)_H := \det\bigl[((\Phi_i, \Psi_j)_H)_{1\le i,j\le2}\bigr].
\een
Moreover, for any $\Psi \in H$, we define \( (\Phi_1 \wedge \Phi_2)\Psi \in H \), such that \( ((\Phi_1 \wedge \Phi_2)\Psi, \Gamma)_H = (\Phi_1 \wedge \Phi_2)(\Psi, \Gamma)_H \) for any \( \Gamma \in H \). Finally, given a linear operator \( \mathbf{L}:H \to H \), its adjoint \( \mathbf{L}^\top \) is defined by \( (\mathbf{L}\Phi, \Psi)_H = (\Phi, \mathbf{L}^\top\Psi)_H \).  The operator \( \mathbf{L} \) is said to be symmetric if \( \mathbf{L} = \mathbf{L}^\top \) and skew-symmetric if \( \mathbf{L}^\top = -\mathbf{L} \). It is straightforward to verify that wedge product \( \Phi \wedge \Psi \) is a skew-symmetric operator on \( H \) and therefore an exterior 2-form.

A nonequilibrium system is termed a generalized gradient flow system if it can be expressed in the following form:
\begin{equation}\label{eq:gradient-general}
	\mathbf{R} \cdot \Phi_t = \mathbf{L}(\Phi) \nabla F(\Phi),
\end{equation}
where $\bL \cdot \bR$ is the friction operator, $\Phi$ denotes the state variable, and $F(\Phi)$ is the associated energy functional. Without loss of generality, we assume the free energy, $F(\Phi) = \frac{1}{2} (\Phi, \mathcal{L} \Phi)_H + f(\Phi)$, where $\mathcal{L}$ is a symmetric and positive definite linear operator, and $f$ is a nonlinear function. In this case, the chemical potential is given by $\nabla F=\frac{\delta}{\delta \Phi }F(\Phi) = \mathcal{L} \Phi + \nabla f(\Phi)$. We call the generalized gradient flow system thermodynamically consistent if
\ben
(\nabla F, \Phi_t)_H=(\nabla F, \mathbf{R} \cdot \Phi_t)_H.
\een
Notice that $\bL(\Phi)$ can be decomposed as follows: $\mathbf{L}(\Phi) = \mathbf{M}(\Phi) + \mathbf{S}(\Phi)$, where
\begin{equation*}
	\mathbf{M}(\Phi) = \tfrac{1}{2}(\mathbf{L}(\Phi) + \mathbf{L}(\Phi)^\top), \ \mathbf{S}(\Phi) = \tfrac{1}{2}(\mathbf{L}(\Phi) - \mathbf{L}(\Phi)^\top).
\end{equation*}
Then, \eqref{eq:gradient-general} becomes
\begin{equation}\label{eq:gradient-split}
	\bR \cdot \Phi_t = (\mathbf{M}(\Phi)  + \mathbf{S}(\Phi) )\nabla F(\Phi).
\end{equation}

A primary motivation for considering the generalized gradient system is that \eqref{eq:gradient-split} provides a unified framework for  developing energy conserving or dissipative numerical algorithms. A representative example is the following discrete gradient flow model:
\begin{equation}\label{eq:gradient}
	\bR \cdot \tfrac{\Phi^{n+1} - \Phi^n}{\tau} = (\mathbf{M}(\Phi^{n+\frac{1}{2}}) + \mathbf{S}(\Phi^{n+\frac{1}{2}})) \overline{\nabla} F(\Phi^n, \Phi^{n+1}),
\end{equation}
where $\overline{\nabla} F(\bullet, \bullet)$ denotes the discrete gradient, defined by
\begin{equation}\label{eq:discrete-gradient}
	\overline{\nabla} F(\Phi, \Phi) = \nabla F(\Phi), \quad (\overline{\nabla} F(\Psi, \Phi), \Psi - \Phi)_H = F(\Psi) - F(\Phi).
\end{equation}
Taking the inner product of \eqref{eq:gradient} with $\overline{\nabla} F(\Phi^{n+1}, \Phi^n)$ on both sides and utilizing the assumption to $\mathbf{R}$, we show that scheme \eqref{eq:gradient} preserves the semi-discrete energy law
\begin{equation}
	F(\Phi^{n+1}) - F(\Phi^n) =
	\left\lbrace
	\begin{aligned}
		 & \tau(\overline{\nabla} F, \mathbf{M}(\Phi^{n+\frac{1}{2}}) \overline{\nabla} F)_H \leq 0, \  &  & \mathbf{M}_{sym} \neq \mathbf{0}, \\
		 & 0,  \                                                                                        &  & \mathbf{M}_{sym} = \mathbf{0}.
	\end{aligned}
	\right.
\end{equation}
Note that a given energy functional may admit multiple discrete gradients satisfying \eqref{eq:discrete-gradient}, among which the averaged vector field (AVF) discrete gradient is particularly useful:
\begin{equation}
	\overline{\nabla} F(\Psi, \Phi) = \medint_0^1 \nabla F(\xi \Psi + (1 - \xi) \Phi) d\xi.\label{dF}
\end{equation}

\subsection{Generalized gradient flow formulation of a general dissipative system}

Notice that most dissipative systems are not originally written as generalized  gradient flow forms \eqref{eq:gradient-general}. For instance, the viscous Burgers equation, the incompressible Naiver-Stokes (NS) equation,   Cahn-Hilliard-Navier-Stokes (CH-NS) equations, thermodynamically consistent equation systems governing complex fluids (e.g., Oldroyd-B, Smoluchowski equations), etc. Nonetheless, these systems share the structure of system \eqref{eq:intro-general1} and
satisfy the following orthogonal condition
\begin{equation}\label{eq:orthogonal}
	(\nabla F(\Phi), \mathbf{J}(\Phi))_H = 0.
\end{equation}
Next, we  present Theorem \ref{thm:grad-reformulation}, which demonstrates how to recast \eqref{eq:intro-general1} into  a generalized gradient flow model.

\begin{theorem}\label{thm:grad-reformulation}
	Assume $\|\nabla F(\Phi)\|_H \neq 0$. Then, \eqref{eq:intro-general1} can be formulated into the following generalized gradient flow:
	\begin{equation}\label{eq:general-reformulation}
		\mathbf{R} \cdot \dot{\Phi} = \mathbf{L}(\Phi) \nabla F(\Phi) := (\mathbf{M}(\Phi) + \mathbf{S}(\Phi))\nabla F(\Phi),
	\end{equation}
	where $\mathbf{S}(\Phi) = \frac{\nabla F(\Phi) \wedge \mathbf{J}(\Phi)}{\|\nabla F(\Phi)\|_H^2}$ is a skew-symmetric operator.
\end{theorem}
\begin{proof}
	The skew-symmetry of $\mathbf{S}(\Phi)$ follows directly from its definition. It remains to show that $\mathbf{S}(\Phi)\nabla F(\Phi) = \mathbf{J}(\Phi)$. For any $\Psi \in H$, we have
	\begin{equation*}
		\begin{aligned}
			(\mathbf{S}(\Phi)\nabla F(\Phi), \Psi)_H & = \frac{1}{\|\nabla F(\Phi)\|_H^2} ( (\nabla F(\Phi) \wedge \mathbf{J}(\Phi)) \nabla F(\Phi), \Psi)_H                                                                                                                                                                                                                \\
			                                         & = \frac{1}{\|\nabla F(\Phi)\|_H^2}  \left| \begin{matrix} (\nabla F(\Phi), \nabla F(\Phi))_H & (\nabla F(\Phi), \Psi)_H \\ (\mathbf{J}(\Phi), \nabla F(\Phi))_H & (\mathbf{J}(\Phi), \Psi)_H \end{matrix} \right|                                                                                                    \\
			                                         & = \frac{1}{\|\nabla F(\Phi)\|_H^2} \left| \begin{matrix} \|\nabla F(\Phi)\|_H^2                                                                                                                     & (\nabla F(\Phi), \Psi)_H \\ 0 & (\mathbf{J}(\Phi), \Psi)_H  \end{matrix} \right| = (\mathbf{J}(\Phi), \Psi)_H.
		\end{aligned}
	\end{equation*}
	Moreover, when $\nabla F(\Phi) = 0$, we define $\mathbf{S}(\Phi) = 0$. This completes the proof.
\end{proof}

In previous works, preserving the original energy properties often requires discretizing $\mathbf{J}(\Phi)$ in \eqref{eq:intro-general1} implicitly to guarantee \eqref{eq:orthogonal} at the discrete level. However, since $\mathbf{J}(\Phi)$ is absorbed into antisymmetric operator $\mathbf{S}(\Phi)$, it becomes feasible to discretize $\mathbf{J}(\Phi)$ explicitly without affecting the energy dissipative property. The explicit discretization can then enhance the efficiency of the obtained numerical scheme, especially when $\mathbf{J}(\Phi)$ is a non-local or nonlinear operator.

There are multiple ways to represent \eqref{eq:intro-general1} as a generalized gradient flow system.
We note that many recent numerical methods can be interpreted within this framework. For example, in \cite{Yang-2021-CPC}, Yang et al. introduced an scalar auxiliary variable $q(t) = 1$ to decouple the antisymmetric part from the symmetric part for fully decouple energy-stable schemes by reformulating \eqref{eq:intro-general1} into the following gradient flow system:
\begin{equation}\label{eq:sav-yang}
	\left\lbrace
	\begin{aligned}
		\dot{\Phi} & = \mathbf{M}(\Phi) \nabla F(\Phi) + q \mathbf{J}(\Phi), \\
		\dot{q}    & = -\mathbf{J}(\Phi)^\top \nabla F(\Phi).                \\
	\end{aligned}
	\right.
\end{equation}
With an initial condition $q(0) = 1$, \eqref{eq:sav-yang} is equivalent to \eqref{eq:orthogonal} since $\mathbf{J}(\Phi)^\top \nabla F(\Phi)=0$. Notice that \eqref{eq:sav-yang} is the following gradient system in higher-dimensional space
\begin{equation}\label{eq:gradient-yang}
	\frac{\partial }{\partial t} \begin{pmatrix} \Phi \\ q \end{pmatrix} = \begin{pmatrix} \mathbf{M}(\Phi) & 0 \\ 0 & 0 \end{pmatrix}
	\begin{pmatrix}
		\nabla F(\Phi) \\ f^\prime(q)
	\end{pmatrix}
	+
	\begin{pmatrix}
		0                      & \mathbf{J}(\Phi) \\
		-\mathbf{J}(\Phi)^\top & 0
	\end{pmatrix}
	\begin{pmatrix}
		\nabla F(\Phi) \\ f^\prime(q)
	\end{pmatrix},
\end{equation}
where $f(q) = \frac{1}{2}|q|^2$. This indicates that \eqref{eq:sav-yang} is a gradient flow system with a modified free energy $\tilde{F}(\Phi, q) = F(\Phi) + f(q)$ and mobility in a higher dimensional phase space. However, due to the modified structure, explicit  discretization of $\mathbf{J}(\Phi)$ often fails to conserve (or dissipate) the original energy.

\begin{remark}
	We note that the proposed technique can be combined with SAV or EQ methods, as well as with the approach in \cite{Lu-2025-SISC}, to produce various generalized gradient flow reformulations of \eqref{eq:intro-general1}. These reformulations facilitate the construction of high order, efficient conservative or dissipative numerical schemes that explicitly discretize $\mathbf{J}(\Phi)$ via the methods presented below. However, due to space limitations, we will not detail the SAV nor EQ approach in this paper.
\end{remark}

We next present several examples of dissipative partial differential equations and show how to recast them into generalized gradient flows using the above theorem.

\begin{example}[Viscous Burgers equation]\label{ex:burgers}
	We consider the viscous Burgers equation subject to a periodic boundary condition as follows,
	\begin{equation}\label{eq:burgers}
		u_t + u u_x = \nu u_{xx},
	\end{equation}
	where  $\nu > 0$ is the viscosity. We set $H = L_p^2$ equipped with the inner-product defined by $(u, v)_{L^2} = \int_\Omega uv dx$, where $p$ stands for periodic boundary conditions, and define the gradient of energy functional $F(u) = \frac{1}{2}\|u\|_{L^2}^2$ as its variational derivative $\frac{\delta F(u)}{\delta u}$.

	Taking the inner-product of \eqref{eq:burgers} with $\frac{\delta F(u)}{\delta u}$, we arrive at $\frac{d F(u)}{dt} = -\nu \|u_x\|_{L^2}^2$. This confirms the dissipative property of $\eqref{eq:burgers}$ based on orthogonality property $(u\partial_xu, u) = 0$. However, $u\partial_x$ is not skew-symmetry, i.e., $v, w \in H_p^1 \subset L_p^2$, $(u \partial_x v, w) \neq -(v, u\partial_x w)$. Consequently, the Burgers equation in its standard form \eqref{eq:burgers} is not in a generalized gradient flow form yet.

	Using the technique introduced above, we recast \eqref{eq:burgers} into an equivalent generalized gradient system as follows
	\begin{equation}\label{eq:burgers-continuous-gradient01}
		u_t = (\nu \partial_{xx} - \tfrac{u\wedge (uu_x)}{\|u\|_{L_p^2}^2} 	) u = \nu\partial_{xx} u - (u, \tfrac{u}{\|u\|_{L^2}^2})_{L_p^2} uu_x + (uu_x, \tfrac{u}{\|u\|^2_{L_p^2}})_{L_p^2} u.
	\end{equation}
	Alternatively, \eqref{eq:burgers} can be recast into yet another generalized gradient flow form:
	\begin{equation}\label{eq:burgers-continuous-gradient02}
		u_t = \nu \partial_{xx} u + \mathbf{S}(u) u,
	\end{equation}
	where $\mathbf{S}(u) v = -\frac{1}{3} (u\partial_x v + \partial_x (uv))$ is skew-symmetry. Note that with formulation \eqref{eq:burgers-continuous-gradient02}, the convective term must be discretized in a linearly implicit manner to ensure energy dissipation. In contrast,   \eqref{eq:burgers-continuous-gradient01} facilitates the development of energy-dissipative schemes through an explicit discretization of the skew gradient term.
\end{example}


\begin{example}[Incompressible NS equation]
	We consider the dimensionless incompressible NS equation with a periodic boundary condition as follows
	\begin{equation}\label{eq:ns-origin}
		\left\lbrace
		\begin{aligned}
			 & \mathbf{v}_t + (\mathbf{v} \cdot \nabla) \mathbf{v} = -\nabla p + \tfrac{1}{Re} \Delta \mathbf{v}, \\
			 & \nabla \cdot \mathbf{v} = 0,
		\end{aligned}
		\right.
	\end{equation}
	where $Re$ is the Reynolds number.
	We introduce variable $\mathbf{\Phi} = (\mathbf{v}, p)$ and define space $H = [ L_p^2(\Omega) ]^3$, with the inner product defined by $(\Phi, \Psi) := \int_\Omega \Phi \cdot \Psi d\mathbf{x}$, where $\Phi \cdot \Psi$ denotes the dot product between vectors.

	We define the "free" energy by $F(\Phi) = \int_\Omega \tfrac{1}{2} \|\mathbf{v}\|^2 d\mathbf{x}$, \eqref{eq:ns-origin} can be recast into the following compact form
	\begin{equation*}
		\mathbf{R}\Phi_t = \mathbf{M} \delta_{\Phi} F +  \mathbf{J}(\Phi),
	\end{equation*}
	where
	\begin{equation}\label{eq:NS-operator}
		\mathbf{R} =
		\begin{pmatrix}
			\mathbf{I} & 0 \\
			0          & 0
		\end{pmatrix}, \quad
		\mathbf{M} =
		\begin{pmatrix}
			\tfrac{1}{Re}\Delta & 0 \\
			0                   & 0 \\
		\end{pmatrix}, \quad
		\mathbf{J}(\Phi) = -
		\begin{pmatrix}
			(\mathbf{v} \cdot \nabla) \mathbf{v} \\ 0
		\end{pmatrix}
		-
		\begin{pmatrix}
			0            & \nabla \\
			\nabla \cdot & 0
		\end{pmatrix}
		\Phi.
	\end{equation}
	Since $(\mathbf{J}(\Phi), \Phi) = 0$, the NS equation can be reformulated into a generalized gradient flow:
	\begin{equation}\label{eq:ns-generalized-gdflow}
		\mathbf{R} \partial_t \Phi = \left(\mathbf{M} + \tfrac{(\mathbf{v}, 0) \wedge \mathbf{J}(\Phi)}{\|\mathbf{v}\|^2} \right) \delta_\Phi F = (\mathbf{M} + \mathbf{S}_1(\Phi) + \mathbf{S}_2(\Phi)) \delta_{\Phi} F,
	\end{equation}
	where $\mathbf{S}_1(\Phi) = - \tfrac{1}{\|\mathbf{v}\|^2} (\mathbf{v}, 0) \wedge (\mathbf{v} \cdot \nabla \mathbf{v}, 0)$ and $\mathbf{S}_2(\Phi) = -\tfrac{1}{\|\mathbf{v}\|^2} (\mathbf{v}, 0) \wedge (\nabla p, \nabla \cdot \mathbf{v})$, both of which are antisymmetric.
\end{example}
\begin{example}[CH-NS equation]
	The CH-NS system describing hydrodynamics of a binary, incompressible, viscous fluid flow is given by
	\begin{equation}\label{eq:CH-NS-origin}
		\left\lbrace
		\begin{aligned}
			 & \rho(\mathbf{v}_t + \mathbf{v} \cdot \nabla \mathbf{v}) = -\nabla p + \nu \Delta \mathbf{v} - \phi \nabla \mu, \\
			 & \nabla \cdot \mathbf{v} = 0,                                                                                   \\
			 & \phi_t + \nabla \cdot (\mathbf{v} \phi) = M \Delta \mu,                                                        \\
			 & \mu = - \gamma \varepsilon \Delta \phi + \tfrac{\gamma}{\varepsilon} f^\prime(\phi),
		\end{aligned}
		\right.
	\end{equation}
	where $\rho$ is the constant density, $\nu$ the viscosity, $M$ the mobility, $\varepsilon$ the interface parameter and $\gamma$ a parameter characterizing the surface tension. Let $\Phi = (\mathbf{v}, \phi, p)^\top$, free energy $F(\Phi) = \int_\Omega [\tfrac{\rho}{2}|\mathbf{v}|^2 + \tfrac{\varepsilon \gamma}{2} |\nabla \phi|^2 + \tfrac{\gamma}{\epsilon} f(\phi)] d\mathbf{x}$. We have  $\nabla_{\Phi}F = (\nabla_{\mathbf{v}} F, \nabla_\phi F, \nabla_p F) = (\rho\mathbf{v}, \mu, 0)$. Then, \eqref{eq:CH-NS-origin} can be recast into the following generalized gradient flow
	\begin{equation}\label{eq:ch-gradient}
		\mathbf{R}\Phi_t = (\mathbf{M} + \tfrac{ \nabla_{\Phi} F \wedge \mathbf{J}(\Phi) }{\rho^2 \|\mathbf{v}\|^2 + \|\mu\|^2} ) \nabla_{\Phi}F = (\mathbf{M} + \mathbf{S}_1(\Phi) + \mathbf{S}_2(\Phi) ) \nabla_{\Phi} F,
	\end{equation}
	where $\bR=diag\{ \bI, 1, 0\}$,  $\mathbf{M} = diag\{ \tfrac{\nu}{\rho}\bm{\Delta}, M \Delta, 0 \}$, $\mathbf{S}_i(\Phi) = \tfrac{\nabla_\Phi F \wedge \mathbf{J}_i(\Phi)}{\rho^2 \|\mathbf{v}\|^2 + \|\mu\|^2}, i = 1, 2$, and $\mathbf{J}(\Phi) = \mathbf{J}_1 (\Phi) + \mathbf{J}_2(\Phi)$
	\begin{equation}\label{eq:CH-NS-operator}
		\mathbf{J}_1(\Phi) = - \begin{pmatrix} \mathbf{v} \cdot \nabla \mathbf{v} + \tfrac{1}{\rho}\phi \nabla \mu \\  \nabla \cdot (\mathbf{v} \phi) \\ 0 \end{pmatrix},  \ \mathbf{J}_2(\Phi) = \begin{pmatrix} - \tfrac{1}{\rho} \nabla p \\ 0 \\ \nabla \cdot \mathbf{v} \end{pmatrix}.
	\end{equation}
\end{example}
\begin{remark}
	It is important to note that while \eqref{eq:ch-gradient} facilitates the construction of an explicit energy dissipative scheme, the resulting scheme may fail to preserve the volume of each fluid phase defined by $\int_\Omega \phi d\bx$.   Consequently, in practical computations, we replace $\mu$ by  $\overline{\mu} = \mu - \frac{1}{|\Omega|}\int_\Omega \mu d\mathbf{x}$ in \eqref{eq:ch-gradient} to safeguard the volume conservation since \eqref{eq:CH-NS-origin} remains valid when a spatial independent constant is added to $\mu$.
\end{remark}

\section{Numerical Schemes for generalized  gradient flow systems}
Given generalized gradient flow models, we develop some structure-preserving numerical schemes for them based on reformulation \eqref{eq:general-reformulation}. Let $N_t$ be a positive integer and $T$ the terminal time. We  partition time domain $[0, T)$ uniformly into a set of grid points with step size $\tau = T/N_t$. Let $\{t_n = n \tau | n = 0, \cdots, N_t-1\}$ be the grid points and $I_{n+\frac{1}{2}} = [t_n, t_{n+1}]$. We define the following finite difference operators in time:
\begin{equation}
	\begin{gathered}
		\delta_\tau (\bullet)^{n+1} = \tfrac{(\bullet)^{n+1} - (\bullet)^n}{\tau}, \ \delta_{2\tau} (\bullet)^{n+1} = \tfrac{3(\bullet)^{n+1} - 4(\bullet)^n + (\bullet)^{n-1}}{2\tau},
		d_t(\bullet)^{n+1} = (\bullet)^{n+1} - (\bullet)^n,\\
		\ d_{tt}(\bullet)^n =  (\bullet)^{n+1} - 2 (\bullet)^n + (\bullet)^{n-1},
		d_{2t} (\bullet)^{n+1} = 3(\bullet)^{n+1} - 4 (\bullet)^n + (\bullet)^{n-1}.
	\end{gathered}
\end{equation}

Next, we propose several schemes that preserve energy-conservation or dissipation  of the generalized gradient flow system.

\subsection{Dissipation rate preserving schemes}
We begin by proposing several dissipation rate preserving schemes. These schemes strictly preserve the energy dissipation rate at the discrete level, making them suitable for both dissipative and conservative systems.

\begin{scheme}[Discrete-gradient scheme] Suppose that we have solved $\Phi^n$, then, we update it to $\Phi^{n+1}$ by
	\begin{equation}\label{eq:dg}
		\begin{aligned}
			\delta_\tau \Phi^{n+1} & = \mathbf{L}(\Phi^{n+\frac{1}{2}}) \overline{\nabla} F(\Phi^{n+1}, \Phi^n),
		\end{aligned}
	\end{equation}
	where the right hand side is defined by \eqref{dF}.
\end{scheme}
\begin{theorem}\label{thm:dg}
	The discrete-gradient scheme in \eqref{eq:dg} preserves energy dissipation rate as follows
	\begin{equation*}
		F(\Phi^{n+1}) - F(\Phi^n) = \tau ( \overline{\nabla} F, \mathbf{M}(\Phi^{n+\frac{1}{2}}) \overline{\nabla} F)_H.
	\end{equation*}
\end{theorem}
The result of Theorem \ref{thm:dg} can be readily obtained by taking inner products on both sides of \eqref{eq:dg} with $\overline{\nabla}F(\Phi^{n+1}, \Phi^n)$.

\begin{scheme}[Temporal Petrov-Galerkin scheme]
	We introduce a temporal Petrov-Galerkin scheme, which can be of arbitrarily high-order. For a better presentation, we recast the gradient flow system into the following form
	\begin{equation}\label{eq:gradient-mix}
		\left\lbrace
		\begin{aligned}
			\dot{\Phi} & = \mathbf{L}(\Phi)\mu, \\
			\mu        & = \nabla F(\Phi).
		\end{aligned}
		\right.
	\end{equation}
	The temporal Petrov-Galerkin method is then presented as follows: finding $\Phi_h \in \mathbb{P}_{s}(I_{n+\frac{1}{2}}) \otimes H$, $\mu_h \in \mathbb{P}_{s-1}(I_{n+\frac{1}{2}}) \otimes H$, such that for any $U_h, \nu_h \in \mathbb{P}_{s-1}(I_{n+\frac{1}{2}}) \otimes H$
	\begin{equation}\label{eq:pg}
		\left\lbrace
		\begin{aligned}
			 & \medint_{I_{n+\frac{1}{2}}} (\dot{\Phi}_h, U_h)_H dt = \medint_{I_{n+\frac{1}{2}}} ( \mathbf{L}(\mu_h) \Psi_h, U_h)_H dt, \\
			 & \medint_{I_{n+\frac{1}{2}}} (\mu_h, \nu_h)_H dt = \medint_{I_{n+\frac{1}{2}}} (\nabla F(\Phi_h), \nu_h)_H dt.
		\end{aligned}
		\right.
	\end{equation}
	Here, $\mathbb{P}_s(I_{n+\frac{1}{2}})$ represents the space of  polynomials defined in $I_{n+\frac{1}{2}}$ with degrees no more than $s$.
\end{scheme}
\begin{theorem}
	We denote $\Phi^n = \Phi_h(t_n)$. The Petrov-Galerkin scheme preserves the energy dissipation rate as follows
	\begin{equation*}
		F(\Phi^{n+1}) - F(\Phi^n) = \medint_{I_{n+\frac{1}{2}}} (\mathbf{M}(\Phi_h)\mu_h, \mu_h)_H dt.
	\end{equation*}
\end{theorem}
\begin{proof}
	Taking the test function $U_h = \mu_h$ and $\nu_h = \dot{\Phi}_h$, we obtain
	\begin{equation*}
		\begin{aligned}
			\medint_{I_{n+\frac{1}{2}}} (\dot{\Phi}_h, \mu_h)_H dt & = \medint_{I_{n+\frac{1}{2}}} ((\mathbf{M}(\Phi_h) + \mathbf{S}(\Phi_h))\mu_h, \mu_h)_H dt \\
			                                                       & = \medint_{I_{n+\frac{1}{2}}} (\mathbf{M}(\Phi_h)\mu_h, \mu_h)_H dt.
		\end{aligned}
	\end{equation*}
	Notice that
	\begin{equation*}
		\begin{aligned}
			\medint_{I_{n+\frac{1}{2}}} (\mu_h, \dot{\Phi}_h)_H dt = \medint_{I_{n+\frac{1}{2}}} (\nabla F(\Phi_h), \dot{\Phi}_h)_H dt = F(\Phi^{n+1})  - F(\Phi^n).
		\end{aligned}
	\end{equation*}
	Combining the above two equations, we obtain the desired result.
\end{proof}

Note that term $\mathbf{J}(\Phi)$ in the above two schemes, \eqref{eq:dg} and \eqref{eq:pg}, is implement fully implicitly. In practices, this may be inefficient, or not solver friendly if there exist non-local or nonlinear  operators in $\mathbf{J}(\Phi)$, such as the discrete convection operator in \eqref{eq:NS-operator} and \eqref{eq:CH-NS-operator}. It is preferable to propose numerical schemes that implement such term explicitly. Since operator $\mathbf{S}(\Phi)$ does not participate in energy dissipation, we   modify the discrete-gradient scheme as follows.
\begin{scheme}[Modified discrete gradient]
	\begin{equation}\label{eq:mdg}
		\delta_\tau \Phi^{n+1}  = \mathbf{L}(\hat{\Phi}^{n+\frac{1}{2}}) \overline{\nabla} F(\Phi^{n+1}, \Phi^n).
	\end{equation}
\end{scheme}
It is readily to confirm that the MDG scheme, \eqref{eq:mdg}, preserves the dissipation law as that in Theorem \ref{thm:dg} with dissipation rate $(\overline{\nabla}F, \mathbf{M}(\hat{\Phi}^{n+\frac{1}{2}}) \overline{\nabla} F)_H$. Moreover, $\hat{\Phi}^{n+\frac{1}{2}}$ is an appropriate second-order approximation to $\Phi(t_{n+\frac{1}{2}})$, which may have multiple choices. For example, extrapolation $\hat{\Phi}^{n+\frac{1}{2}} = \tfrac{3}{2} \Phi^n - \tfrac{1}{2} \Phi^{n-1}$ is a choice. To guarantee stability, it is recommended to generate $\hat{\Phi}^{n+\frac{1}{2}}$ using a prediction step, for example, we can use the following scheme to get a prediction
\begin{equation}
	\tfrac{\Phi^{n+1/2} - \Phi^n}{\tau / 2} = (\mathbf{M}(\Phi^{n}) + \mathbf{S}(\Phi^n)) (\mathcal{L} \Phi^{n+1/2} + \nabla f(\Phi^n)) .
\end{equation}
Since the antisymmetric component does not show up in the energy dissipation rate, the explicit treatment of its discretization would not affect the overall energy dissipation rate.

\subsection{Stabilized schemes for dissipative systems}
We next introduce several other energy-stable schemes for the generalized gradient flow system systematically motivated by the concept of convex splitting. Although these schemes do not exactly preserve the energy dissipation rate, they are energy-stable nevertheless. In these schemes, a modified energy is shown to decay faster than a modified dissipation rate in which the perturbation to the original energy and dissipation rate are of higher order in terms of the temporal step size.  The convex-splitting scheme is designed based on the assumption that the energy can be reformulated into a difference of two convex functions:  $F(\Phi) = F_c(\Phi) - F_e(\Phi)$, where both $F_c(\Phi)$ and $F_e(\Phi)$ are convex. Under this assumption, the following first order numerical scheme can be shown to be energy stable:
\begin{scheme}[Convex-splitting scheme]
	\begin{equation}\label{eq:cs1}
		\delta_\tau \Phi^{n+1} = (\mathbf{M}(\Phi^n) + \mathbf{S}(\Phi^n))(\nabla F_c(\Phi^{n+1}) - \nabla F_e(\Phi^n)).
	\end{equation}
\end{scheme}

To show the energy stability of the convex-splitting scheme, we invoke the following Lemma \ref{lem:lem-cs}.
\begin{lemma}[\cite{CS2}]\label{lem:lem-cs}
	The following inequality holds true for convex functions $F_c(\phi)$, $F_e(\phi)$ and $F=F_c-F_e$:
	\begin{equation*}
		F(\Phi^{n+1}) - F(\Phi^n) \leq (\nabla F_c(\Phi^{n+1}) - \nabla F_e(\Phi^n), \delta_{\tau} \Phi^{n+1})_H.
	\end{equation*}
\end{lemma}
In fact, for any convex functions $F_c$ and $F_e$, we have
\begin{equation*}
	\bea{l}
	F_c(\Phi^{n+1}) - F_c(\Phi^n) \leq (\nabla F_c(\Phi^{n+1}), \delta_{\tau} \Phi^{n+1})_H,\\
	\\
	-F_e(\Phi^{n+1}) + F_e(\Phi^n) \leq -( \nabla F_e(\Phi^n), \delta_{\tau} \Phi^{n+1})_H.
	\eea
\end{equation*}
Adding the two inequalities, we arrive at
\begin{equation*}
	F_c(\Phi^{n+1})-F_e(\Phi^{n+1})  - (F_c(\Phi^n)-F_e(\Phi^n) )\leq (\nabla F_c(\Phi^{n+1}) - \nabla F_e(\Phi^n), \delta_{\tau} \Phi^{n+1})_H.
\end{equation*}

\begin{theorem}\label{thm:stability-cs1}
	The first-order convex-splitting scheme is energy-stable in that
	\begin{equation}
		\begin{gathered}
			d_t F(\Phi^{n+1}) \leq \tau \left(\nabla F_c(\Phi^{n+1}) - \nabla F_e(\Phi^n), \mathbf{M}(\Phi^n) (\nabla F_c(\Phi^{n+1}) - \nabla F_e(\Phi^n))\right)_H.
		\end{gathered}
	\end{equation}
\end{theorem}
\begin{proof}
	By taking the inner product of \eqref{eq:cs1} with $\tau(\nabla F_c(\Phi^{n+1}) - \nabla F_e(\Phi^n))$ and invoking Lemma \ref{lem:lem-cs} as well as the skew-symmetry property of $\mathbf{S}(\Phi^n)$, we obtain
	\begin{equation*}
		\begin{aligned}
			 & d_t F(\Phi^{n+1}) \leq (\nabla F_c(\Phi^{n+1}) - \nabla F_e(\Phi^n), d_t \Phi^{n+1})_H,                                                    \\
			 & \quad = \tau \left(\nabla F_c(\Phi^{n+1}) - \nabla F_e(\Phi^n), \mathbf{M}(\Phi^n) (\nabla F_c(\Phi^{n+1}) - \nabla F_e(\Phi^n))\right)_H.
		\end{aligned}
	\end{equation*}
	The proof is thus completed.
\end{proof}
The theorem claims that the discrete energy dissipation rate is smaller than an approximation to the exact energy dissipation rate given by
\begin{equation*}
	\left(\nabla F_c(\Phi^{n+1}) - \nabla F_e(\Phi^n), \mathbf{M}(\Phi^n) (\nabla F_c(\Phi^{n+1}) - \nabla F_e(\Phi^n))\right)_H,
\end{equation*}
indicating that the convex splitting strategy may introduce additional dissipation.

We note that in real-world applications, the range of $\Phi$ is normally confined in a compact domain. As a result, we can always assume the free energy, $F(\Phi)$, is bounded by a quadratic function when $\|\Phi\|>>1$ and $\nabla f(\Phi)$ is Lipschitz continuous with Lipschitz constant $L$. Under this assumption, we can always construct a convex-splitting of $F(\Phi)$ by setting
\begin{equation*}
	F_c(\Phi) = \tfrac{1}{2}(\Phi, \mathcal{L}\Phi)_H+\tfrac{1}{2}(\bM_q\Phi, \Phi)_H, \  F_e(\Phi) = \tfrac{1}{2}(\bM_q\Phi, \Phi)_H -  f(\Phi),
\end{equation*}
where $\bM_q>0$ is a symmetric and positive definite linear operator with a large enough $\|\bM_q\|$.

For example, we set $\bM_q= A\mathbf{I}$ and $A$ is large enough. The following is a legitimate convex splitting of $F$.
\begin{equation*}
	F_c(\Phi) = \tfrac{1 }{2}(\Phi, \mathcal{L}\Phi)_H+\tfrac{A}{2}(\Phi, \Phi)_H, \  F_e(\Phi) = \tfrac{A}{2}(\Phi, \Phi)_H -  f(\Phi)
\end{equation*}
In this case, \eqref{eq:cs1} yields the following first order linearly implicit stabilization scheme
\begin{equation}\label{eq:stabilization-BDF1}
	\delta_{\tau} \Phi^{n+1} = \mathbf{L} (\Phi^n) (\mathcal{L} \Phi^{n+1} + A d_t \Phi^{n+1} + \nabla f(\Phi^n)), \ A\geq L.
\end{equation}
The energy stability of \eqref{eq:stabilization-BDF1} follows Theorem \ref{thm:stability-cs1}.

\begin{remark}\label{rmk:efficient-solver}
	It is readily to show that the stabilization scheme leads to a linear system of the form,
	\begin{equation}\label{eq:general-linear-system}
		\mathbf{A} \mathbf{x} + \mathbf{S} \mathbf{x} = \mathbf{f},
	\end{equation}
	where $\mathbf{A}$ is symmetric positive-definite and $\mathbf{S} = \mathbf{a}\mathbf{b}^\top - \mathbf{b} \mathbf{a}^\top$ is a skew-symmetric, rank-2 operator with $\mathbf{a}, \mathbf{b}, \bx \in H$. Introducing $\xi = \mathbf{b}^\top \mathbf{x}$ and $\eta = \mathbf{a}^\top \mathbf{x}$, The linear system \eqref{eq:general-linear-system} can be rewritten equivalently as
	\begin{equation} \label{eq:general-linear-system-equiv}
		\left\lbrace
		\begin{aligned}
			 & \mathbf{A} \mathbf{x} + \xi \mathbf{a} - \eta \mathbf{b} = \mathbf{f}, \\
			 & \xi = \mathbf{b}^\top \mathbf{x},                                      \\
			 & \eta = \mathbf{a}^\top \mathbf{x}.
		\end{aligned}
		\right.
	\end{equation}
	According to the first equation above, we have
	\begin{equation}\label{eq:general-linear-system-x}
		\mathbf{x} = -\xi \mathbf{A}^{-1} \mathbf{a} + \eta \mathbf{A}^{-1} \mathbf{b} + \mathbf{A}^{-1} \mathbf{f}.
	\end{equation}
	Taking the inner products of \eqref{eq:general-linear-system-x} with $\mathbf{b}$ and $\mathbf{a}$, respectively, and substituting the second and third equations from \eqref{eq:general-linear-system-equiv} into the resulting expression yields
	\begin{equation*}
		\left\lbrace
		\begin{aligned}
			 & (1 + \mathbf{b}^\top \mathbf{A}^{-1} \mathbf{a})	\xi - \mathbf{b}^\top \mathbf{A}^{-1} \mathbf{b} \eta = \mathbf{b}^\top \mathbf{A}^{-1} \mathbf{f},   \\
			 & (\mathbf{a}^\top \mathbf{A}^{-1} \mathbf{a}) \xi + (1 - \mathbf{a}^\top \mathbf{A}^{-1} \mathbf{b}) \eta = \mathbf{a}^\top \mathbf{A}^{-1} \mathbf{f}.
		\end{aligned}
		\right.
	\end{equation*}
	The two scalar variables $\xi, \eta$ are then obtained by solving the above $2 \times 2$ linear system, after which $\mathbf{x}$ is finally recovered from \eqref{eq:general-linear-system-x}. Throughout the entire process, we only need to solve three linear equations
	\begin{equation*}
		\bA\cdot \bx_1=\ba, \ \bA\cdot \bx_2=\bb, \ \bA\cdot \bx_3=\mathbf{f},
	\end{equation*}
	with respect to operator $\mathbf{A}$, which can be done efficiently in many cases.
\end{remark}



Inspired by the convex‐splitting idea of first‐order schemes, we introduce two genuinely second‐order, energy‐stable SGE-stabilized-BDF2 (SGE-SBDF2) and SGE-stabilized-Crank-Nicolson (SGE-SCN) schemes by choosing operator $\mathbf{M}_q = -\tau A \mathbf{M}(\Phi)$ and employing a second‐order time discretization for general gradient flows.

\begin{scheme}[SGE-SBDF2 scheme]
	\begin{equation}\label{eq:stabilization-BDF-2}
		\left\lbrace
		\begin{aligned}
			\delta_{2\tau} \Phi^{n+1} & = \mathbf{L} (\hat{\Phi}^{n+1}) \mu^{n+1}                                                                                       \\
			\mu^{n+1}                 & = \mathcal{L} \Phi^{n+1} - \tau A \mathbf{M}(\hat{\Phi}^{n+1}) d_{2t} \Phi^{n+1} + 2 \nabla f(\Phi^n) -  \nabla f(\Phi^{n-1}) .
		\end{aligned}
		\right.
	\end{equation}
\end{scheme}

\begin{scheme}[SGE-SCN scheme]
	\begin{equation}\label{eq:stabilization-CN-2}
		\left\lbrace
		\begin{aligned}
			\delta_{\tau} \Phi^{n+1} & = \mathbf{L} (\hat{\Phi}^{n+\frac{1}{2}}) \mu^{n+\frac{1}{2}}                                                                                                          \\
			\mu^{n+\frac{1}{2}}      & = \mathcal{L} \Phi^{n+\frac{1}{2}} - \tau A \mathbf{M}(\hat{\Phi}^{n+\frac{1}{2}}) d_t \Phi^{n+1}, + \tfrac{3}{2} \nabla f(\Phi^n) - \tfrac{1}{2} \nabla f(\Phi^{n-1})
		\end{aligned}
		\right.
	\end{equation}
\end{scheme}
where $\bL=\bM+\bS$, $\bM=\bM^T<0$, and $\bS=-\bS^T$.

\begin{lemma}\label{lem:difference-identity}
	In $H$, we have the following identities:
	\begin{equation*}
		\begin{aligned}
			2(d_{2t} \Phi^{n+1}, \Phi^{n+1})_H     & = d_t (\|\Phi^{n+1}\|_H^2 + \|2 \Phi^{n+1} - \Phi^n\|_H^2 ) + \|d_{tt}\Phi^n\|_H^2,               \\
			2(d_{2t} \Phi^{n+1}, d_t \Phi^{n+1})_H & = \|d_t \Phi^{n+1}\|_H^2 - \|d_t \Phi^n\|_H^2 + \|d_{tt} \Phi^n\|_H^2 + 4 \|d_t \Phi^{n+1}\|_H^2, \\
			2 (d_{tt} \Phi^n, d_{t} \Phi^{n+1})_H  & = \|d_t \Phi^{n+1}\|_H^2 - \|d_t \Phi^n\|_H^2 + \|d_{tt} \Phi^n\|_H^2,                            \\
			\|d_{2t} \Phi^{n+1}\|_H^2              & = 6 \|d_t \Phi^{n+1}\|_H^2 - 2\|d_t \Phi^n\|_H^2 + 3 \|d_{tt} \Phi^n\|_H^2.
		\end{aligned}
	\end{equation*}
\end{lemma}
\begin{lemma}\label{lem:difference-inequality}
	The following inequalities hold true in $H$:
	\begin{equation*}
		\begin{aligned}
			d_{2t} f(\Phi^{n+1}) & \leq (d_{2t} \Phi^{n+1}, 2\nabla f(\Phi^n) - \nabla f(\Phi^{n-1}))_H + 3L \|d_t \Phi^{n+1}\|_H^2 + 3L\|d_t \Phi^n\|_H^2,                                            \\
			d_{t} f(\Phi^{n+1})  & \leq (d_t \Phi^{n+1}, \tfrac{3}{2} \nabla f(\Phi^n) - \tfrac{1}{2}\nabla f(\Phi^{n-1}))_H + \tfrac{3L}{4} \|d_t \Phi^{n+1}\|_H^2 + \tfrac{L}{4} \|d_t \Phi^n\|_H^2.
		\end{aligned}
	\end{equation*}
\end{lemma}
\begin{proof}
	We only prove the first inequality as an example. The other can be proved analogously. Invoking the Taylor's expansion, we have
	\begin{equation*}
		\begin{aligned}
			(\nabla f(\Phi^n), d_t \Phi^{n+1})_H     & = d_t f(\Phi^{n+1})+ \tfrac{1}{2} (d_t \Phi^{n+1} , \nabla^2 f(\Phi^n + \xi d_t \Phi^{n+1})  \cdot d_t \Phi^{n+1} )_H, \\
			(\nabla f(\Phi^{n+1}), d_t \Phi^{n+1})_H & = d_t f(\Phi^{n+1}) - \tfrac{1}{2} (d_t \Phi^{n+1},  \nabla^2 f(\Phi^n + \xi d_t \Phi^{n+1}) \cdot d_t \Phi^{n+1})_H.
		\end{aligned}
	\end{equation*}
	It follows that
	\begin{equation*}
		\begin{aligned}
			 & \left(d_{2 t}\Phi^{n+1}, 2\nabla f(\Phi^n) -\nabla f(\Phi^{n-1})\right)_H                                     \\
			 & \ = 3(d_t \Phi^{n+1}, \nabla f(\Phi^n))_H + 3(d_t \Phi^{n+1}, d_t \nabla f(\Phi^n))_H                         \\
			 & \ \ - (d_t \Phi^n, \nabla f(\Phi^n))_H - (d_t \Phi^n, d_t \nabla f(\Phi^n))_H                                 \\
			 & \ \geq 3 d_t f(\Phi^{n+1}) - \tfrac{3 L}{2} \|d_t \Phi^{n+1}\|_H^2 - 3L \|d_t \Phi^{n+1}\|_H \|d_t \Phi^n\|_H \\
			 & \ \ - d_t f(\Phi^n) - \tfrac{L}{2} \|d_t \Phi^n\|_H^2 - L \|d_t \Phi^n\|_H^2                                  \\
			 & \ \geq d_{2t} f(\Phi^{n+1}) - 3L \|d_t \Phi^{n+1}\|_H^2 - 3L \|d_t \Phi^n\|_H^2.
		\end{aligned}
	\end{equation*}
	The proof is thus completed.
\end{proof}

\begin{lemma}\label{lem:operator}
	We assume that $\mathbf{S}(\Phi)$ is bounded on $H$, and that $-\mathbf{M}(\Phi)$ is symmetric, positive definte, and coercive on $H$. I.e.,  there exists positive constants $\alpha_\mathbf{M}$ and $\beta_\mathbf{S}$ so that for any $\Psi_1, \Psi_2, \Psi \in H$,
	\begin{equation}\label{eq:assum-operators}
		(\mathbf{S}(\Phi)\Psi_1, \Psi_2)_H \leq \beta_\mathbf{S} \|\Psi_1\|_H\|\Psi_2\|_H, \quad (-\mathbf{M}(\Phi)\Psi, \Psi)_H \geq \alpha_\mathbf{M} \|\Psi\|_H^2.
	\end{equation}
	It then follows that
	\begin{equation*}
		( \mathbf{L}(\Phi) \Psi_1, \Psi_2 )_H \leq (1 + \tfrac{\beta_\mathbf{S}}{\alpha_\mathbf{M}}) (\Psi_1, -\mathbf{M}(\Phi) \Psi_1)_H^{1/2}(\Psi_2, -\mathbf{M}(\Phi) \Psi_2)_H^{1/2}.
	\end{equation*}
\end{lemma}
\begin{proof}
	By invoking the Cauchy–Schwarz inequality together with hypotheses \eqref{eq:assum-operators}, we obtain
	\begin{equation*}
		\begin{aligned}
			 & ( \mathbf{L}(\Phi) \Psi_1, \Psi_2 )_H  = (\mathbf{M}(\Phi) \Psi_1, \Psi_2)_H + (\mathbf{S}(\Phi) \Psi_1, \Psi_2)_H                             \\
			 & \  \leq (\Psi_1, -\mathbf{M}(\Phi) \Psi_1)_H^{1/2}(\Psi_2, -\mathbf{M}(\Phi) \Psi_2)_H^{1/2} + \beta_\mathbf{S} \|\Psi_1\|_H \|\Psi_2\|_H      \\
			 & \ \leq (1 + \tfrac{\beta_\mathbf{S}}{\alpha_\mathbf{M}}) (\Psi_1, -\mathbf{M}(\Phi) \Psi_1)_H^{1/2}(\Psi_2, -\mathbf{M}(\Phi) \Psi_2)_H^{1/2}.
		\end{aligned}
	\end{equation*}
\end{proof}

Using the lemmas, we establish  following stability results for the SGE-SCN and SGE-SBDF2 schemes, respectively.
\begin{theorem}
	Provided that the stabilization parameter satisfies $A \geq \tfrac{L^2}{4}(1 + \tfrac{\beta_\mathbf{S}}{\alpha_\mathbf{M}})^2$, the SGE-SCN scheme is unconditionally energy stable in the sense that
	\begin{equation*}
		d_t \tilde{E}_{CN}(\Phi^{n+1}, \Phi^n) \leq	- \left( 2(1 + \tfrac{\beta_\mathbf{S}}{\alpha_\mathbf{M}})^{-1}\sqrt{A} - L\right)\|d_t \Phi^{n+1}\|_H^2.
	\end{equation*}
	Here,
	\begin{equation*}
		\tilde{E}_{CN}(\Phi^{n+1}, \Phi^n) = \tfrac{1}{2}\|\mathcal{L}^{\frac{1}{2}} \Phi^n\|_H^2 + f(\Phi^n) + \tfrac{L}{4} \|d_t \Phi^{n+1}\|_H^2
	\end{equation*}
\end{theorem}
\begin{proof}
	Taking the inner product of the first equation of \eqref{eq:stabilization-CN-2} with $\tau\mu^{n+\frac{1}{2}}$ and the second with $d_t \Phi^{n+1}$, we obtain
	\begin{equation}\label{eq:proof-stabcn2-1}
		( d_t \Phi^{n+1}, \mu^{n+\frac{1}{2}}  )_H = \tau (\mu^{n+\frac{1}{2}}, \mathbf{M}(\hat{\Phi}^{n+\frac{1}{2}}) \mu^{n+\frac{1}{2}}),
	\end{equation}
	\begin{equation}\label{eq:proof-stabcn2-2}
		\begin{aligned}
			( d_t \Phi^{n+1}, \mu^{n+\frac{1}{2}}  )_H & = \tfrac{1}{2} d_t \|\mathcal{L}^{\frac{1}{2}} \Phi^{n+1}\|_H^2 - \tau A (d_t \Phi^{n+1}, \mathbf{M}(\hat{\Phi}^{n+\frac{1}{2}}) d_t \Phi^{n+1}) \\
			                                           & \ + (\tfrac{3}{2}\nabla f(\Phi^n) - \tfrac{1}{2} \nabla f(\Phi^{n-1}), d_t \Phi^{n+1}).
		\end{aligned}
	\end{equation}
	Combining \eqref{eq:proof-stabcn2-1}, \eqref{eq:proof-stabcn2-2} and Lemmas \ref{lem:difference-identity}, \ref{lem:difference-inequality}, we obtain
	\begin{equation} \label{eq:proof-stabcn2-3}
		\begin{aligned}
			\delta_ t E^{n+1} & \leq -\tfrac{L}{4} \|d_t \Phi^{n+1}\|_H^2 + \tfrac{L}{4} \|d_t \Phi^n\|_H^2  + L\|d_t \Phi^{n+1}\|_H^2                                                                       \\
			                  & \ + \tau A (d_t \Phi^{n+1}, \mathbf{M}(\hat{\Phi}^{n+\frac{1}{2}}) d_t \Phi^{n+1}) + \tau (\mu^{n+\frac{1}{2}}, \mathbf{M}(\hat{\Phi}^{n+\frac{1}{2}}) \mu^{n+\frac{1}{2}}).
		\end{aligned}
	\end{equation}
	Taking the inner product of the first equation of \eqref{eq:stabilization-CN-2} with $2(1 + \tfrac{\beta_\mathbf{S}}{\alpha_\mathbf{M}})^{-1}\sqrt{A} \tau d_t \Phi^{n+1}$ and invoking Lemma \ref{lem:operator}, we arrive at
	\begin{equation*}
		\begin{aligned}
			 & 2(1 + \tfrac{\beta_\mathbf{S}}{\alpha_\mathbf{M}})^{-1}\sqrt{A} \|d_t \Phi^{n+1}\|_H^2                                                                                                             \\
			 & \ \leq 2\tau \sqrt{A}   (\mu^{n+\frac{1}{2}}, -\mathbf{M}(\hat{\Phi}^{n+\frac{1}{2}}) \mu^{n+\frac{1}{2}})_H^{1/2}(d_t \Phi^{n+1}, -\mathbf{M}(\hat{\Phi}^{n+\frac{1}{2}}) d_t \Phi^{n+1})_H^{1/2} \\
			 & \ \leq  - \tau A (d_t \Phi^{n+1}, \mathbf{M}(\hat{\Phi}^{n+\frac{1}{2}}) d_t \Phi^{n+1})_H - \tau (\mu^{n+\frac{1}{2}}, \mathbf{M}(\hat{\Phi}^{n+\frac{1}{2}}) \mu^{n+\frac{1}{2}})_H.
		\end{aligned}
	\end{equation*}
	Summing this result with \eqref{eq:proof-stabcn2-3} produces
	\begin{equation}
		d_t \tilde{E}_{CN}(\Phi^{n+1}, \Phi^n) \leq - \left( 2(1 + \tfrac{\beta_\mathbf{S}}{\alpha_\mathbf{M}})^{-1}\sqrt{A} - L\right)\|d_t \Phi^{n+1}\|_H^2.
	\end{equation}
	The energy stability of \eqref{eq:stabilization-CN-2} follows immediately when $A \geq \tfrac{L^2}{4}(1 + \tfrac{\beta_\mathbf{S}}{\alpha_\mathbf{M}})^2$.
\end{proof}
\begin{theorem}\label{thm:stability-bdf2}
	Provided that the stabilization parameter satisfies $A \geq \tfrac{9L^2}{8}(1 + \tfrac{\beta_\mathbf{S}}{\alpha_\mathbf{M}})^2$, The SGE-SBDF2 scheme is unconditionally energy stable in the sense that
	\begin{equation*}
		\begin{aligned}
			d_t \tilde{E}_{BDF2}(\Phi^{n+1}, \Phi^n) & \leq	-(2\sqrt{2A} (1 + \tfrac{\beta_\mathbf{S}}{\alpha_\mathbf{M}})^{-1} - 3L)\|d_t \Phi^{n+1}\|^2                                                               \\
			                                         & \ - \tfrac{3}{2}\sqrt{2A} (1+ \tfrac{\beta_\mathbf{S}}{\alpha_\mathbf{M}})^{-1}\|d_{tt} \Phi^n\|^2 - \tfrac{1}{4} \|\mathcal{L}^{\frac{1}{2}} d_{tt} \Phi^n\|^2.
		\end{aligned}
	\end{equation*}
	Here,
	\begin{equation*}
		\begin{aligned}
			 & \tilde{E}_{BDF2}(\Phi^{n+1}, \Phi^n) = \tfrac{1}{4} \|\mathcal{L}^{\frac{1}{2}}\Phi^{n+1}\|^2 + \tfrac{1}{4} \|\mathcal{L}^{\frac{1}{2}}(2\Phi^{n+1} - \Phi^n)\|^2 \\
			 & \ + \tfrac{3}{2} f(\Phi^{n+1}) - \tfrac{1}{2} f(\Phi^n) + (\tfrac{3L}{2} + \sqrt{2A} (1 + \tfrac{\beta_\mathbf{S}}{\alpha_\mathbf{M}})^{-1}) \|d_t \Phi^{n+1}\|^2.
		\end{aligned}
	\end{equation*}
\end{theorem}
\begin{proof}
	Taking the inner product of the first and second equations in \eqref{eq:stabilization-BDF-2} with $\mu^{n+1}$ and $\delta_{2\tau} \Phi^{n+1}$, respectively, subsequently comparing the resulting identities, one obtains
	\begin{equation}\label{eq:proof_stabbdf22-1}
		\begin{aligned}
			 & (\mathcal{L} \Phi^{n+1}, d_{2t} \Phi^{n+1})_H + (2\nabla f(\Phi^n) - \nabla f(\Phi^{n-1}) , d_{2t} \Phi^{n+1})_H                               \\
			 & \ = 2\tau (\mu^{n+1}, \mathbf{M}(\hat{\Phi}^{n+1}) \mu^{n+1})_H + \tau A (d_{2t} \Phi^{n+1}, \mathbf{M}(\hat{\Phi}^{n+1}) d_{2t}\Phi^{n+1})_H. \\
		\end{aligned}
	\end{equation}
	Taking the inner product of the first equation in \eqref{eq:stabilization-BDF-2} with $2\tau \sqrt{2A} (1 + \tfrac{\beta_S}{\alpha_M})^{-1}$ and exploiting Lemma \ref{lem:operator}, results in
	\begin{equation}\label{eq:proof_stabbdf22-2}
		\begin{aligned}
			 & \sqrt{2A} (1 + \tfrac{\beta_S}{\alpha_M})^{-1} \|d_{2t} \Phi^{n+1}\|^2                                                                                                    \\
			 & \ \leq 2\tau \sqrt{2A} (\mu^{n+1}, -\mathbf{M}(\hat{\Phi}^{n+1})\mu^{n+1})^{\frac{1}{2}} (d_{2t} \Phi^{n+1}, -\mathbf{M}(\hat{\Phi}^{n+1})d_{2t}\Phi^{n+1})^{\frac{1}{2}} \\
			 & \ \leq -\tau A (d_{2t} \Phi^{n+1}, \mathbf{M}(\hat{\Phi}^{n+1})d_{2t}\Phi^{n+1})_H - 2\tau (\mu^{n+1}, \mathbf{M}(\hat{\Phi}^{n+1})\mu^{n+1}).
		\end{aligned}
	\end{equation}
	Notice that we have identity
	\begin{equation}
		\begin{aligned}
			 & (\mathcal{L}\Phi^{n+1}, d_{2t}\Phi^{n+1}) = \tfrac{1}{2} \|\mathcal{L}^{\frac{1}{2}}\Phi^{n+1}\|^2 + \tfrac{1}{2} \|\mathcal{L}^{\frac{1}{2}}(2\Phi^{n+1} - \Phi^n)\|^2                           \\
			 & \quad - \tfrac{1}{2} \|\mathcal{L}^{\frac{1}{2}}\Phi^{n}\|^2 - \tfrac{1}{2} \|\mathcal{L}^{\frac{1}{2}}(2\Phi^{n} - \Phi^{n-1})\|^2 + \tfrac{1}{2} \|\mathcal{L}^{\frac{1}{2}} d_{tt} \Phi^n\|^2.
		\end{aligned}
	\end{equation}
	Combing Lemmas \ref{lem:difference-identity}, \ref{lem:difference-inequality} with the identity above, we obtain
	\begin{equation}
		\begin{aligned}
			 & \tfrac{1}{2} \|\mathcal{L}^{\frac{1}{2}}\Phi^{n+1}\|^2 + \tfrac{1}{2} \|\mathcal{L}^{\frac{1}{2}}(2\Phi^{n+1} - \Phi^n)\|^2  + 3 f(\Phi^{n+1}) - f(\Phi^n)                                                                                               \\
			 & \ - \tfrac{1}{2} \|\mathcal{L}^{\frac{1}{2}}\Phi^{n}\|^2 - \tfrac{1}{2} \|\mathcal{L}^{\frac{1}{2}}(2\Phi^{n} - \Phi^{n-1})\|^2  - 3 f(\Phi^{n}) + f(\Phi^{n-1})                                                                                         \\
			 & \ + (3L +2 \sqrt{2A} (1 + \tfrac{\beta_\mathbf{S}}{\alpha_\mathbf{M}})^{-1}) \|d_t \Phi^{n+1}\|^2 - (3L +2 \sqrt{2A} (1 + \tfrac{\beta_\mathbf{S}}{\alpha_\mathbf{M}})^{-1}) \|d_t \Phi^n\|^2                                                            \\
			 & \ \leq -(4\sqrt{2A} (1 + \tfrac{\beta_\mathbf{S}}{\alpha_\mathbf{M}})^{-1} - 6L)\|d_t \Phi^{n+1}\|^2 - 3\sqrt{2A} (1+ \tfrac{\beta_\mathbf{S}}{\alpha_\mathbf{M}})^{-1}\|d_{tt} \Phi^n\|^2 - \tfrac{1}{2} \|\mathcal{L}^{\frac{1}{2}} d_{tt} \Phi^n\|^2.
		\end{aligned}
	\end{equation}
	The energy stability follows after imposing $A\geq \tfrac{9}{8} L^2(1 + \tfrac{\beta_\mathbf{S}}{\alpha_\mathbf{M}})^2$.
\end{proof}

\begin{remark}
	We note that this is the first time one is able to construct and prove  second-order stabilized schemes for generalized gradient flows. A key element in establishing energy stability for these schemes lies in the boundedness of operator $\mathbf{S}$, characterized by constant $\beta_{\mathbf{S}}$ in \eqref{eq:assum-operators}. Within the SGE framework, this assumption is reasonable, although a rigorous justification would require a detailed convergence analysis of the numerical solution. In contrast, skew-symmetrization approaches commonly found in the literature-such as those applied to the Burgers equation in Example \ref{ex:burgers}-render this assumption untenable, as operator $\mathbf{S}$ in those settings, distinct from ours, typically involves unbounded linear operators.
\end{remark}

\section{Numerical Results}

We implement some numerical schemes on the three equation systems alluded to earlier to show their accuracy, robustness, and dissipative properties. We begin with the viscous Burgers equation.


\subsection{Burgers Equation}
Taking the Burgers equation as an example, we detail the construction of the scheme with a fully explicit discretization for the convective term. Since the energy functional is quadratic, its discrete gradient simplifies to $\overline{\nabla} F(u^n, u^{n+1}) = u^{n+\frac{1}{2}}$. The SGE-CN scheme for the Burgers equations is given as follows:
\begin{equation*}
	\delta_\tau u^{n+1} + \tfrac{\hat{u}^{n+\frac{1}{2}} \wedge (\hat{u}^{n+\frac{1}{2}}\hat{u}_x^{n+\frac{1}{2}})}{\|\hat{u}^{n+\frac{1}{2}}\|^2} u^{n+\frac{1}{2}} = \nu \partial_{xx} u^{n+\frac{1}{2}}.
\end{equation*}
We also present the following second-order SGE-BDF2 scheme for the equation as well,
\begin{equation*}
	\delta_{2\tau} u^{n+1} + \tfrac{\hat{u}^{n+1} \wedge (\hat{u}^{n+1}\hat{u}_x^{n+1})}{\|\hat{u}^{n+1}\|^2} u^{n+1} = \nu \partial_{xx} u^{n+1}.
\end{equation*}
As demonstrated in Remark \ref{rmk:efficient-solver}, both schemes are linearly implicit and the solution process actually requires solving only linear equations with constant-coefficients.

We consider the Burgers equation in $\Omega = (-1, 1)$ subject to periodic boundary conditions with initial condition
\begin{equation}
	u(x, 0) = - \sin(\pi x).
\end{equation}
We first test the temporal convergence of the presented schemes using a refinement test. Because the exact solution for \eqref{eq:burgers} is not available, we generate a reference solution with a sixth-order time Petrov-Galerkin method, using a spatial resolution of $N = 2560$ and a time step of $\tau=$1e-5 until $t = 1$. Then, the Burgers equation is solved with various schemes for time steps $\tau = \tau_0/2^k$ (with $\tau_0 =$8e-3 and $k = 0,\cdots,5$) and a spatial resolution of $N=1280$. The solution at the terminal time $t=1$ is recorded and the discrete $L^2$ error is computed. As depicted in the first subplot of Figure \ref{fig:burgers-fig}, all schemes achieve second-order convergence as expected.

\begin{figure}[htbp]
	\includegraphics[width=0.3\textwidth, height=101pt]{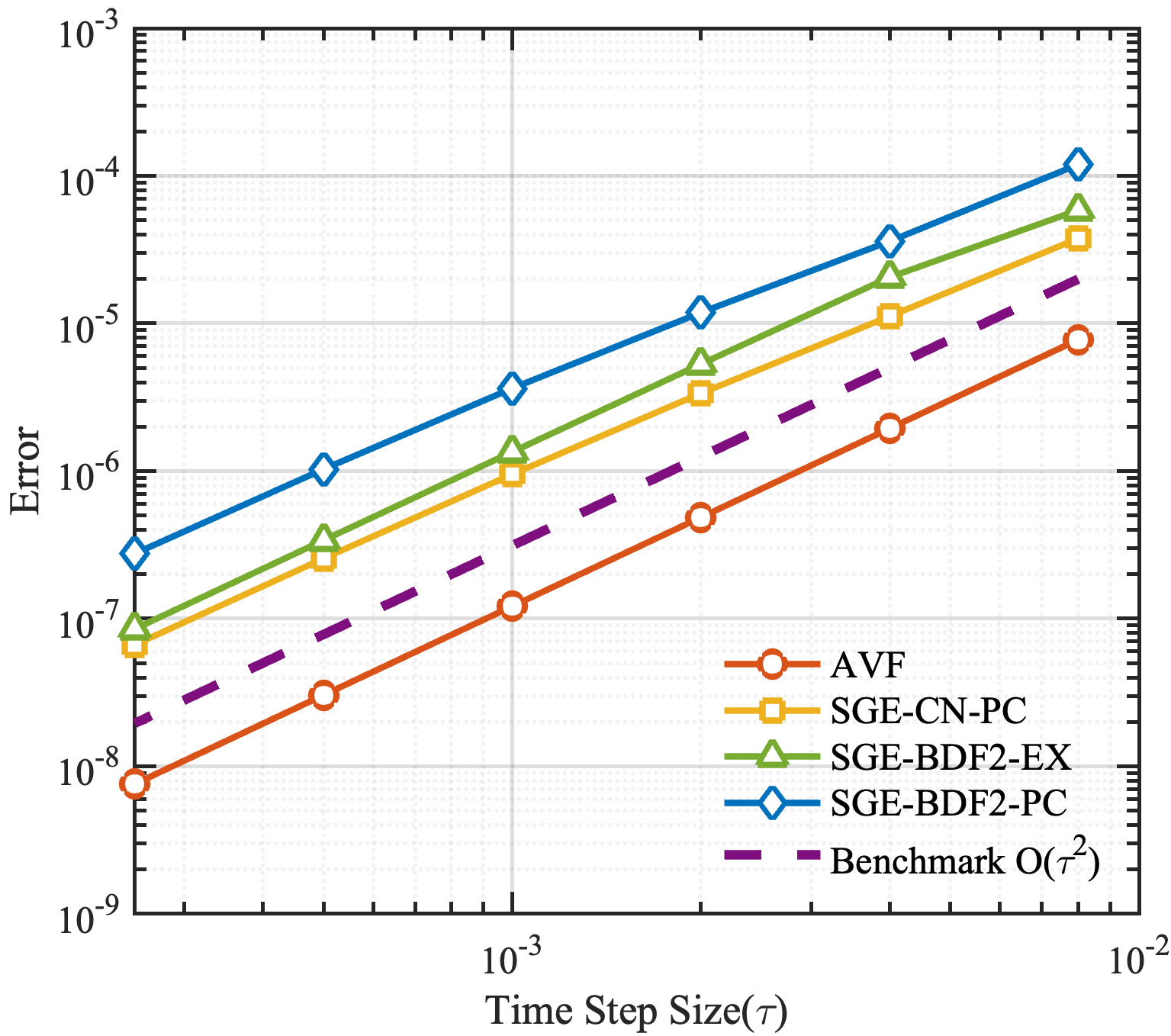}
	\includegraphics[width=0.3\textwidth, height=100pt]{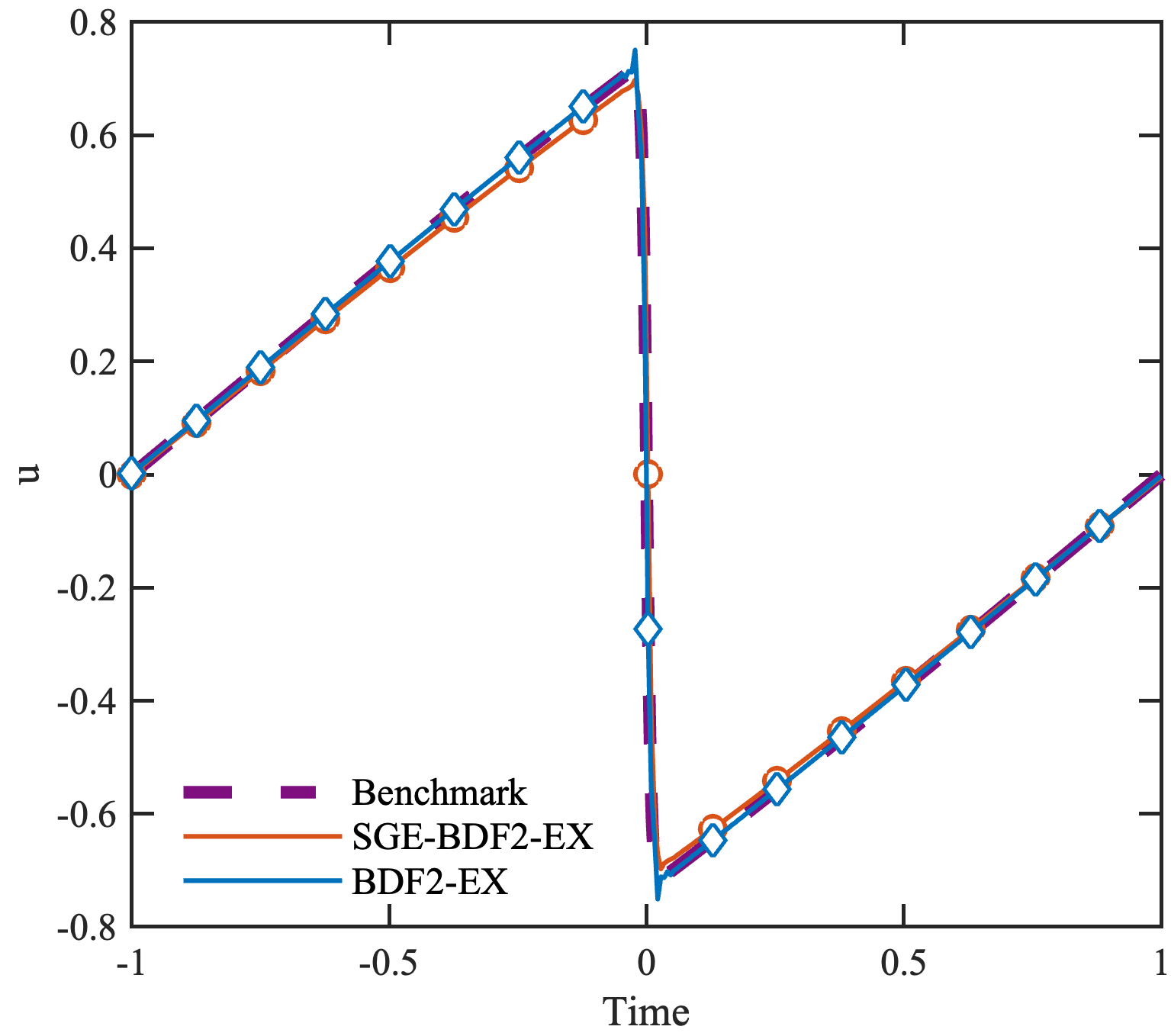}
	\includegraphics[width=0.3\textwidth, height=100pt]{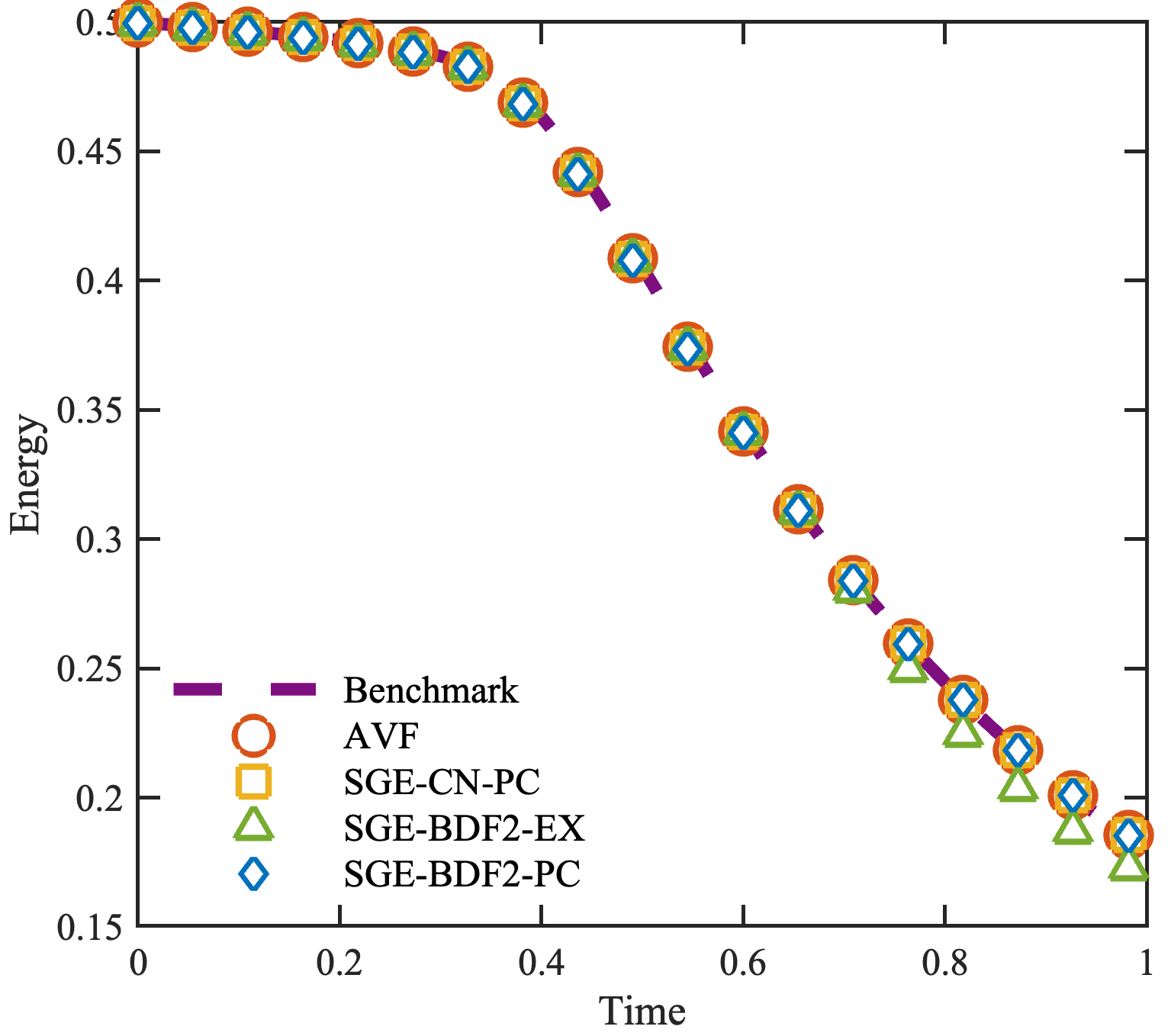}
	\caption{Numerical results of the Burgers equations. Five methods are used to produce the results. The benchmark is computed using the 6th order Petrov-Galerkin method and AVF is the energy dissipation rate preserving method using the AVF approximation to the  free energy. The other methods are explained in the text. } \label{fig:burgers-fig}
\end{figure}

Furthermore, the AVF scheme exhibits the highest accuracy due to its fully implicit nature, while SGE-CN-PC and SGE-DBF2-PC schemes perform similarly. The SGE-BDF2-EX scheme however, is the least accurate, reflecting limitations in the accuracy and robustness of the extrapolation approach employed. The second subplot of Figure \ref{fig:burgers-fig} compares the solution at $t=1$ computed using the SGE-BDF2-EX scheme with those obtained from the classical implicit-explicit SGE-BDF2-EX scheme. We observe that the classical BDF2-EX scheme produces an oscillatory solution, whereas the SGE-BDF2-EX scheme yields the correct result, thereby highlighting the robustness of energy-stable schemes over other schemes. The final subplot illustrates the evolution of the discrete energy for the proposed methods, which decrease monotonically in time.

\subsection{Incompressible NS Equation}

In the second example, we consider the incompressible NS equation, where the velocity field is subject to Dirichlet boundary conditions.

The first-order SGE-BDF1 and second-order SGE-CN schemes for the NS equation based on  reformulation \eqref{eq:ns-generalized-gdflow} are given by
\begin{equation}\label{eq:ns-scheme-ggd}
	\mathbf{R}\cdot  \delta_\tau \Phi^{n+1} = (\mathbf{M} + \mathbf{S}_1(\hat{\Phi}^{n+\delta}) + \mathbf{S}_2(\Phi^{n+\delta})) \delta_\Phi F\left(\Phi^{n+\delta}\right),
\end{equation}
where $\delta = 1$ for the SGE-BDF1 approach and $\delta = \frac{1}{2}$ for the SGE-CN scheme. For efficiency, we choose $\hat{\Phi}^{n+1} = \Phi^n$ and $\hat{\Phi}^{n+\frac{1}{2}} = \tfrac{3}{2} \Phi^n - \tfrac{1}{2}\Phi^{n-1}$. Alternatively, \eqref{eq:ns-scheme-ggd} can be recast into its component form as follows
\begin{equation*}
	\left\lbrace
	\begin{aligned}
		 & \delta_\tau \mathbf{v}^{n+1} = \tfrac{1}{Re} \Delta \mathbf{v}^{n+\delta} - \nabla p^{n+\delta} - \tfrac{\hat{\mathbf{v}}^{n+\delta} \wedge (\hat{\mathbf{v}}^{n+\delta} \cdot \nabla \hat{\mathbf{v}}^{n+\delta})}{\|\hat{\mathbf{v}}^{n+\delta}\|^2} \mathbf{v}^{n+\delta} \\
		 & \nabla \cdot \mathbf{v}^{n+\delta} = 0.
	\end{aligned}
	\right.
\end{equation*}
The above scheme can be solved efficiently. For illustration, we consider the case $\delta = 1$ and define
\begin{equation*}
	\hat{\mathbf{g}}^{n+1} = \tfrac{\mathbf{v}^n}{\|\mathbf{v}^n\|^2}, \ \hat{\mathbf{c}}^{n+1} = (\mathbf{v}^n \cdot \nabla)\mathbf{v}^n, \  \xi^{n+1} = (\hat{\mathbf{g}}^{n+1}, \mathbf{v}^{n+1}), \ \eta^{n+1} = (\hat{\mathbf{c}}^{n+1}, \mathbf{v}^{n+1}).
\end{equation*}
We decompose $\mathbf{v}^{n+1}$ in the following form
\begin{equation} \label{eq:NS-Combination}
	\mathbf{v}^{n+1} = \mathbf{v}_1^{n+1} - \xi^{n+1} \mathbf{v}_2^{n+1} + \eta^{n+1} \mathbf{v}_3^{n+1},
\end{equation}
where, $\mathbf{v}_i^{n+1}\ (i=1,2,3)$ denote solutions of three generalized Stokes systems
\begin{equation*}
	\left\lbrace
	\begin{aligned}
		 & \left(\tfrac{1}{\tau} - \tfrac{1}{Re}\Delta\right) \mathbf{v}_i^{n+1} + \nabla p_i^{n+1} = \mathbf{f}_i^{n+1}, \quad i=1,2,3, \\
		 & \nabla \cdot \mathbf{v}_i^{n+1} = 0,
	\end{aligned}
	\right.
\end{equation*}
$\mathbf{f}_1^{n+1} = \frac{1}{\tau} \mathbf{v}^n, \ \mathbf{f}_2^{n+1} =  \hat{\mathbf{c}}^{n+1}, \ \mathbf{f}_3^{n+1} = \hat{\mathbf{g}}^{n+1}$. After computing $\mathbf{v}_i^{n+1}$, we obtain the following $2 \times 2$ system for $\xi^{n+1}$ and $\eta^{n+1}$.
\begin{equation*}
	\left\lbrace
	\begin{aligned}
		 & ( 1 + (\hat{\mathbf{g}}^{n+1}, \mathbf{v}_2^{n+1})) \xi^{n+1} - (\hat{\mathbf{g}}^{n+1}, \mathbf{v}_3^{n+1}) \eta^{n+1} = (\hat{\mathbf{g}}^{n+1}, \mathbf{v}_1^{n+1}), \\
		 & (\hat{\mathbf{c}}^{n+1}, \mathbf{v}_2^{n+1}) \xi^{n+1} + ( 1 - (\hat{\mathbf{c}}^{n+1}, \mathbf{v}_3^{n+1})) \eta^{n+1} = (\hat{\mathbf{c}}^{n+1}, \mathbf{v}_1^{n+1}),
	\end{aligned}
	\right.
\end{equation*}
In the proposed discretization, we only need to solve three generalized Stokes equations and a small $2 \times 2$ linear system, thanks to the explicit implementation of the convective term. To further enhance efficiency, alternative techniques, such as projection methods (see \cite{Shen-Projection}), can be employed to discretize the NS equations. This approach would reduce the problem to solving only Poisson equations with constant coefficients, which can be solved efficiently using fast Fourier transformation on the square domain $\Omega$. Due to space limitations, a detailed discussion of the projection method is omitted.

\begin{figure}[htbp]
	\begin{center}
		\includegraphics[width=0.3\textwidth, height=101pt]{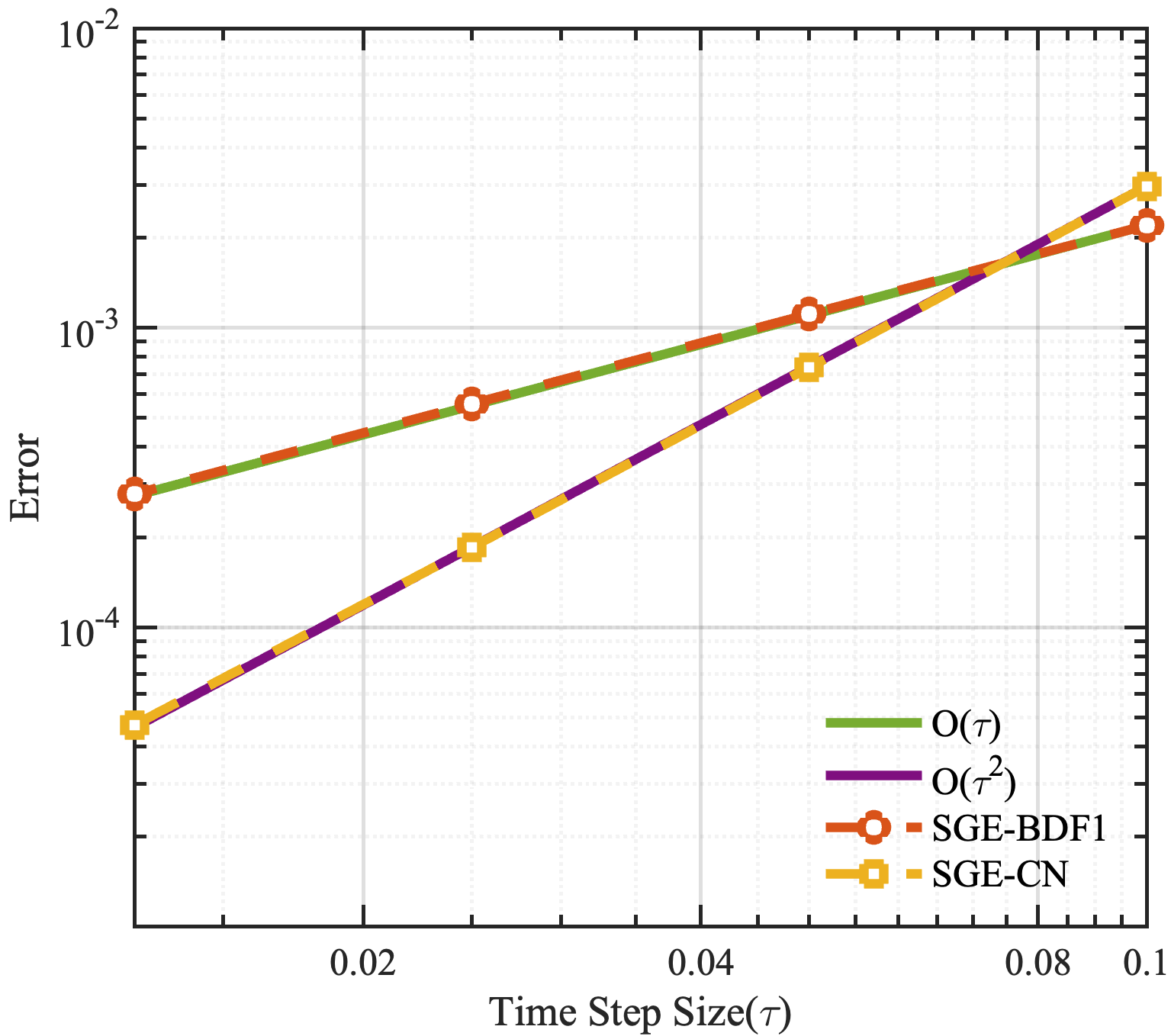}
		\includegraphics[width=0.3\textwidth, height=100pt]{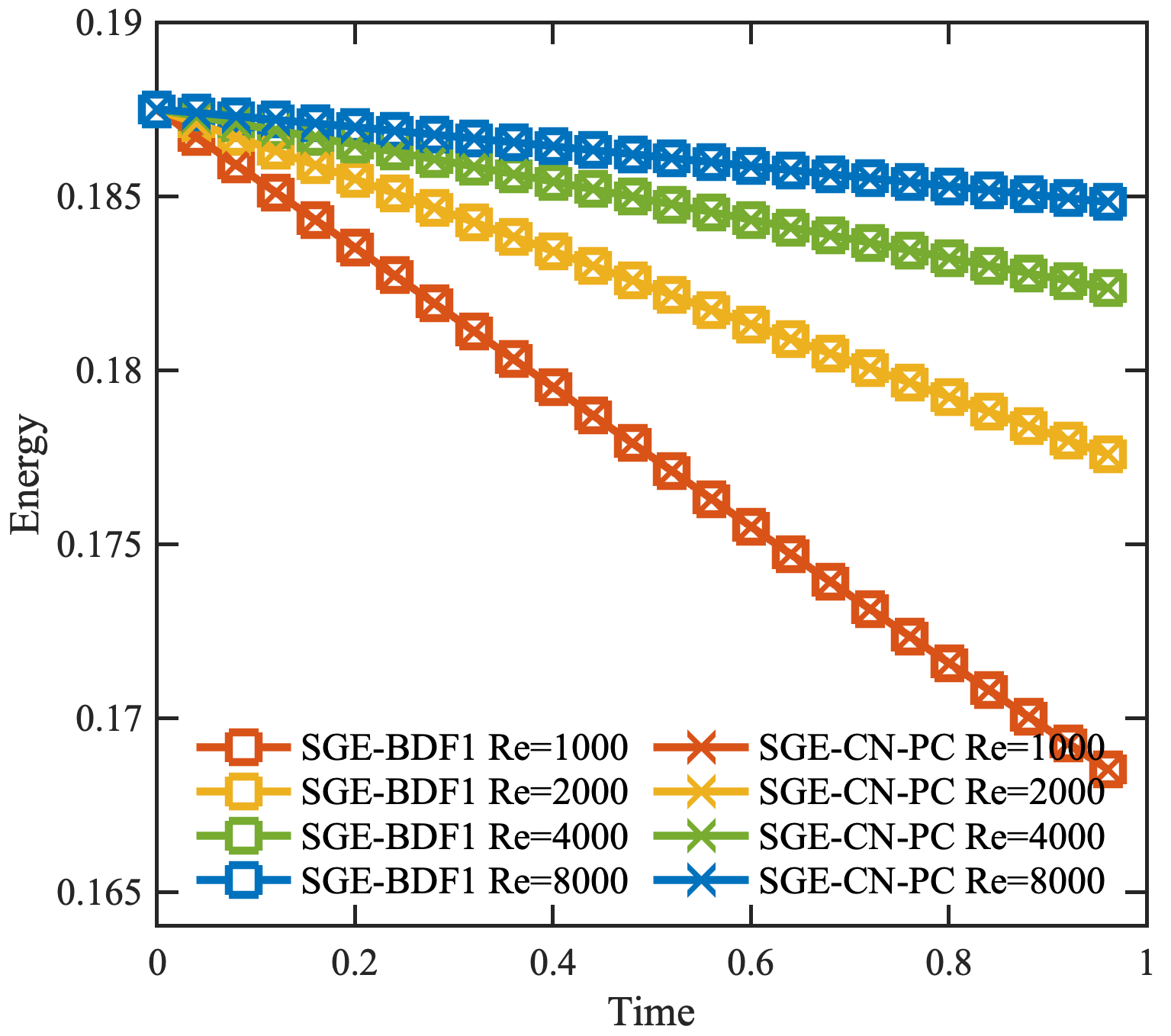}
	\end{center}
	\caption{Temporal accuracy test (left)  and time history of the discrete energy (right) of the NS equations}\label{fig:ns-fig}
\end{figure}

We simulate the forced incompressible NS system in the unit square domain with the following exact solution
\begin{equation}\label{eq:ns-exact}
	\left\lbrace
	\begin{aligned}
		 & u(x, y, t) = 0.5 \cos{(t^2)} \left( -\cos{(2 \pi x)} \sin{(2 \pi y)} + \sin{(2 \pi y)} \right), \\
		 & v(x, y, t) = 0.5 \cos{(t^2)} \left(\sin{(2 \pi x)} \cos{(2 \pi y)} - \sin{(2 \pi x)}\right),    \\
		 & p(x, y, t) = 0.5 \cos{(t^2)} \cos{(2 \pi x)} \cos{(2 \pi y)},                                   \\
	\end{aligned}
	\right.
\end{equation}
which can be derived by augmenting the momentum equation with a source term. To assess the temporal convergence of the proposed schemes, we fix the spatial resolution at $N=1024$ to ensure that spatial errors are negligible. The NS equation is then solved using different methods until $t=1$. The first subplot of Figure \ref{fig:ns-fig} illustrates the results of the mesh refinement, revealing that the SGE-BDF1 and SGE-CN schemes achieve first- and second-order accuracy, respectively, as anticipated.

We then investigate the energy dissipation property. We solve the NS equation using the same initial condition as in \eqref{eq:ns-exact} but without the source term, for different Reynolds numbers with a fixed time step, $\tau=$1e-3. The temporal evolution of the energy for various Reynolds numbers is shown in the right subplot of Figure \ref{fig:ns-fig}. The energy profiles obtained from both schemes agree very well and decrease monotonically across all Reynolds numbers.
\begin{figure}[htbp]
	\begin{center}
		\includegraphics[width=0.3\textwidth]{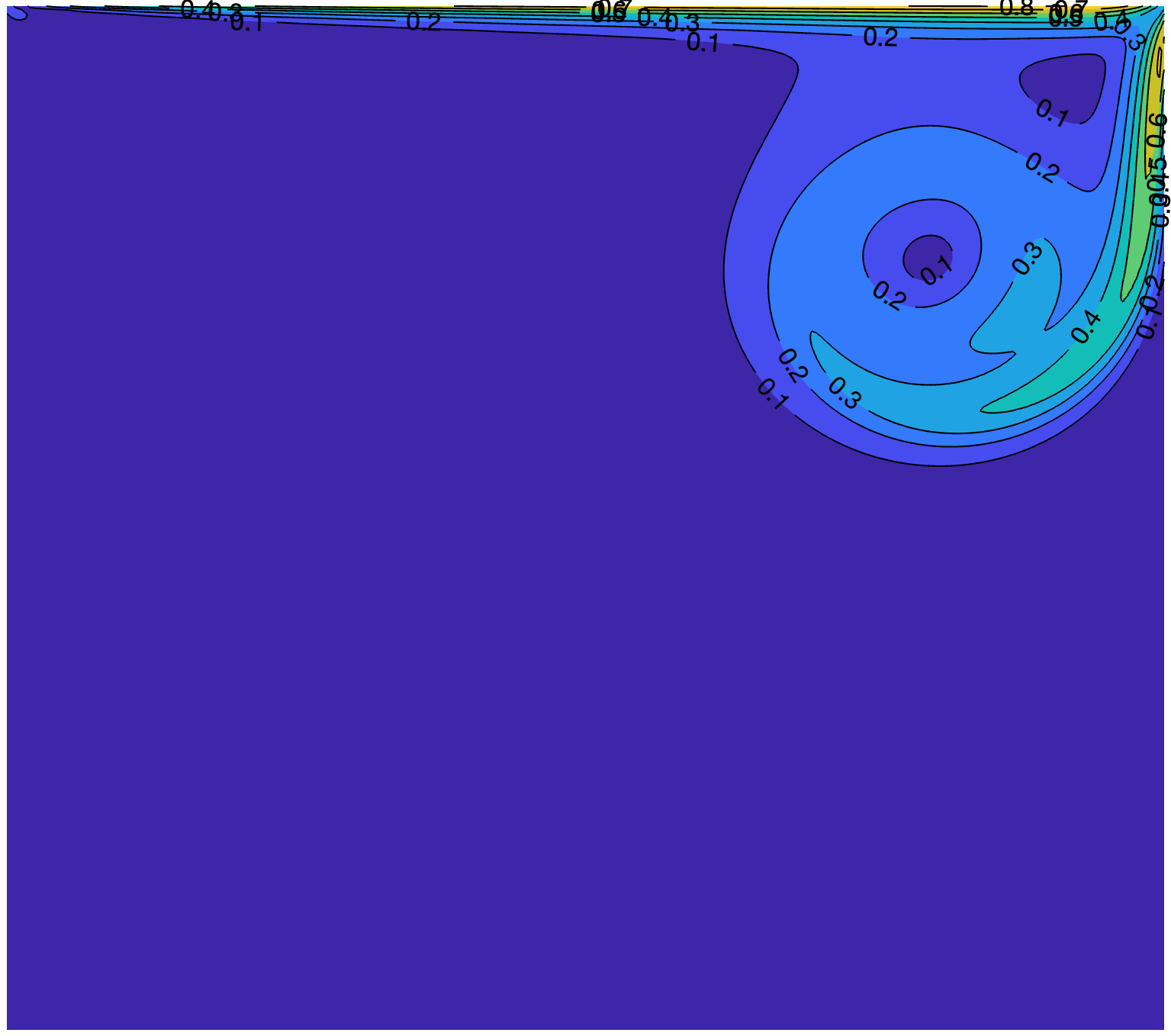}
		\includegraphics[width=0.3\textwidth]{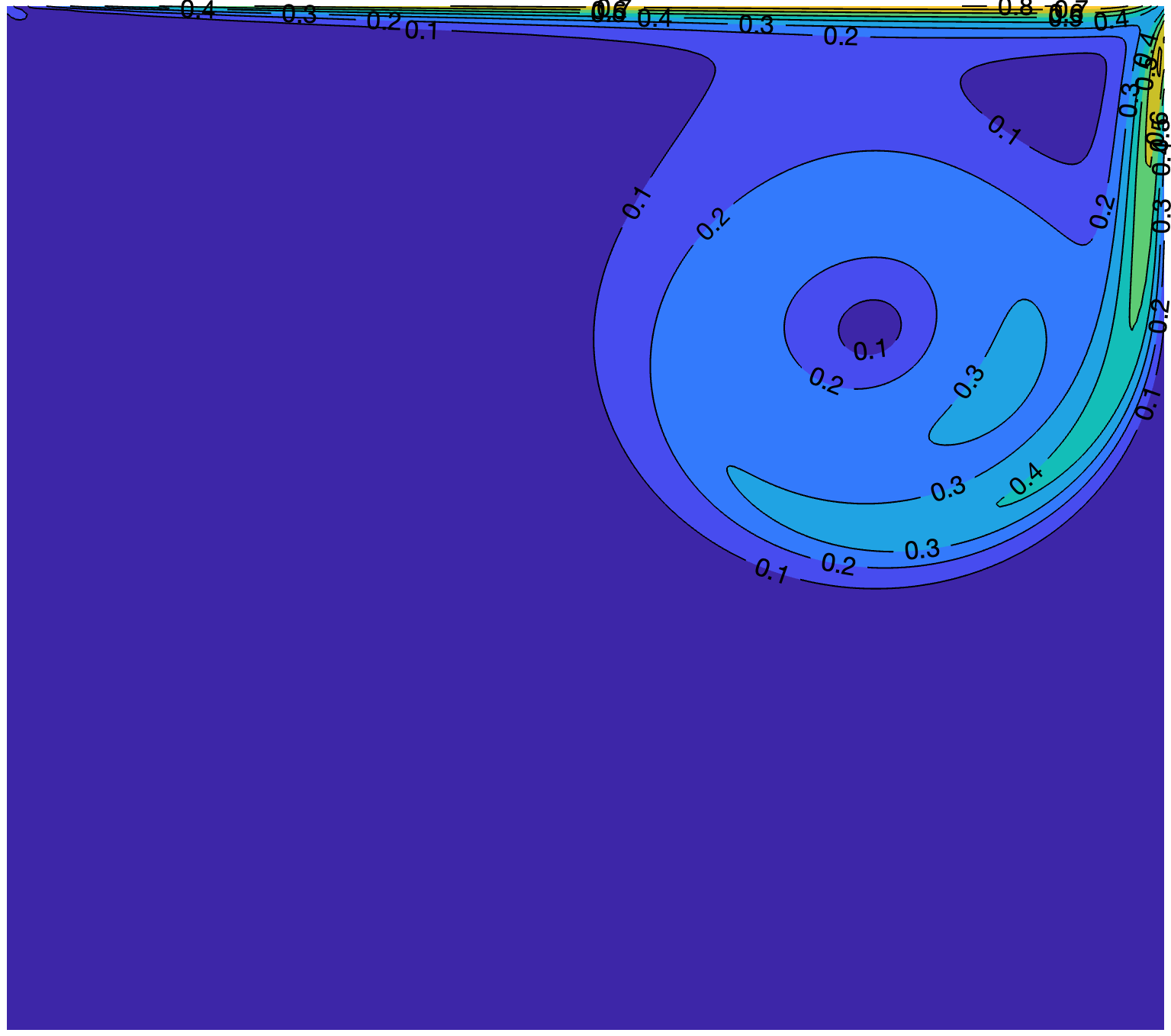}
		\includegraphics[width=0.3\textwidth]{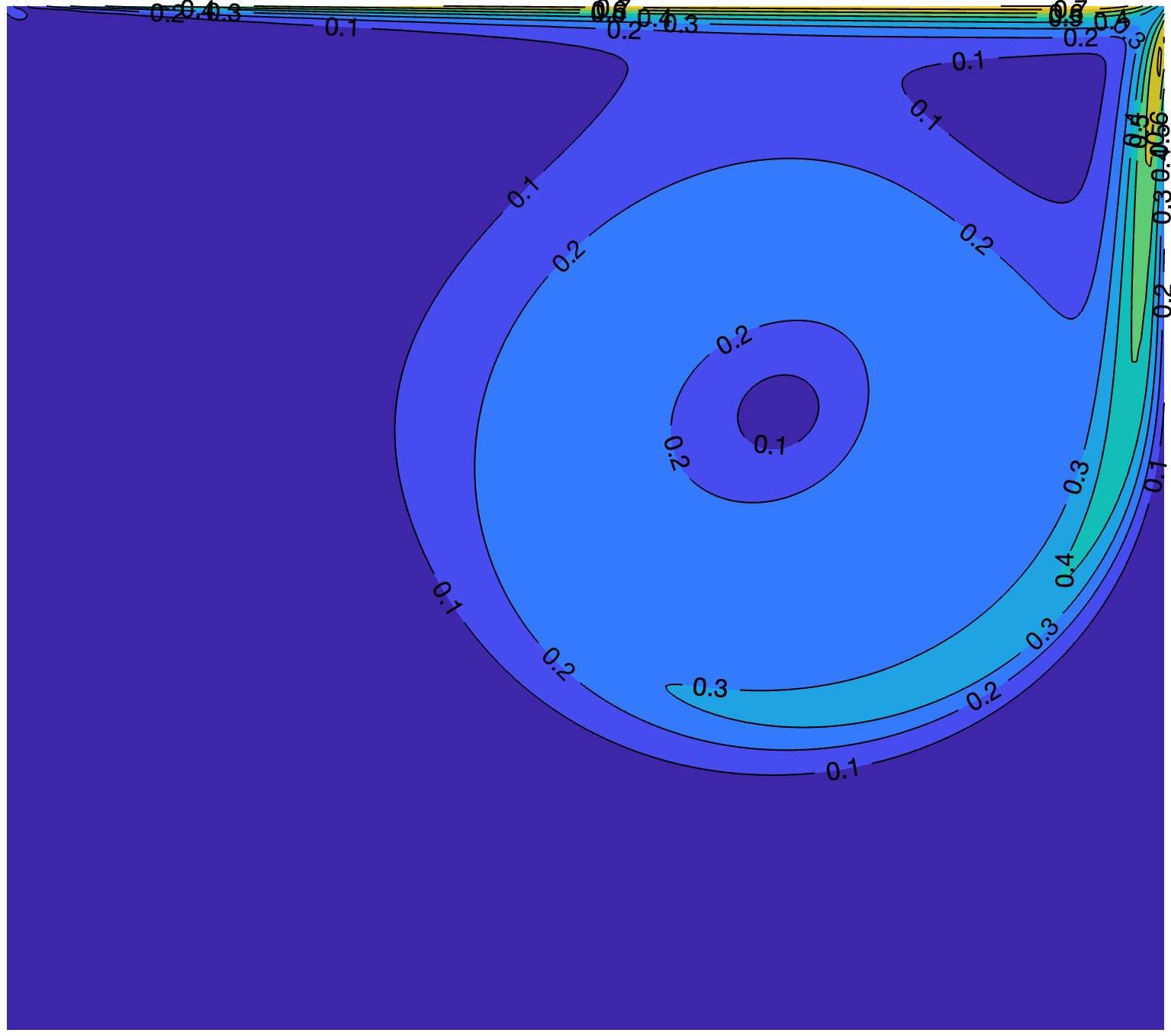}
	\end{center}
	\caption{Snapshots of the lid-driven cavity benchmark problem at $t = 4, 6, 10$ for $Re = 5000$. }\label{fig:driven-cavity}
\end{figure}
We next investigate the classical lid-driven cavity benchmark problem on the unit square domain. The horizontal velocity is prescribed as $u(x, 1, t) = 1$ at the top boundary, while no-slip boundary conditions are applied on all other boundaries. The initial velocity field is set identically zero. Figure \ref{fig:driven-cavity} presents the solution of the driven cavity problem solved by the SGE-CN scheme with time step $\tau=$4e-3 and $Re = 5000$ at $t = 4, 6, 10$, respectively. The results are in excellent agreement with those reported in \cite{Doan-2025} using the dynamically regularized Lagrange multiplier method.

\subsection{CH-NS equation}
Based on \eqref{eq:ch-gradient}, the second-order SGE-SBDF2 scheme for the CH-NS equation is given by
\begin{equation}\label{eq:chns-bdf2}
	\left\lbrace
	\begin{aligned}
		 & \mathbf{R} \cdot \delta_{2\tau}\Phi^{n+1} = (\mathbf{M} + \mathbf{S}_1(\hat{\Phi}^{n+1}) + \mathbf{S}_2(\Phi^{n+1}) ) ( \mathbf{v}^{n+1},  \overline{\mu}^{n+1}, 0 )^\top,               \\
		 & \mu^{n+1} = - \gamma \varepsilon \Delta \phi^{n+1} - \tfrac{\gamma}{\varepsilon} \tau M A\Delta d_t \phi^{n+1} + \tfrac{\gamma}{\varepsilon} (2f^\prime(\phi^n) - f^\prime(\phi^{n-1})),
	\end{aligned}
	\right.
\end{equation}
where $\overline{\mu}^{n+1} = \mu^{n+1} - \tfrac{1}{|\Omega|} \int_\Omega \mu^{n+1} d\mathbf{x}$.

The second-order stabilization SGE-SCN scheme for the CH-NS equation is
\begin{equation} \label{eq:chns-cn}
	\left\lbrace
	\begin{aligned}
		 & \mathbf{R} \cdot \delta_\tau \Phi^{n+1} = (\mathbf{M} + \mathbf{S}_1(\hat{\Phi}^{n+\frac{1}{2}}) + \mathbf{S}_2(\Phi^{n+\frac{1}{2}}) ) ( \tilde{\mathbf{v}}^{n+\frac{1}{2}},  \overline{\mu}^{n+\frac{1}{2}}, 0 )^\top, \\
		 & \mu^{n+\frac{1}{2}} = - \gamma \varepsilon\Delta \tilde{\phi}^{n+\frac{1}{2}} - \tfrac{\gamma}{\varepsilon} \tau M A \Delta d_{2t} \phi^{n+1} + \tfrac{\gamma}{2\varepsilon} (3f^\prime(\phi^n) - f^\prime(\phi^{n-1})), \\
	\end{aligned}
	\right.
\end{equation}
where $\tilde{\phi}^{n+\frac{1}{2}} = \tfrac{3}{4}\phi^{n+1} + \tfrac{1}{4}\phi^{n-1}$.

Notice that the stabilized term is only added partially to the Cahn-Hilliard component here, but we will show that \eqref{eq:chns-bdf2} and \eqref{eq:chns-cn} are energy stable. We only discuss the SGE-SBDF2 scheme as that for the SGE-SCN scheme is analogous. The SGE-SBDF2 scheme, expressed in the component form, is given by
\begin{equation}\label{eq:ch-bdf2-component}
	\left\lbrace
	\begin{aligned}
		 & \delta_{2\tau} \mathbf{v}^{n+1} = \rho^{-1}( \nu \Delta \mathbf{v}^{n+1} - \nabla p^{n+1} ) - (\hat{\mathbf{g}}^{n+1}, \mathbf{g}^{n+1}) \hat{\mathbf{c}}^{n+1}_{\mathbf{v}} + (\hat{\mathbf{c}}^{n+1}, \mathbf{g}^{n+1}) \hat{\mathbf{g}}_{\mathbf{v}}^{n+1}, \\
		 & \nabla \cdot \mathbf{v}^{n+1} = 0,                                                                                                                                                                                                                             \\
		 & \delta_{2\tau}\phi^{n+1} = M \Delta \overline{\mu}^{n+1} - (\hat{\mathbf{g}}^{n+1}, \mathbf{g}^{n+1}) \hat{\mathbf{c}}_\phi^{n+1} + (\hat{\mathbf{c}}^{n+1}, \mathbf{g}^{n+1}) \hat{\mathbf{g}}_\phi^{n+1}                                                     \\
		 & \mu^{n+1} = - \gamma \varepsilon \Delta \phi^{n+1} - \tfrac{\gamma}{\varepsilon} \tau MA \Delta d_{2t} \phi^{n+1} + \tfrac{\gamma}{\varepsilon} (2f^\prime(\phi^n) - f^\prime(\phi^{n-1})),
	\end{aligned}
	\right.
\end{equation}
where,
\begin{equation*}
	\begin{aligned}
		 & \mathbf{g}^{n+1} = (\mathbf{g}_\mathbf{v}^{n+1}, \mathbf{g}_\phi^{n+1}), \ \mathbf{g}_\mathbf{v}^{n+1} =  \rho \mathbf{v}^{n+1}, \ \mathbf{g}_{\phi}^{n+1} = \overline{\mu}^{n+1},                                                                                                                                                                                                        \\
		 & \hat{\mathbf{g}}^{n+1} = (\hat{\mathbf{g}}_\mathbf{v}^{n+1}, \hat{\mathbf{g}}_\phi^{n+1}), \ \hat{\mathbf{g}}^{n+1}_\mathbf{v} = \tfrac{\rho \hat{\mathbf{v}}^{n+1}}{\rho^2 \|\hat{\mathbf{v}}^{n+1}\|^2 + \|\hat{\overline{\mu}}^{n+1}\|^2}, \ \hat{\mathbf{g}}_\phi^{n+1} = \tfrac{\hat{\overline{\mu}}^{n+1}}{\rho^2 \|\hat{\mathbf{v}}^{n+1}\|^2 + \|\hat{\overline{\mu}}^{n+1}\|^2}, \\
		 & \hat{\mathbf{c}}^{n+1} = (\hat{\mathbf{c}}_\mathbf{v}^{n+1}, \hat{\mathbf{c}}_\phi^{n+1}), \ \hat{\mathbf{c}}^{n+1}_\mathbf{v} = \hat{\mathbf{v}}^{n+1} \cdot \nabla \hat{\mathbf{v}}^{n+1} + \tfrac{1}{\rho} \hat{\phi}^{n+1} \nabla \hat{\overline{\mu}}^{n+1}, \hat{\mathbf{c}}_\phi^{n+1} = \nabla \cdot (\hat{\mathbf{v}}^{n+1} \hat{\phi}^{n+1}).
	\end{aligned}
\end{equation*}
We next show  that \eqref{eq:ch-bdf2-component} preserves mass conservation and possesses energy stability.
\begin{theorem}
	Suppose initial data $\phi^0$ and first-step numerical solution $\phi^1$ satisfy mass constraint $(\phi^0, 1) = (\phi^1, 1) = C$. Then, for any $n \geq 2$, the solution $\phi^n$ obtained from \eqref{eq:ch-bdf2-component} preserves mass conservation
	\begin{equation*}
		(\phi^n, 1) = C, \quad n \geq 2,
	\end{equation*}
	and provided the stabilization parameter $A$ is sufficiently large, it is energy stable in the sense that
	\begin{equation*}
		\begin{gathered}
			\widetilde{E}_{BDF2}^{n+1} - \widetilde{E}_{BDF2}^n \leq - \tfrac{\gamma}{\varepsilon} ( 2 \min \{1, \sqrt{\nu}\} \tfrac{M \sqrt{A}}{M + \beta_\Omega} - 3 L),
		\end{gathered}
	\end{equation*}
	where $\beta_\Omega$ will be defined in the proof
	\begin{equation*}
		\begin{gathered}
			\widetilde{E}_{BDF2}^{n+1} = \tfrac{\rho}{4} \|\mathbf{v}^{n+1}\|^2 + \tfrac{\rho}{4}\|2\mathbf{v}^{n+1} - \mathbf{v}^n\|^2 + \tfrac{\varepsilon \gamma}{4} \|\nabla \phi^{n+1}\|^2 + \tfrac{\varepsilon \gamma}{4} \|\nabla (2\phi^{n+1} - \phi^n)\|^2 \\
			\quad + \tfrac{\gamma}{2\varepsilon} (3f(\phi^{n+1}) - f(\phi^n))+ \tfrac{\gamma}{2\varepsilon}\left(2 \min \{1, \sqrt{\nu}\} \tfrac{M \sqrt{A}}{M + \beta_\Omega} + 3 L\right) \|d_t\phi^{n+1}\|^2.
		\end{gathered}
	\end{equation*}
\end{theorem}
\begin{proof}
	We note here that, by the definition, $(\hat{\mathbf{c}}_\phi^{n+1}, 1) = (\hat{\mathbf{g}}_\phi^{n+1}, 1) = (\Delta \overline{\mu}^{n+1}, 1)= 0$. Taking the inner-product of both sides of the fourth equation in \eqref{eq:ch-bdf2-component} with $2\tau$, we have
	\begin{equation*}
		(d_{2t} \phi^{n+1}, 1) = 0.
	\end{equation*}
	Combining this result with the initial condition and applying an inductive argument, we obtain
	\begin{equation*}
		(\phi^{n+1}, 1) = \tfrac{1}{3} (4\phi^n - \phi^{n-1}, 1) = C.
	\end{equation*}
	Taking the inner product of the first quation in \eqref{eq:ch-bdf2-component} with $2\tau\mathbf{g}_\mathbf{v}^{n+1} = 2\tau\rho \mathbf{v}^{n+1}$, and then applying integration by parts, together with the incompressibility condition, we deduce
	\begin{equation}\label{eq:chhs-es-01}
		\begin{gathered}
			d_t \left( \tfrac{\rho}{2}\|\mathbf{v}^{n+1}\|^2 + \tfrac{\rho}{2} \|2\mathbf{v}^{n+1} - \mathbf{v}^n\|^2\right) + \tfrac{\rho}{2} \|d_{tt} \mathbf{v}^n\|^2 \\
			= -2\tau (\nu\|\nabla \mathbf{v}^{n+1}\|^2 + (\hat{\mathbf{g}}^{n+1}, \mathbf{g}^{n+1})(\hat{\mathbf{c}}_\mathbf{v}^{n+1}, \mathbf{g}_\mathbf{v}^{n+1}) - (\hat{\mathbf{c}}^{n+1}, \mathbf{g}^{n+1}) (\hat{\mathbf{g}}_{\mathbf{v}}^{n+1}, \mathbf{g}_{\mathbf{v}}^{n+1})).
		\end{gathered}
	\end{equation}
	Taking the inner product of the third equation in \eqref{eq:ch-bdf2-component} with $2\tau \mathbf{g}^{n+1}_{\phi} = 2\tau \overline{\mu}^{n+1}$ yields
	\begin{equation}\label{eq:chns-es-02}
		\begin{gathered}
			(d_{2t} \phi^{n+1}, \overline{\mu}^{n+1}) = -2\tau\|\nabla \overline{\mu}^{n+1}\|^2 \\
			- 2\tau(\hat{\mathbf{g}}^{n+1}, \mathbf{g}^{n+1}) (\hat{\mathbf{c}}_{\phi}^{n+1}, \hat{\mathbf{g}}_\phi^{n+1}) + 2\tau(\hat{\mathbf{c}}^{n+1}, \mathbf{g}^{n+1})(\hat{\mathbf{g}}_\phi^{n+1}, \mathbf{g}_\phi^{n+1}).
		\end{gathered}
	\end{equation}
	Furthermore, by taking the inner product of the last equation in \eqref{eq:ch-bdf2-component} with $d_{2t} \phi^{n+1}$ and invoking the argument detailed in the proof of Theorem \ref{thm:stability-bdf2}, we deduce
	\begin{equation}\label{eq:chns-es-03}
		\begin{aligned}
			 & d_t \left(\tfrac{\gamma \varepsilon}{2} \|\nabla\phi^{n+1}\|^2 + \tfrac{\gamma \varepsilon}{2} \|2\nabla\phi^{n+1} - \nabla \phi^n\|^2 + \tfrac{\gamma}{\varepsilon} (3 f(\phi^{n+1}) - f(\phi^n))\right) \\
			 & \ \leq - \tfrac{\gamma \varepsilon}{2} \|\nabla d_{tt} \phi^n\|^2 - \tfrac{\gamma}{\varepsilon} \tau M A \|\nabla d_{2t} \phi^{n+1}\|^2 + (\mu^{n+1}, d_{2t}\phi^{n+1})                                   \\
			 & \ \ + \tfrac{3\gamma L}{\varepsilon} (\|d_t \phi^{n+1}\|^2 + \|d_t \phi^n\|^2).
		\end{aligned}
	\end{equation}
	Invoking the definition of $\overline{\mu}^{n+1}$ together with the mass conservation immediately yields $(d_{2t} \phi^{n+1}, \overline{\mu}^{n+1}) = (d_{2t} \phi^{n+1}, \mu^{n+1})$. Thereafter, by summing \eqref{eq:chhs-es-01}, \eqref{eq:chns-es-02}, \eqref{eq:chns-es-03}, we arrive at
	\begin{equation}\label{eq:chns-es-E1}
		\begin{aligned}
			 & d_t \tilde{E}_1 (\phi^n, \phi^{n+1}, \mathbf{v}^n, \mathbf{v}^{n+1}) \leq -2\tau \nu \|\nabla \mathbf{v}^{n+1}\|^2 - 2\tau \|\nabla \overline{\mu}^{n+1}\|^2 - \tfrac{\rho}{2} \|d_{tt} \mathbf{v}^n\|^2      \\
			 & - \tfrac{\gamma \varepsilon}{2} \|\nabla d_{tt} \phi^n\|^2  - \tfrac{\gamma}{\varepsilon} \tau M A \|\nabla d_{2t} \phi^{n+1}\|^2 + \tfrac{3\gamma L}{\varepsilon} (\|d_t \phi^{n+1}\|^2 + \|d_t \phi^n\|^2),
		\end{aligned}
	\end{equation}
	where
	\begin{equation*}
		\begin{aligned}
			 & \tilde{E}_1 (\phi^n, \phi^{n+1}, \mathbf{v}^n, \mathbf{v}^{n+1}) = \tfrac{\rho}{2}\|\mathbf{v}^{n+1}\|^2 + \tfrac{\rho}{2} \|2\mathbf{v}^{n+1} - \mathbf{v}^n\|                               \\
			 & \ + \tfrac{\gamma \varepsilon}{2} \|\nabla\phi^{n+1}\|^2 + \tfrac{\gamma \varepsilon}{2} \|2\nabla\phi^{n+1} - \nabla \phi^n\|^2 + \tfrac{\gamma}{\varepsilon} (3 f(\phi^{n+1}) - f(\phi^n)).
		\end{aligned}
	\end{equation*}
	By taking the \(L^2\)–inner product of the third equation in \eqref{eq:ch-bdf2-component} with \(\omega\,2\tau\,d_{2t}\phi^{n+1}\), we obtain
	\begin{equation}\label{eq:chns-es-04}
		\begin{aligned}
			 & \omega \|d_{2t} \phi^{n+1}\|^2  \leq - 2 \tau \omega  M (\nabla \overline{\mu}^{n+1}, \nabla d_{2t} \phi^{n+1})                                                                                                          \\
			 & \ - 2\tau \omega (\hat{\mathbf{g}}^{n+1}, \mathbf{g}^{n+1}) (\hat{\mathbf{c}}_\phi^{n+1}, d_{2t} \phi^{n+1}) + 2\tau \omega (\hat{\mathbf{c}}^{n+1}, \mathbf{g}^{n+1}) (\hat{\mathbf{g}}_\phi^{n+1}, d_{2t} \phi^{n+1}).
		\end{aligned}
	\end{equation}
	Employing the Cauchy-Schwarz and Poincar\'e inequalities, performing integration by parts and invoking the divergence-free constraint, one readily establishes
	\begin{equation}\label{eq:chns-es-05}
		\begin{aligned}
			 & -(\hat{\mathbf{g}}^{n+1}, \mathbf{g}^{n+1}) (\hat{\mathbf{c}}_\phi^{n+1}, d_{2t} \phi^{n+1}) +  (\hat{\mathbf{c}}^{n+1}, \mathbf{g}^{n+1}) (\hat{\mathbf{g}}_\phi^{n+1}, d_{2t} \phi^{n+1}), \\
			 & \ \leq \beta_\Omega \left( \|\nabla \mathbf{v}^{n+1}\|  + \|\nabla \overline{\mu}^{n+1}\| \right) \|\nabla d_{2t} \phi^{n+1}\|.
		\end{aligned}
	\end{equation}
	Here, $\beta_\Omega = C_\Omega \max\limits_{0 \leq n \leq N_t} \{ \|\phi^n\|_\infty + \|\nabla \phi^{n}\|_\infty + \|\mathbf{v}^n\|_\infty)$, and $C_\Omega$ is a coefficient with respect to the area of domain. Combing \eqref{eq:chns-es-04} and \eqref{eq:chns-es-05} gives us
	\begin{equation} \label{eq:chns-es-06}
		\begin{aligned}
			\omega \|d_{2t} \phi^{n+1}\|^2 & \leq 2\tau \omega (M + \beta_\Omega) ( \|\nabla \overline{\mu}^{n+1}\|  + \|\nabla \mathbf{v}^{n+1}\|) \|\nabla d_{2t} \phi^{n+1}\|                                                                                      \\
			                               & \leq \tfrac{2\tau \omega^2 \epsilon}{\gamma} \tfrac{(M + \beta_\Omega)^2}{MA}  ( \|\nabla \overline{\mu}^{n+1}\|^2 + \|\nabla \mathbf{v}^{n+1}\|^2) + \tfrac{\gamma}{\varepsilon} \tau MA \|\nabla d_{2t} \phi^{n+1}\|^2
		\end{aligned}
	\end{equation}
	Defining $\omega = \min\limits \{ 1, \sqrt{\nu} \} \tfrac{M \gamma}{\varepsilon(M + \beta_\Omega)} \sqrt{A}$ and then summing \eqref{eq:chns-es-05} and \eqref{eq:chns-es-E1} while invoking Lemma \ref{lem:difference-identity}, we deduce
	\begin{equation*}
		\begin{aligned}
			 & d_t \tilde{E}_1 + \tfrac{\gamma}{\varepsilon}\left(2 \min \{1, \sqrt{\nu}\} \tfrac{M \sqrt{A}}{M + \beta_\Omega} + 3 L\right) d_t \|d_t \phi^{n+1}\|^2 \leq - \tfrac{2\gamma}{\varepsilon} ( 2 \min \{1, \sqrt{\nu}\} \tfrac{M \sqrt{A}}{M + \beta_\Omega} - 3 L).
		\end{aligned}
	\end{equation*}
	The proof is thus completed.
\end{proof}

We now detail the implementation of the SGE-SBDF2 scheme. Through straightforward computations, one can express $\mathbf{v}^{n+1}, \phi^{n+1}, \overline{\mu}^{n+1}$ as follows
\begin{equation*}
	\begin{aligned}
		 & \mathbf{v}^{n+1} = \mathbf{v}_1^{n+1} - \xi^{n+1} \mathbf{v}_2^{n+1} + \eta^{n+1} \mathbf{v}_3^{n+1},                 \\
		 & \phi^{n+1} = \phi_1^{n+1} - \xi^{n+1} \phi_2^{n+1} + \eta^{n+1} \phi_3^{n+1}             ,                            \\
		 & \overline{\mu}^{n+1} = \overline{\mu}_1^{n+1} - \xi^{n+1} \overline{\mu}_2^{n+1} + \eta^{n+1} \overline{\mu}_3^{n+1},
	\end{aligned}
\end{equation*}
where we define $\xi^{n+1} = (\hat{\mathbf{g}}^{n+1}, \mathbf{g}^{n+1})$, $\eta^{n+1} = (\hat{\mathbf{c}}^{n+1}, \mathbf{g}^{n+1})$. Moreover, $\phi_i^{n+1}$ and $\mathbf{v}_i^{n+1}$ denote solutions of the following decoupled and linear CH-NS systems
\begin{equation*}
	\left\lbrace
	\begin{aligned}
		 & (\tfrac{3}{2\tau} - \tfrac{\nu}{\rho} \Delta )\mathbf{v}_i^{n+1} + \tfrac{1}{\rho} \nabla p_i^{n+1} = \mathbf{f}_{i, \mathbf{v}},   i=1,2,3, \\
		 & \nabla \mathbf{v}_i^{n+1} = 0,                                                                                                               \\
		 & ( \tfrac{3}{2\tau} + \tfrac{3\gamma}{\varepsilon} \tau M^2 A \Delta^2 + \gamma \varepsilon M \Delta^2) \phi_i^{n+1} = \mathbf{f}_{i, \phi},
	\end{aligned}
	\right.
\end{equation*}
with $\mathbf{f}_{2, \mathbf{v}} = \hat{\mathbf{c}}_\mathbf{v}^{n+1},\ \mathbf{f}_{2, \phi} = \hat{\mathbf{c}}_\phi^{n+1},\ \mathbf{f}_{3, \mathbf{v}} = \hat{\mathbf{g}}_\mathbf{v}^{n+1},\ \mathbf{f}_{3, \phi} = \hat{\mathbf{g}}_\phi^{n+1}$, and
\begin{equation}
	\begin{aligned}
		 & \mathbf{f}_{1, \mathbf{v}} = \tfrac{4\mathbf{v}^n - \mathbf{v}^{n-1}}{2\tau},                                                                                                           \\
		 & \mathbf{f}_{1, \phi} = (\tfrac{2}{\tau} + \tfrac{4\gamma}{\varepsilon} \tau M^2 A  \Delta^2) \phi^n - ( \tfrac{1}{2\tau} + \tfrac{\gamma}{\varepsilon} \tau M^2 A \Delta^2  )\phi^{n-1} \\
		 & \ + \tfrac{\gamma}{\varepsilon} M \Delta (2 f^\prime(\phi^n) - f^\prime(\phi^{n-1}) ).                                                                                                  \\
	\end{aligned}
\end{equation}
Moreover, $\overline{\mu}_i^{n+1} = \mu_i^{n+1} - \tfrac{1}{|\Omega|} \int_\Omega \mu_i^{n+1} d\mathbf{x}$, with
\begin{equation*}
	\begin{aligned}
		 & \mu_1^{n+1} = - \gamma \varepsilon \Delta \phi_1^{n+1} - \tfrac{\gamma}{\varepsilon} \tau M A \Delta(3\phi_1 - 4\phi^n + \phi^{n-1}) + \tfrac{\gamma}{\varepsilon} (2f^\prime(\phi^n) - f^\prime(\phi^{n-1})),                             \\
		 & \mu_2^{n+1} = - \gamma \varepsilon \Delta \phi_2^{n+1} - \tfrac{3\gamma}{\varepsilon} \tau M A \Delta \phi_2^{n+1}, \ \mu_3^{n+1} = - \gamma \varepsilon \Delta \phi_3^{n+1} - \tfrac{3\gamma}{\varepsilon} \tau M A \Delta \phi_3^{n+1} .
	\end{aligned}
\end{equation*}
After the above preparations, the coefficients $\xi^{n+1}$ and $\eta^{n+1}$ are computed by solving a 2-by-2 linear system discussed earlier.

We consider the CH-NS equations on the unit square domain and construct the following exact solution by incorporating additional source terms:
\begin{equation*}
	\left\lbrace
	\begin{aligned}
		 & u(x, y, t) = 0.5 \cos{(t^2)} \left( -\cos{(2 \pi x)} \sin{(2 \pi y)} + \sin{(2 \pi y)} \right), \\
		 & v(x, y, t) = 0.5 \cos{(t^2)} \left(\sin{(2 \pi x)} \cos{(2 \pi y)} - \sin{(2 \pi x)}\right),    \\
		 & p(x, y, t) = 0.5 \cos{(t^2)} \cos{(2 \pi x)} \cos{(2 \pi y)},                                   \\
		 & \phi(x, y, t) = 0.5 \cos{(t^2)} \cos{(4 \pi x)} \cos{(4 \pi y)}.                                \\
	\end{aligned}
	\right.
\end{equation*}
In our tests, we set $\gamma = \varepsilon = M  = 1$. The stabilization parameter is set to $A = 3$.
\begin{table}[htbp]
	\caption{Error and convergence rate of the SGE-SBDF2 and SGE-SCN schemes for solving the CH-NS equations.} \label{tab:convergence-chns}
	\centering
	\resizebox{\linewidth}{!}{%
		\begin{tabular}{ccccccccc}
			\toprule
			\multirow{2}{*}{$(h, \tau)$}                   & \multicolumn{4}{c}{SGE-SBDF2} & \multicolumn{4}{c}{SGE-SCN}                                                                                                                   \\
			\cmidrule(lr){2-5} \cmidrule(lr){6-9}
			                                               & {Error}$_\mathbf{v}$          & {Order}$_\mathbf{v}$        & {Error}$_\phi$ & {Order}$_\phi$ & {Error}$_\mathbf{v}$ & {Order}$_\mathbf{v}$ & {Error}$_\phi$ & {Order}$_\phi$ \\
			\midrule
			$\left(\tfrac{1}{128}, \tfrac{1}{20}\right)$   & 5.2541e-05                    & $\star$                     & 2.3129e-03     & $\star$        & 2.2774e-03           & $\star$              & 1.1415e-03     & $\star$        \\
			$\left(\tfrac{1}{256}, \tfrac{1}{40}\right)$   & 1.4673e-05                    & 1.8403                      & 5.8260e-04     & 1.9891         & 5.6818e-04           & 2.0030               & 2.8410e-04     & 2.0065         \\
			$\left(\tfrac{1}{512}, \tfrac{1}{80}\right)$   & 3.8751e-06                    & 1.9209                      & 1.4591e-04     & 1.9974         & 1.4199e-04           & 2.0006               & 7.0931e-05     & 2.0019         \\
			$\left(\tfrac{1}{1024}, \tfrac{1}{160}\right)$ & 9.9541e-07                    & 1.9609                      & 3.6494e-05     & 1.9999         & 3.5498e-05           & 2.0000               & 1.7726e-05     & 2.0005         \\
			\bottomrule
		\end{tabular}
	}
\end{table}
As shown in Table \ref{tab:convergence-chns}, the refinement test confirms that both the SGE-SBDF2 and SGE-SCN schemes achieve second-order accuracy. Additionally, the SGE-SBDF2 scheme is more accurate than the SGE-SCN scheme.

Next, we simulate the merging of two bubbles driven by surface tension on a unit square domain. The initial velocity is set to zero, and the initial phase field $\phi$ is prescribed as
\begin{equation*}
	\phi(x, y, 0) = 1 - \tanh \tfrac{-r + \sqrt{(x - x_a)^2 + (y - y_a)^2}}{2\varepsilon} - \tanh \tfrac{-r + \sqrt{(x - x_b)^2 + (y - y_b)^2}}{2\varepsilon},
\end{equation*}
with $x_a = 0.5 - \tfrac{r}{\sqrt{2}}, y_a = 0.5 + \tfrac{r}{\sqrt{2}}, x_b = 0.5 + \tfrac{r}{\sqrt{2}}, y_b = 0.5 - \tfrac{r}{\sqrt{2}}$ and $r = 0.15$. The CH-NS equation is solved with $\rho = 1$, $M = 0.01$, $\gamma = 0.01$, and $\varepsilon = 0.01$. We first test the energy stability and mass conservation of the proposed schemes. Using a spatial resolution of $N = 128$ and various time steps, the results are reported in Figure \ref{fig:energy-mass-chns}. The first subplot of Figure \ref{fig:energy-mass-chns} shows the original energy evolution for $\nu = 0.01$, while the second subplot displays it for $\nu = 0.001$. The energy profiles obtained from different methods agree well, even with large time steps. The third subplot confirms that the total mass is preserved to machine accuracy in all cases.
\begin{figure}[htbp]
	\begin{center}
		\includegraphics[width=0.3\textwidth, height=100pt]{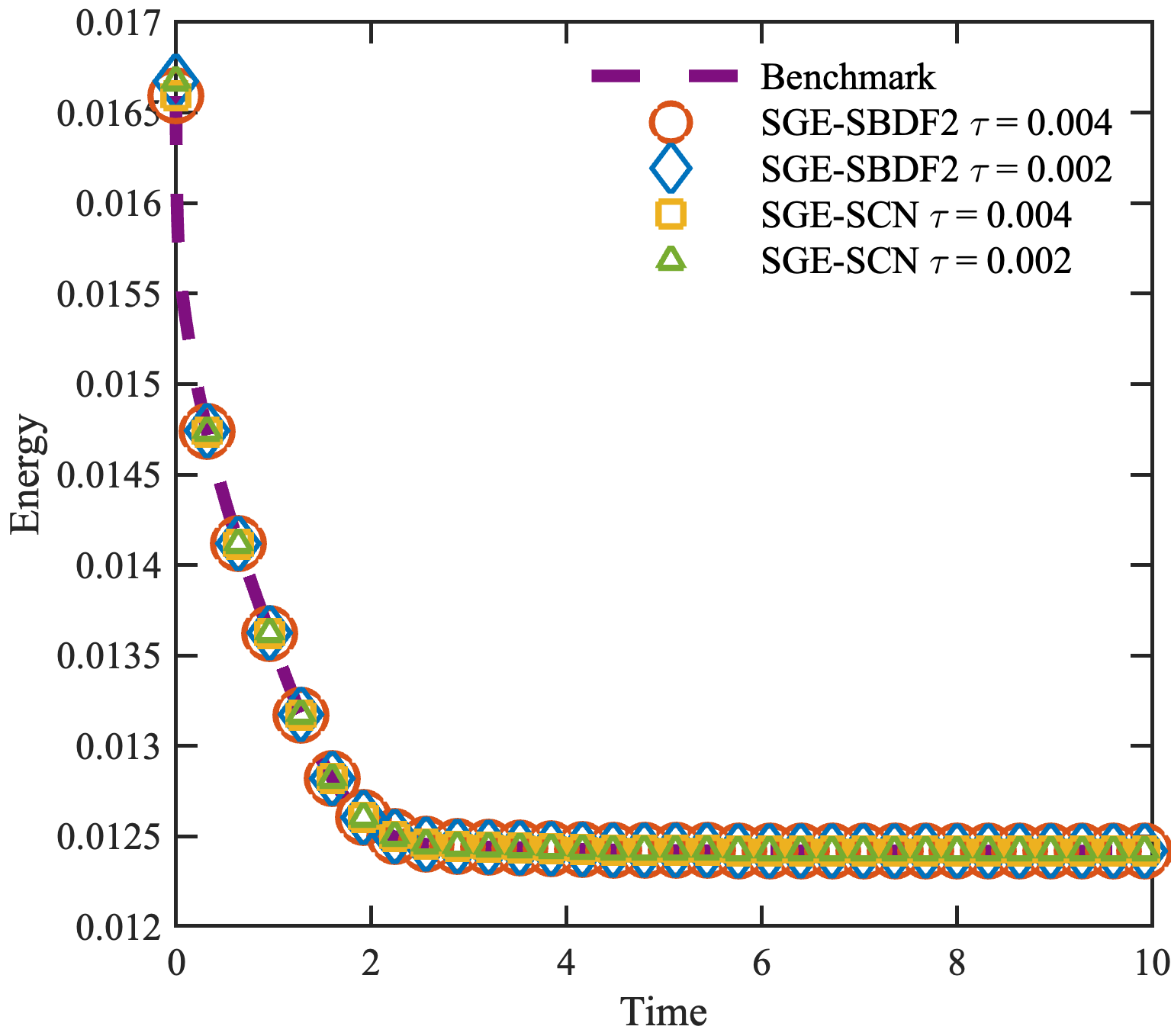}
		\includegraphics[width=0.3\textwidth, height=100pt]{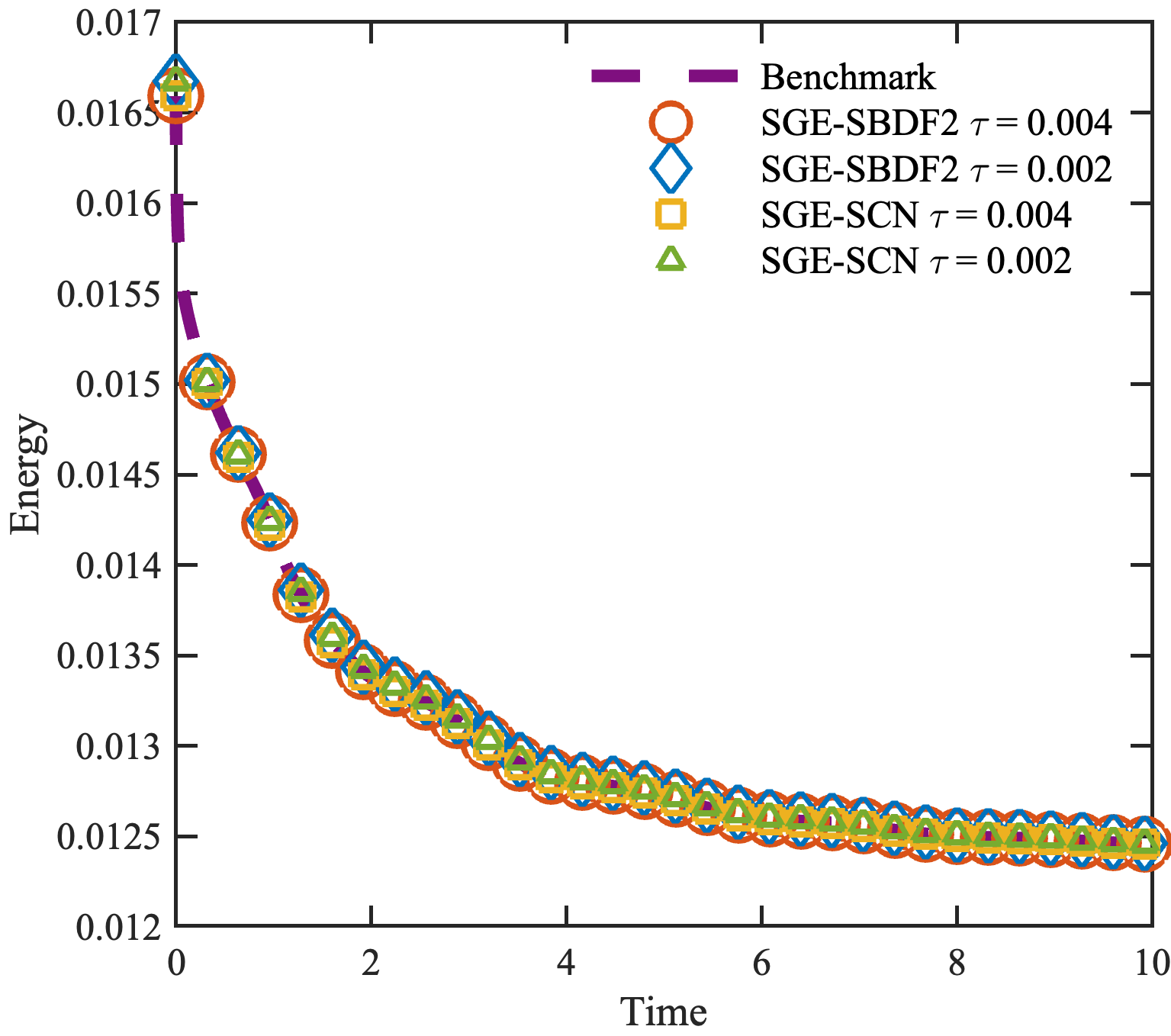}
		\includegraphics[width=0.3\textwidth, height=101pt]{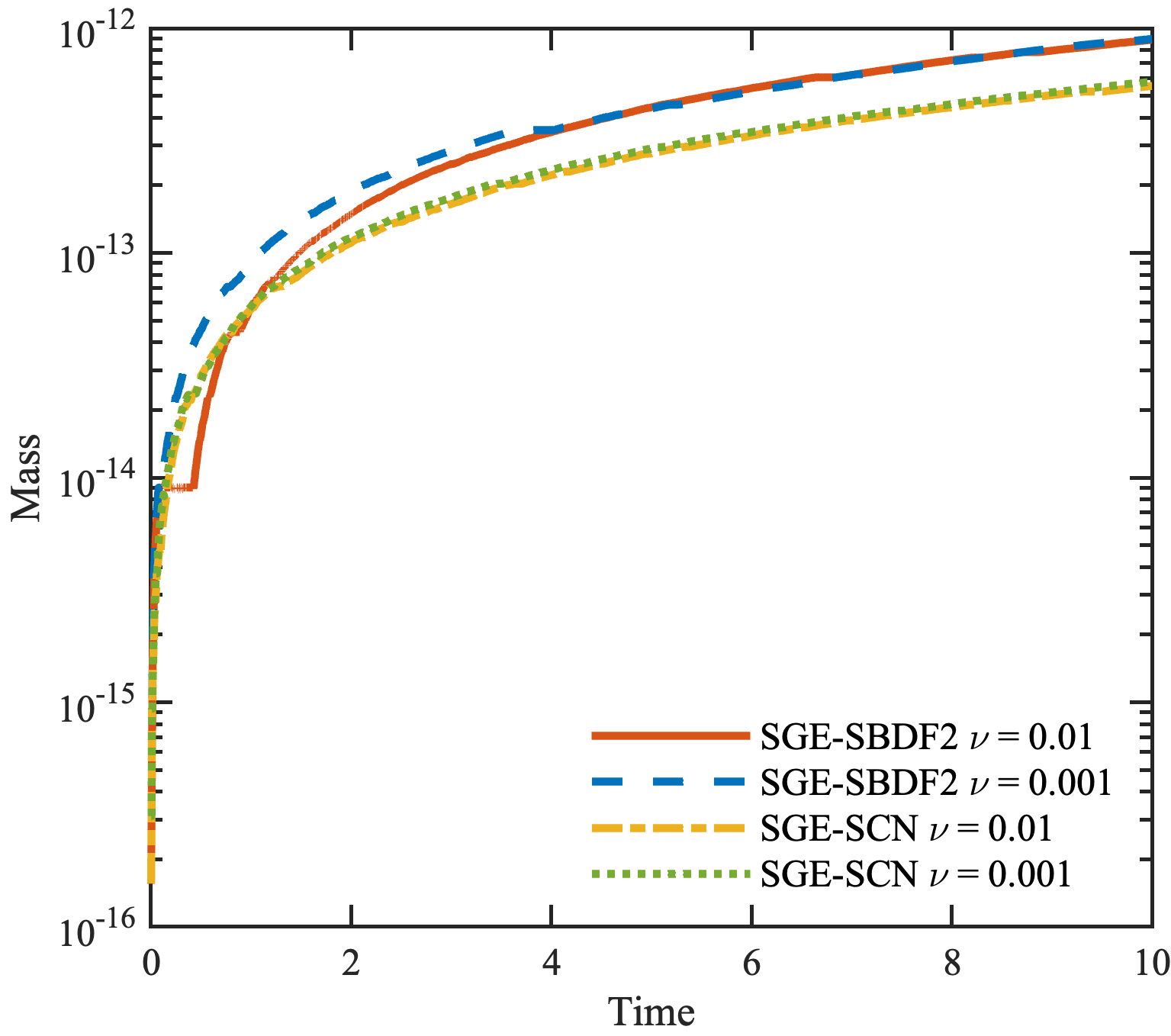}
	\end{center}
	\caption{Energy and mass evolutions for the CH-NS equations}\label{fig:energy-mass-chns}
\end{figure}
Figure \ref{fig:chns-merged-bubble-1} displays snapshots of the phase variable and velocity field at various time instances for $\nu = 0.01$, whereas Figure \ref{fig:chns-merged-bubble-2} presents the corresponding results for $\nu = 0.001$. In both cases, the two bubbles coalesce into a single one with the resulting bubble ultimately attaining a steady circular configuration. Nonetheless, the dynamical behaviors differ markedly between the two scenarios. Under a lower viscosity, inertia effects become significant. Specifically, as depicted in Figure \ref{fig:chns-merged-bubble-2}, the bubble undergoes successive compression and elongation, exhibiting oscillatory damping driven by the interplay between kinetic energy and Helmholtz free energy, before settling into a round bubble. This observation underscores that when inertial effects are prominent, the underlying hydrodynamics as described by the NS equation cannot be neglected.

\begin{figure}[htbp]
	\begin{center}
		\includegraphics[width=0.18\textwidth]{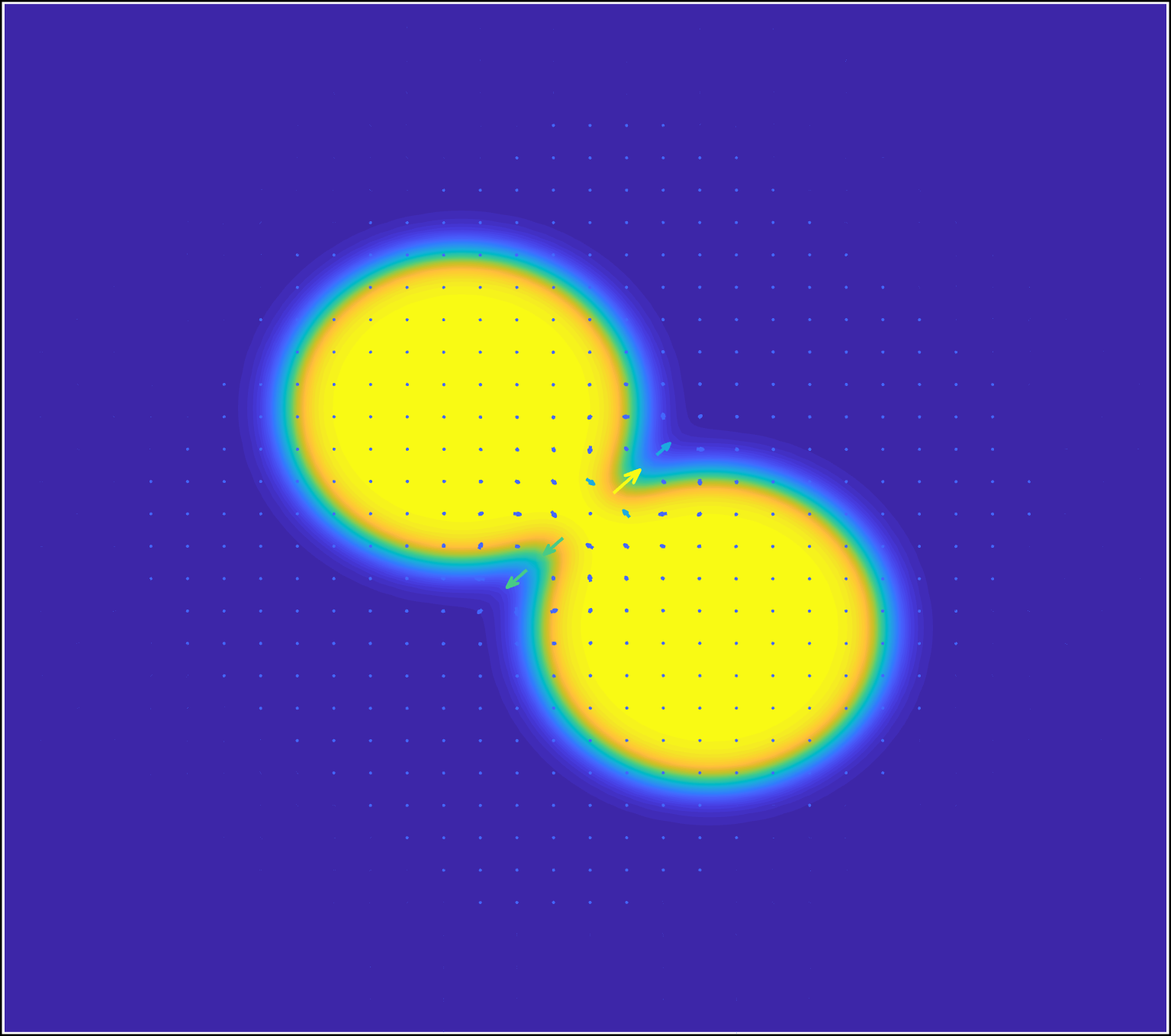}
		\includegraphics[width=0.18\textwidth]{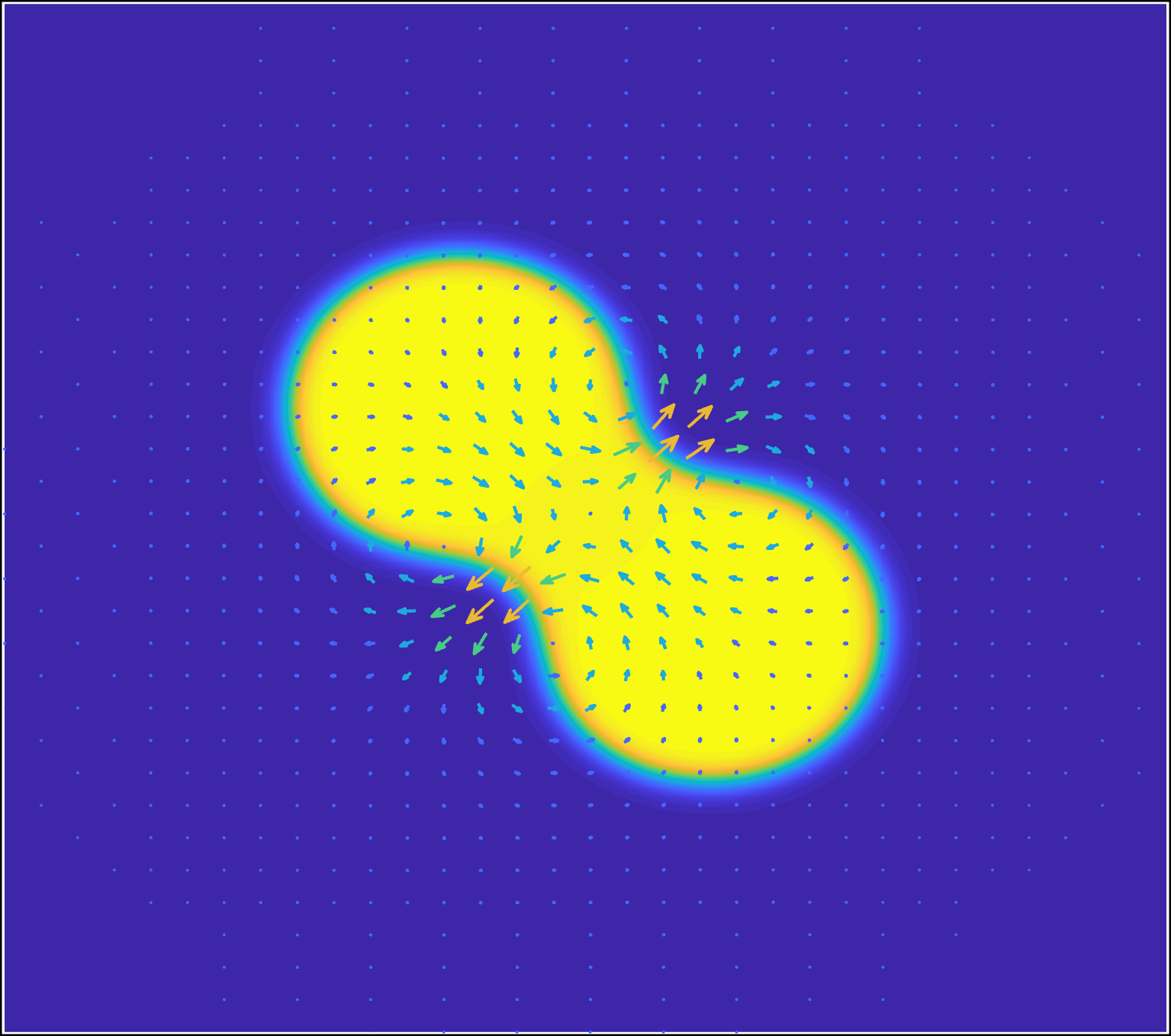}
		\includegraphics[width=0.18\textwidth]{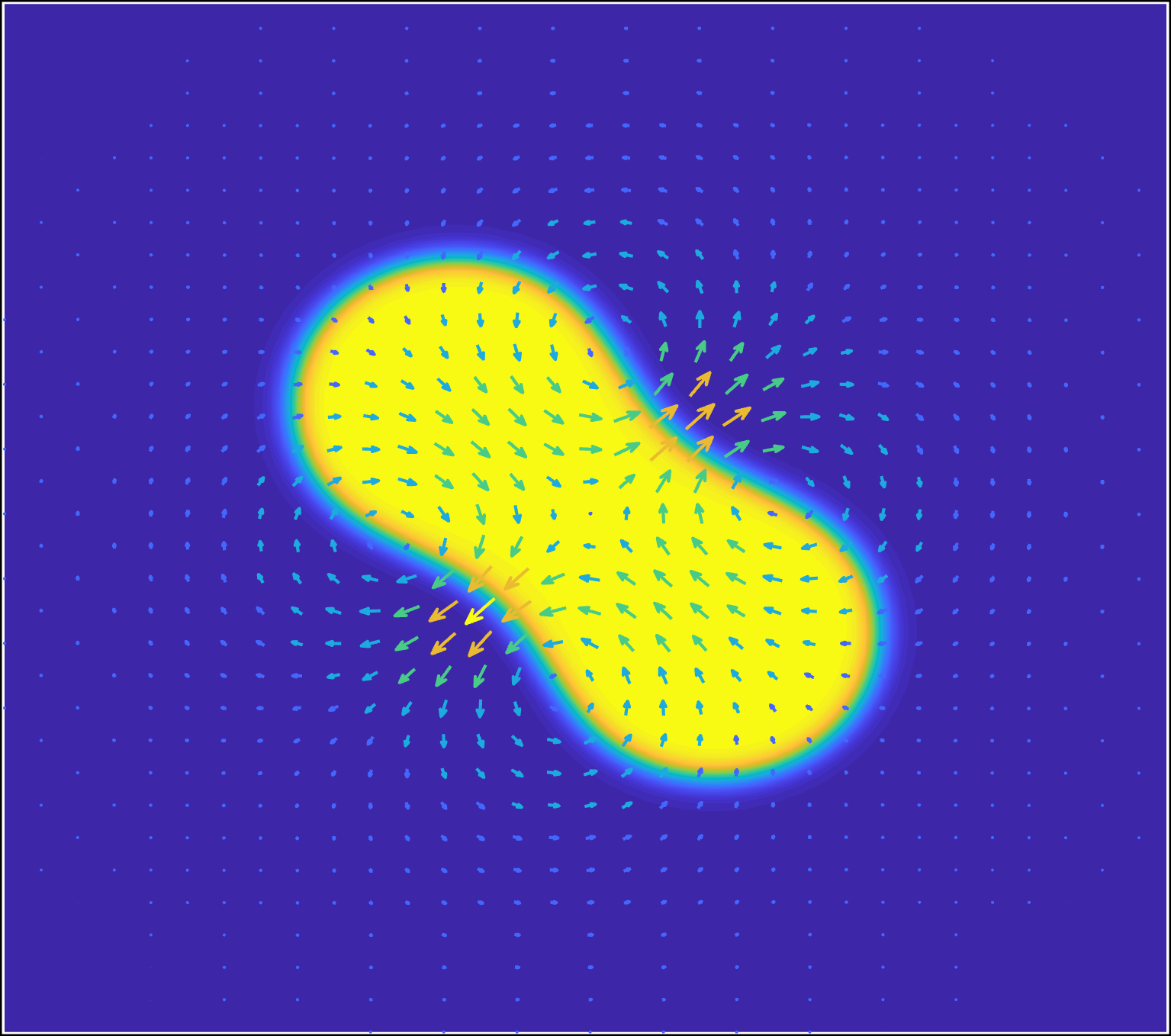}
		\includegraphics[width=0.18\textwidth]{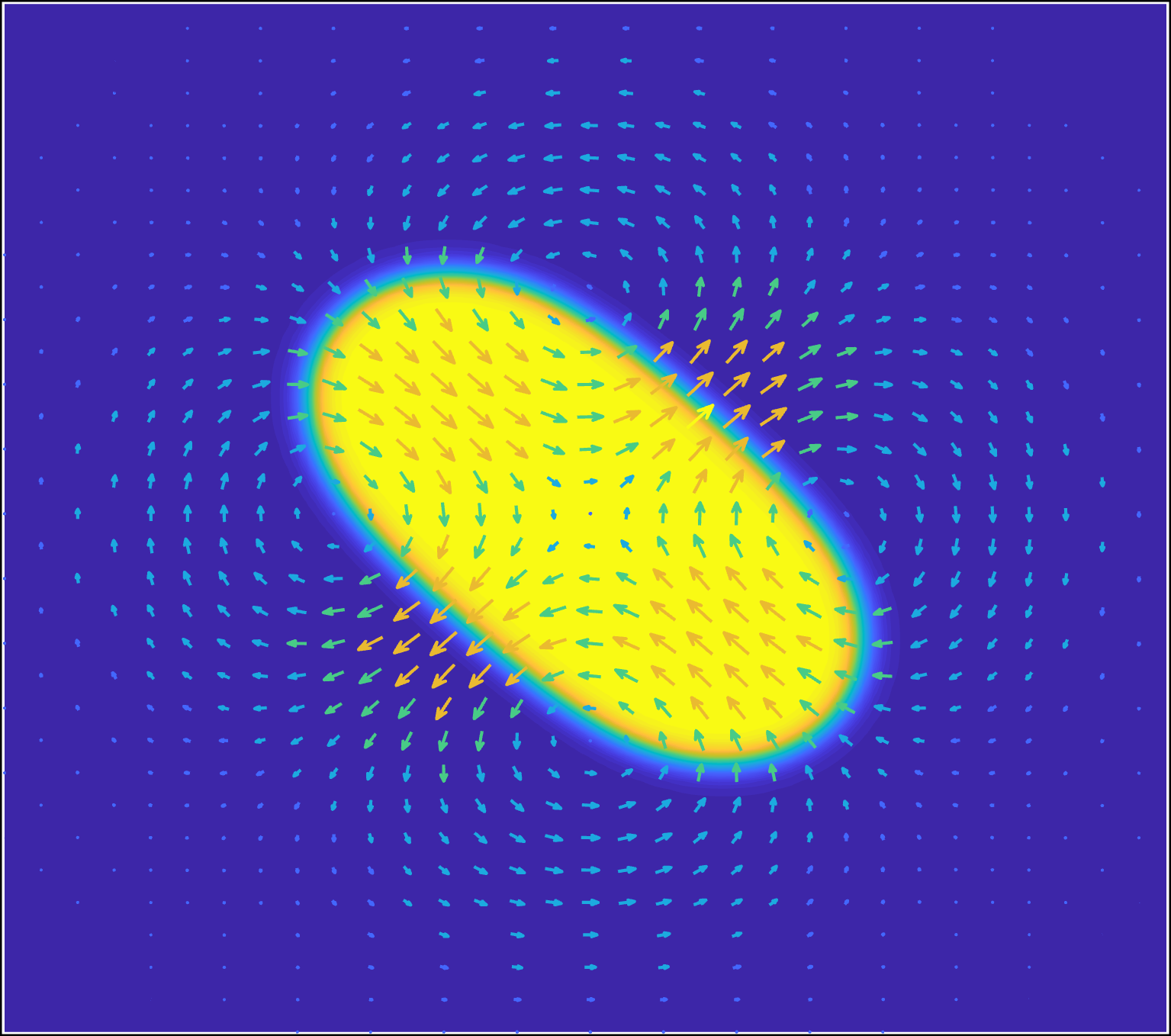}
		\includegraphics[width=0.18\textwidth]{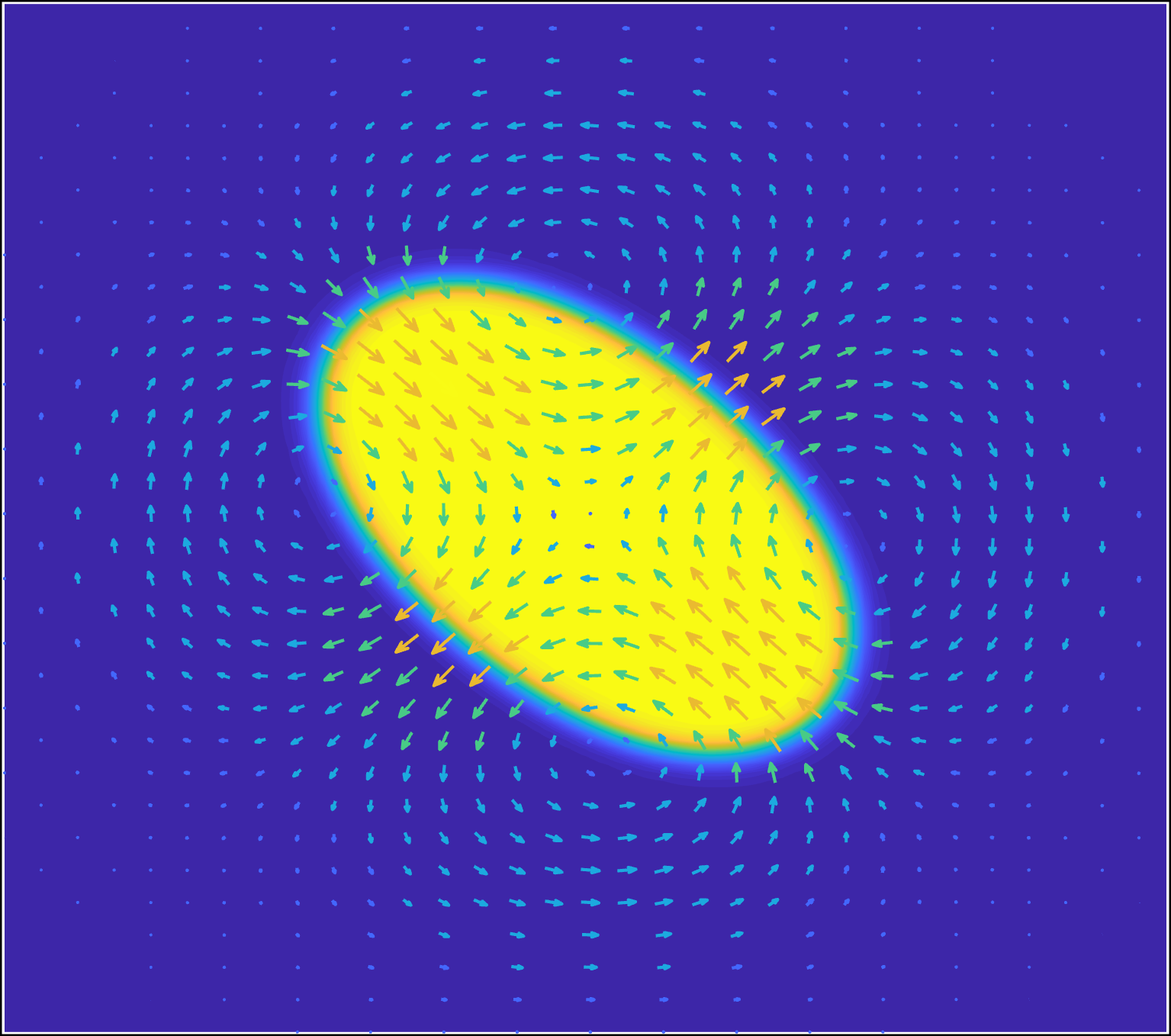}

		\includegraphics[width=0.18\textwidth]{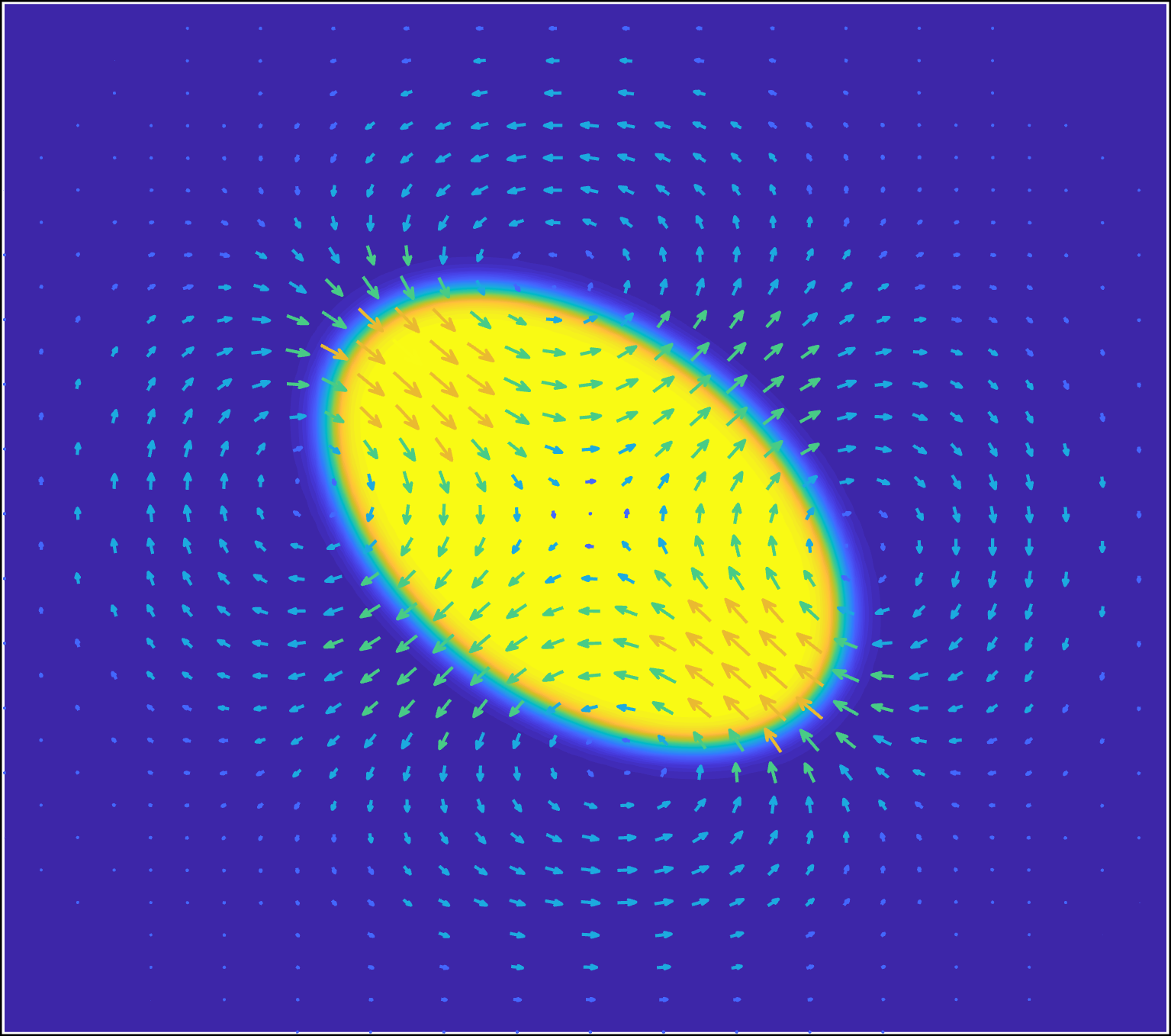}
		\includegraphics[width=0.18\textwidth]{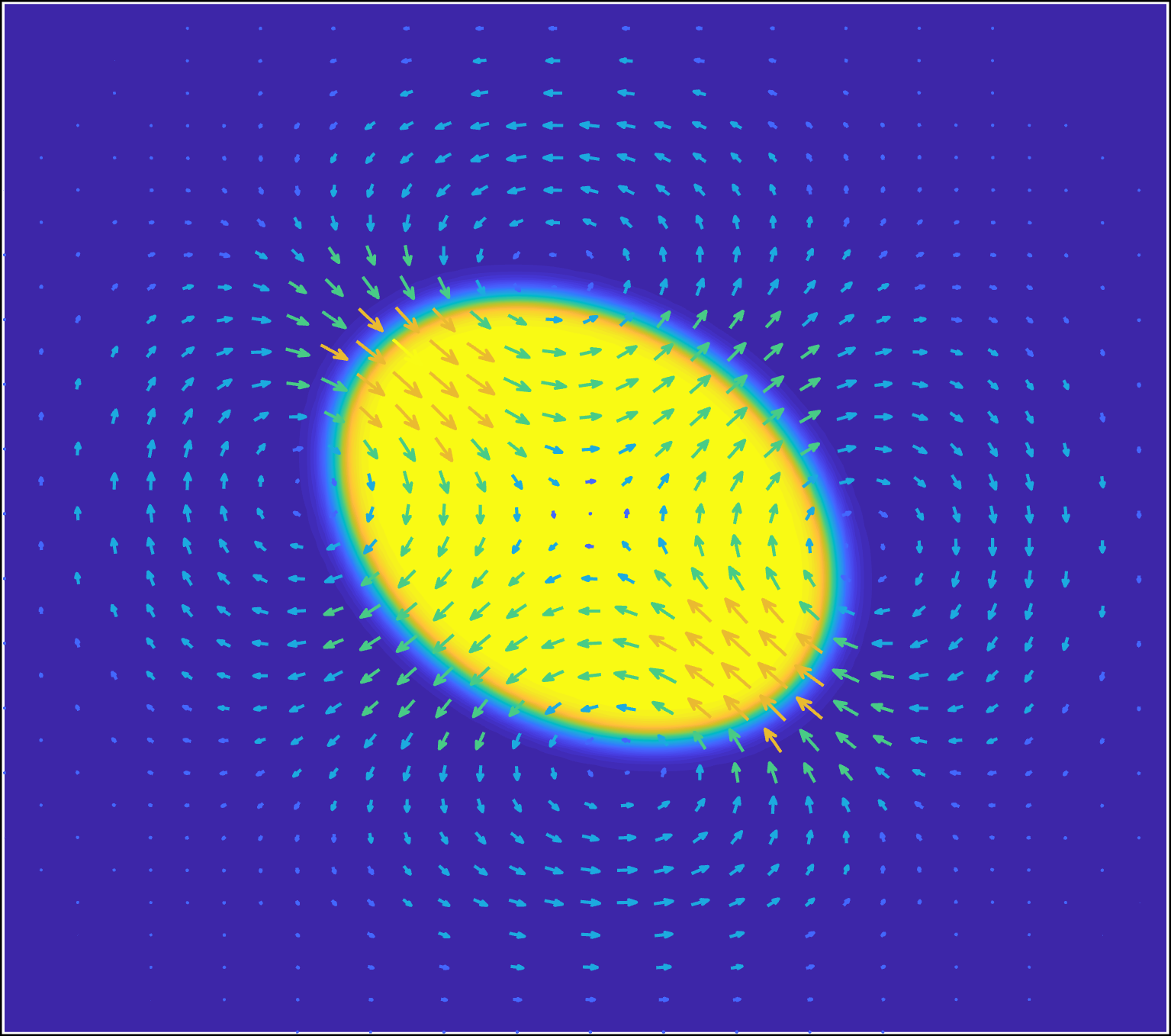}
		\includegraphics[width=0.18\textwidth]{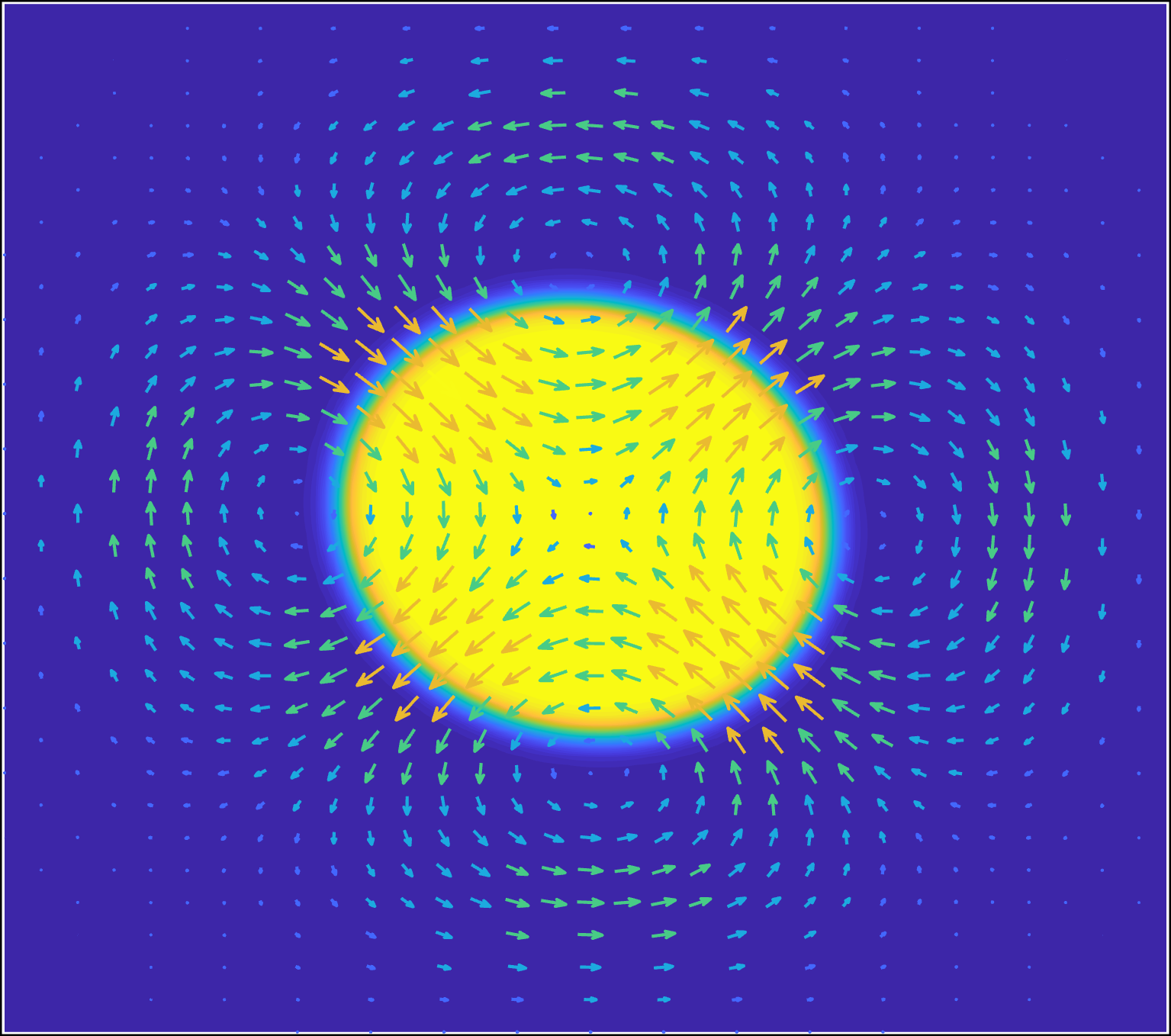}
		\includegraphics[width=0.18\textwidth]{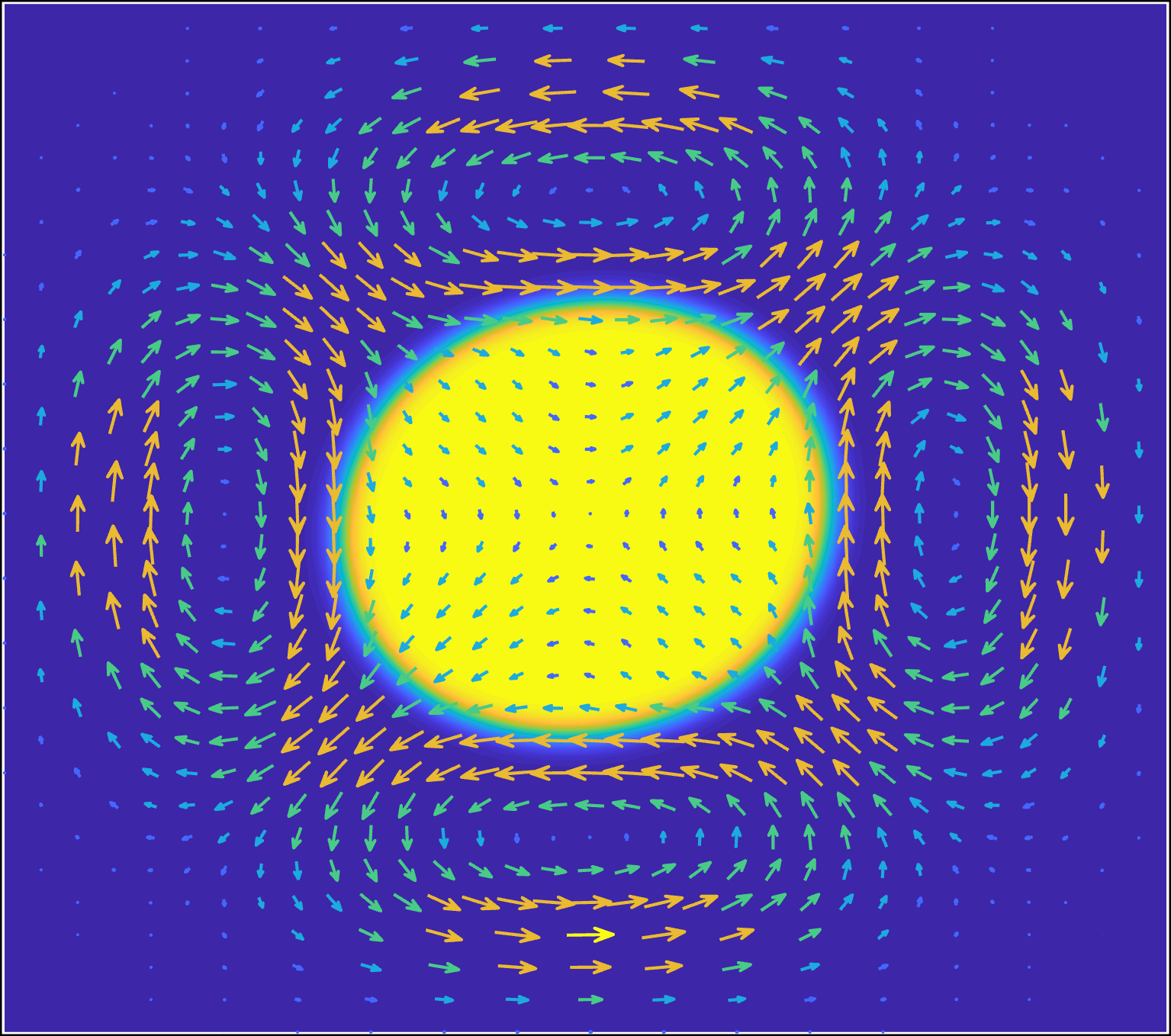}
		\includegraphics[width=0.18\textwidth]{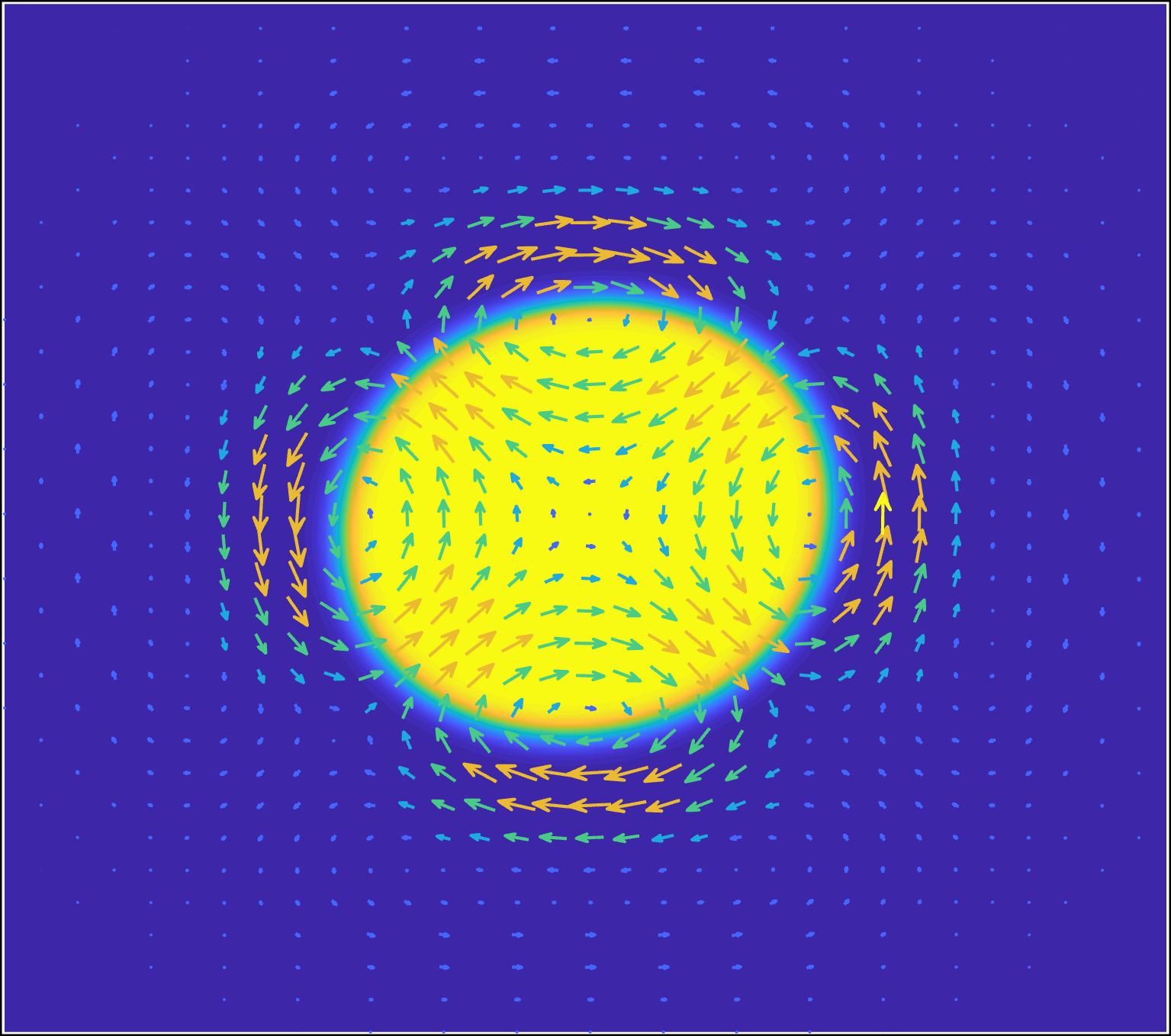}

		\includegraphics[width=0.18\textwidth]{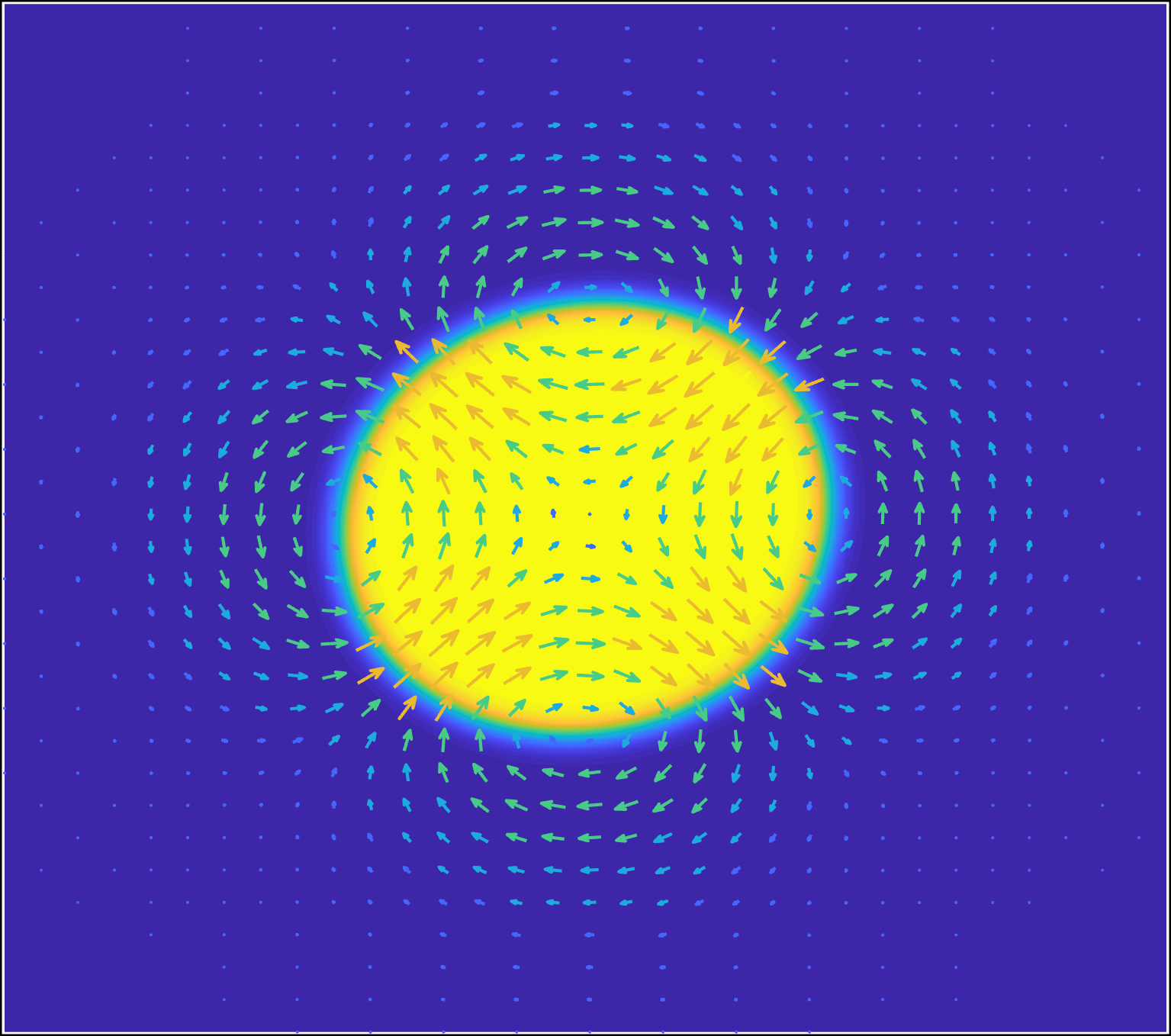}
		\includegraphics[width=0.18\textwidth]{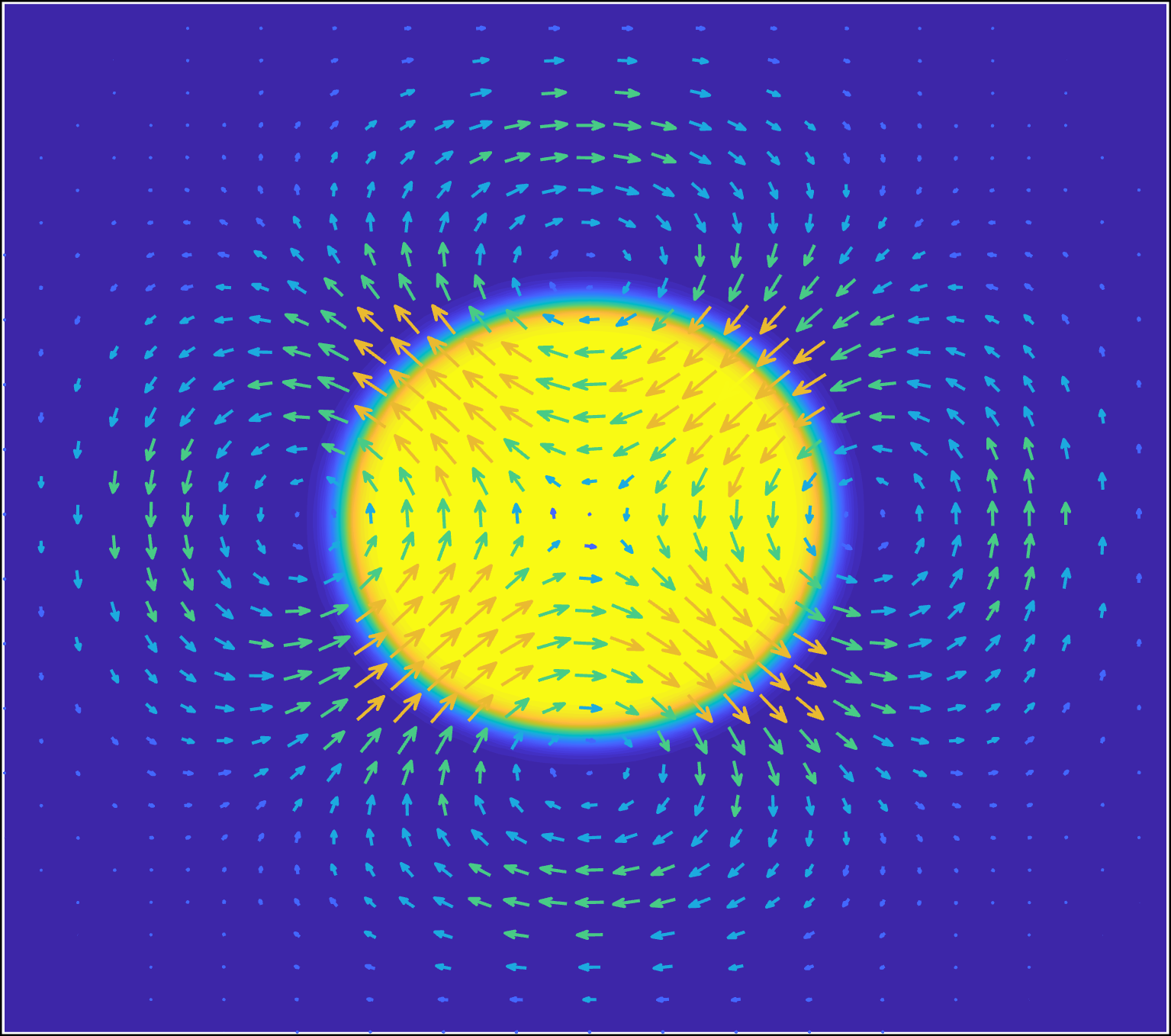}
		\includegraphics[width=0.18\textwidth]{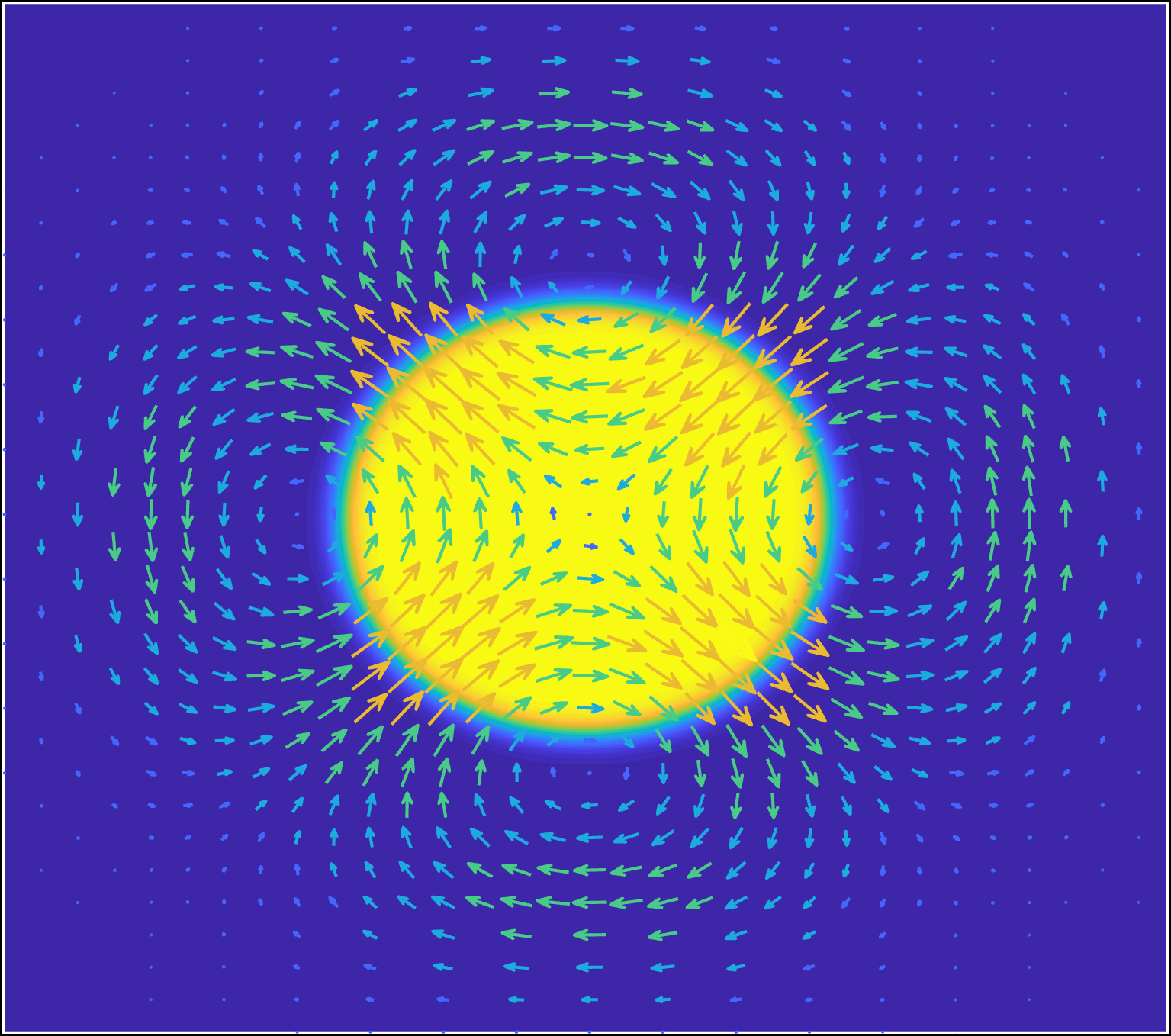}
		\includegraphics[width=0.18\textwidth]{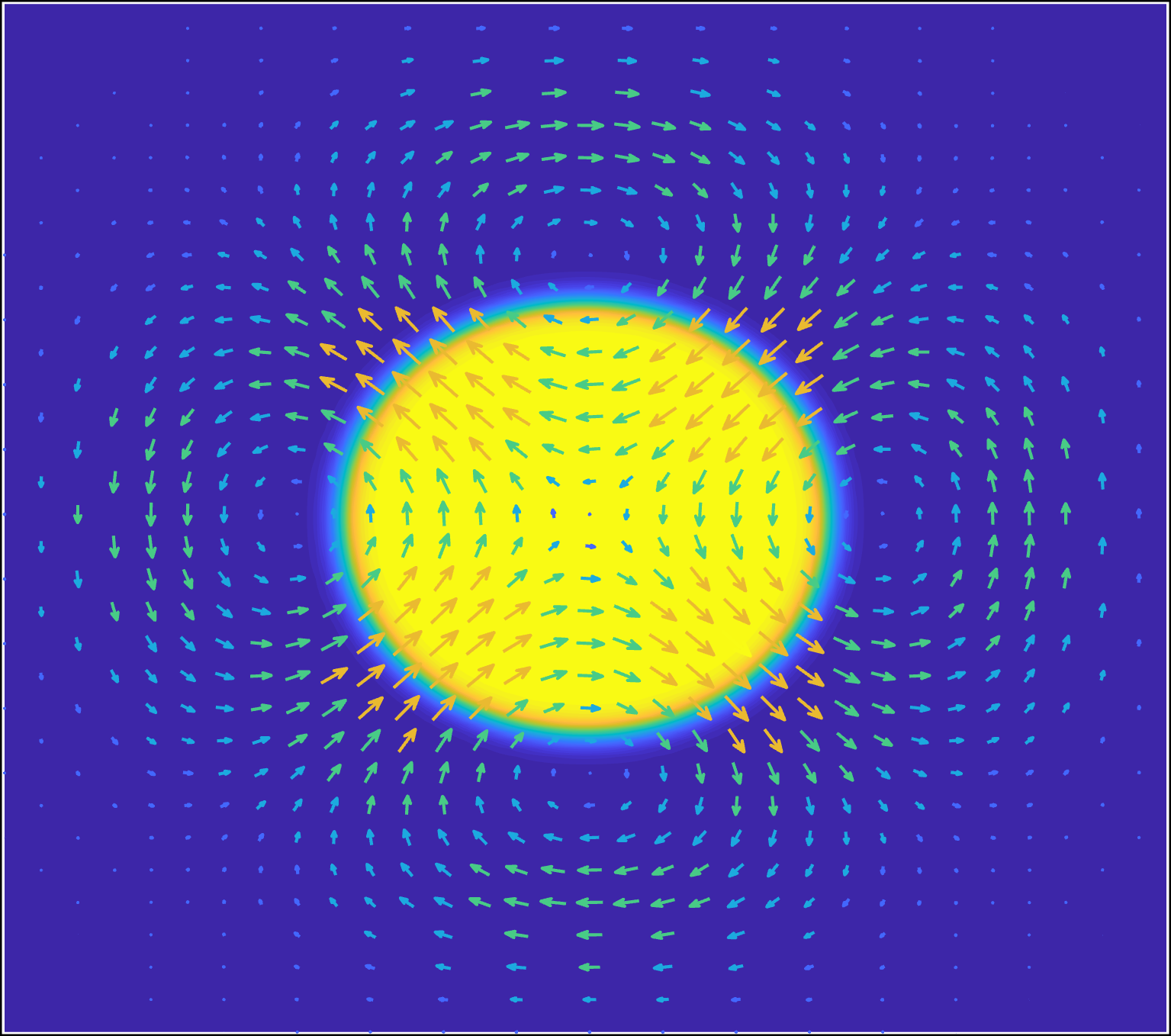}
		\includegraphics[width=0.18\textwidth]{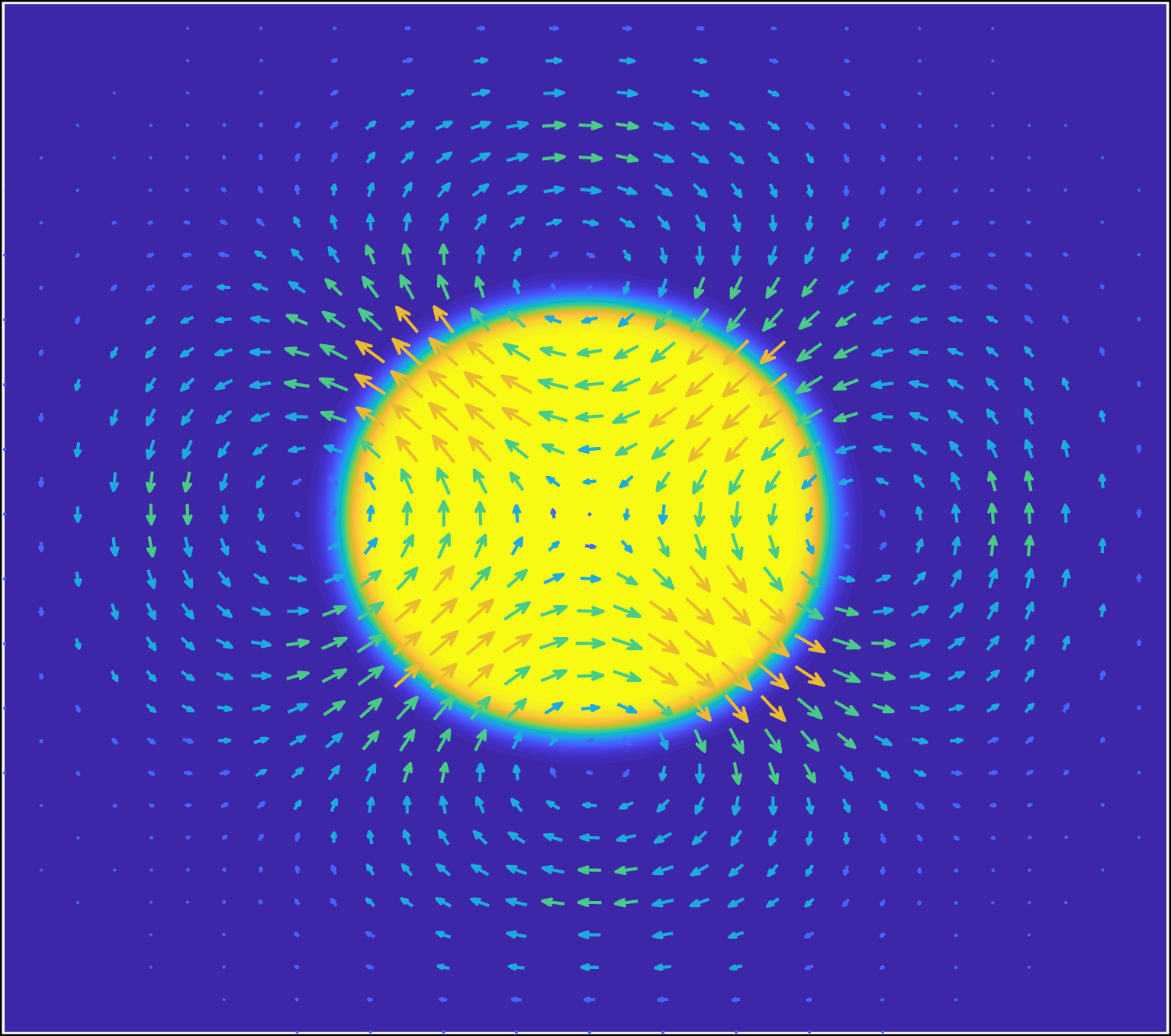}
	\end{center}
	\caption{Snapshots of the phase field $\phi$ and the corresponding velocity for the merged bubble with $\nu=0.01$ at selected time instances. The top row shows the evolution at $t = 0.001, 0.2, 0.4, 0.8, 1.0$; the middle row at $t = 1.2, 1.5, 2.0, 3.0, 3.2$; and the bottom row at $t = 3.4, 4.6, 4.8, 5.0, 10.0$.}\label{fig:chns-merged-bubble-1}
\end{figure}

\begin{figure}[htbp]
	\begin{center}
		\includegraphics[width=0.18\textwidth]{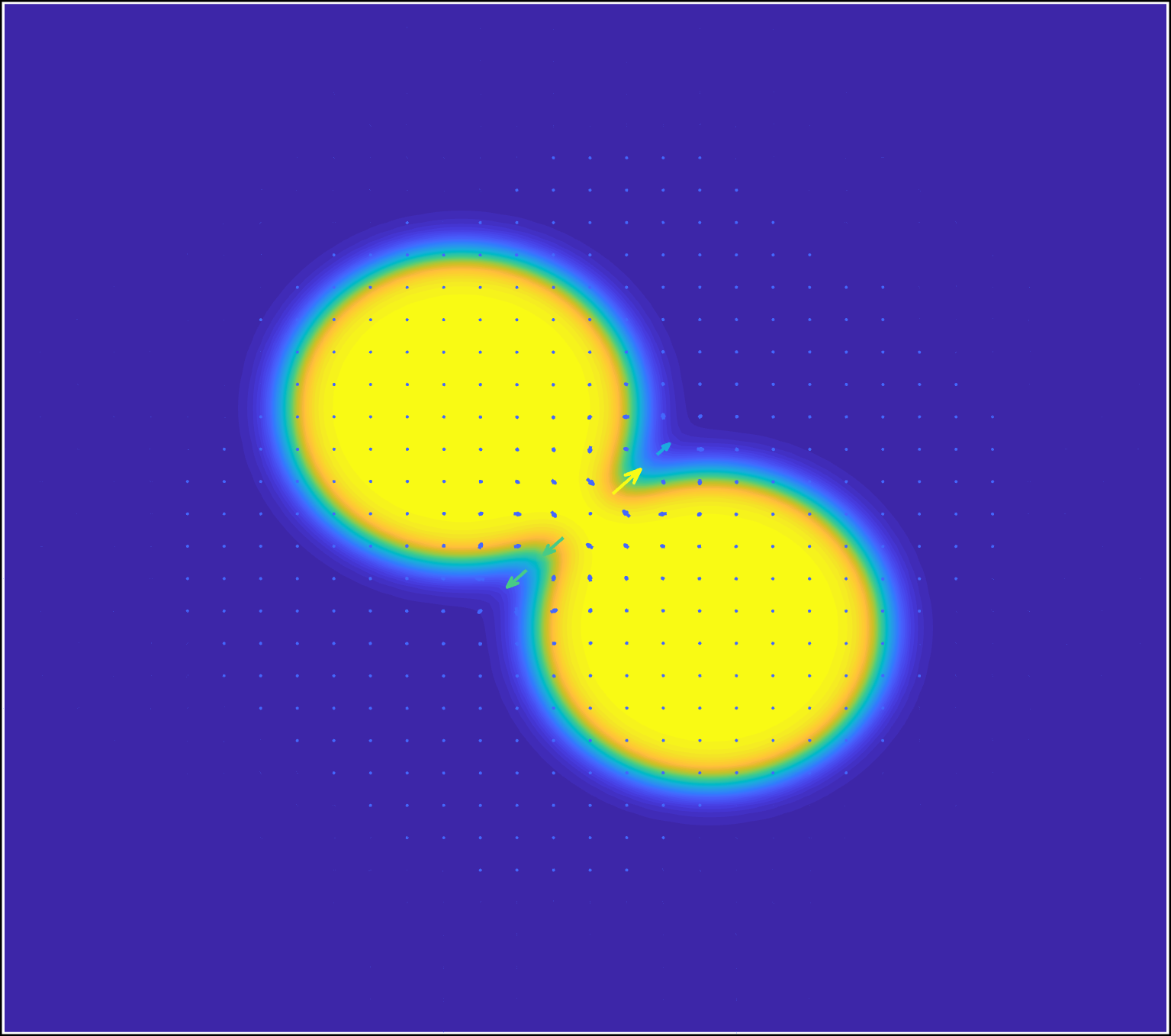}
		\includegraphics[width=0.18\textwidth]{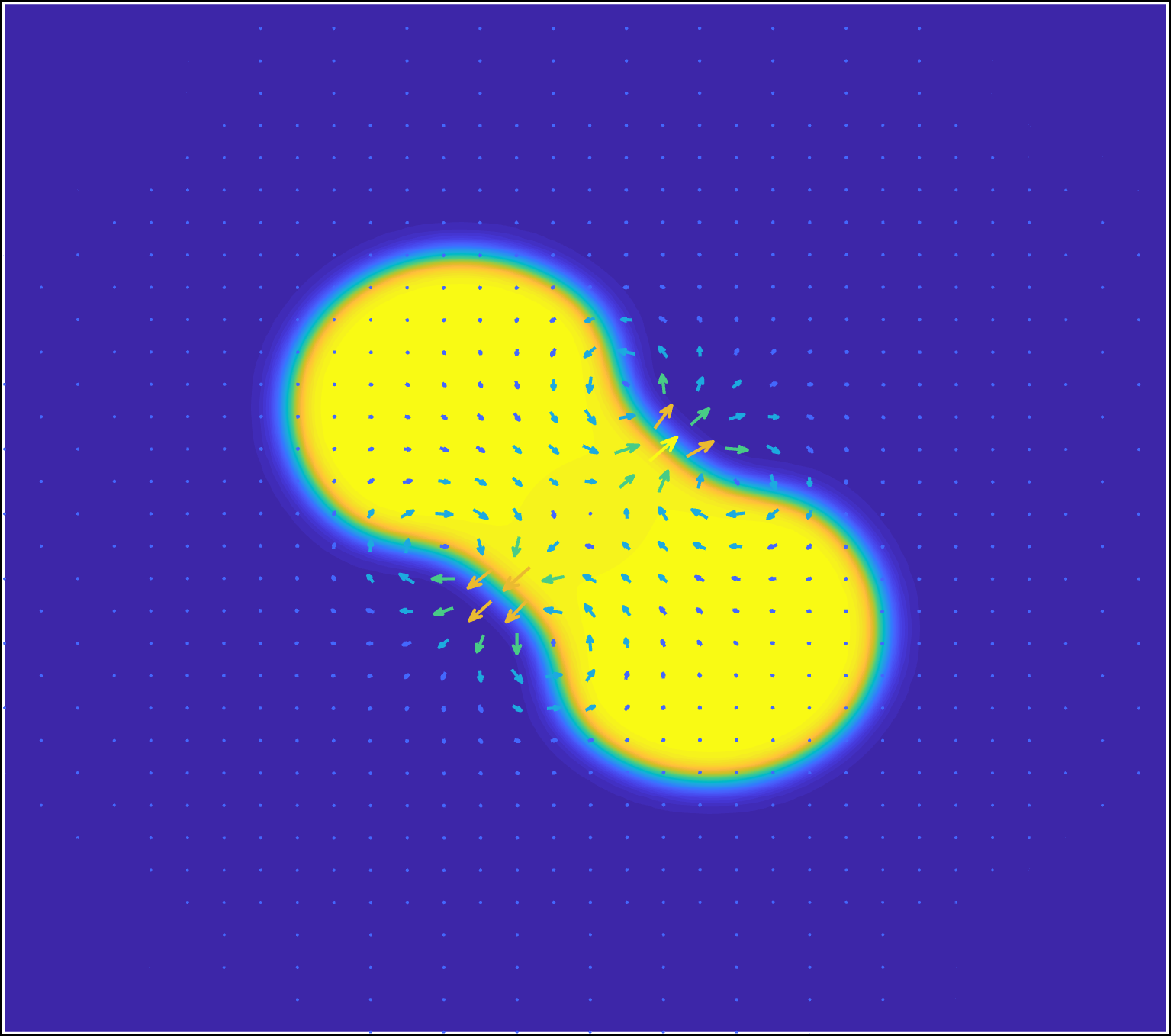}
		\includegraphics[width=0.18\textwidth]{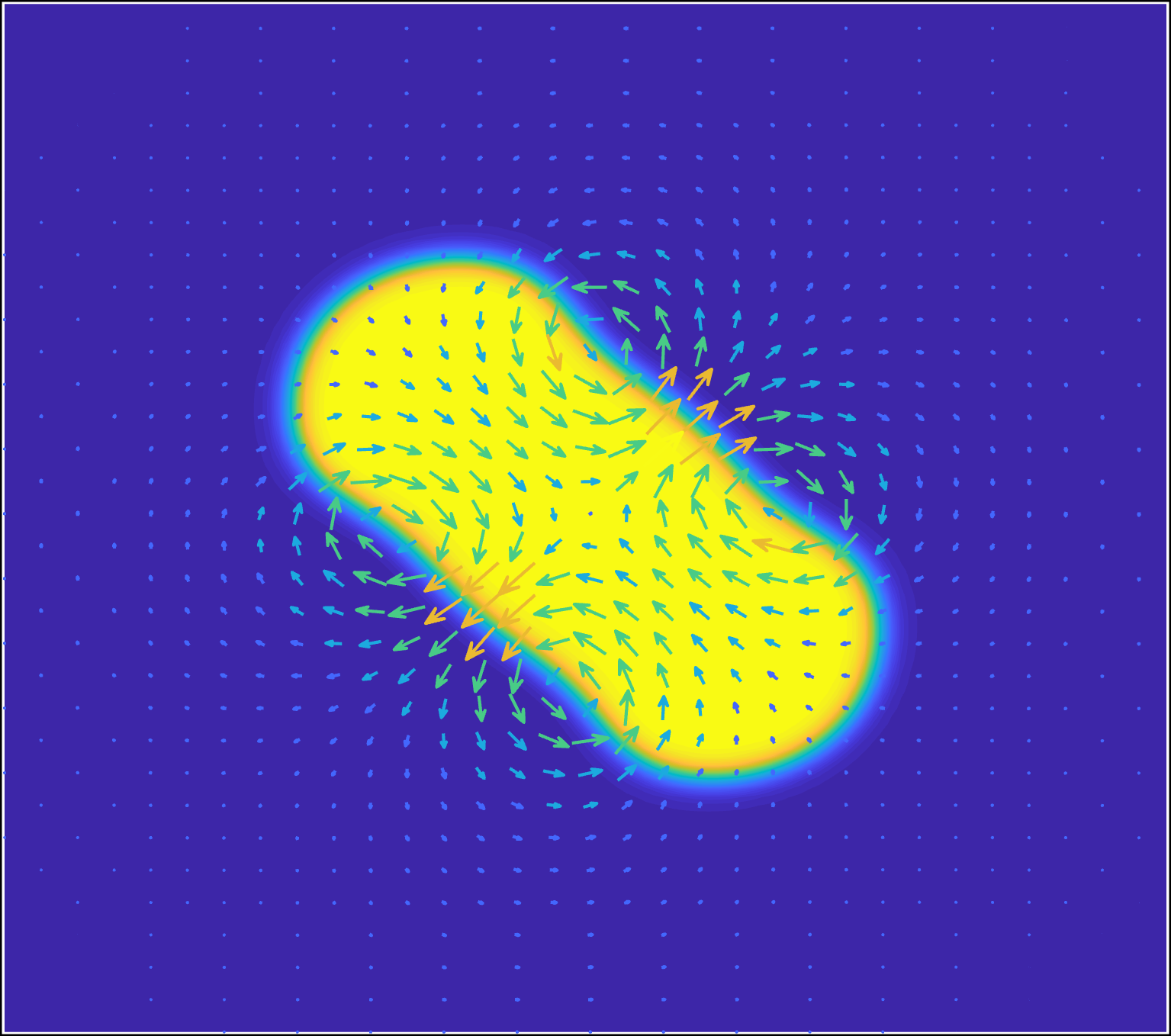}
		\includegraphics[width=0.18\textwidth]{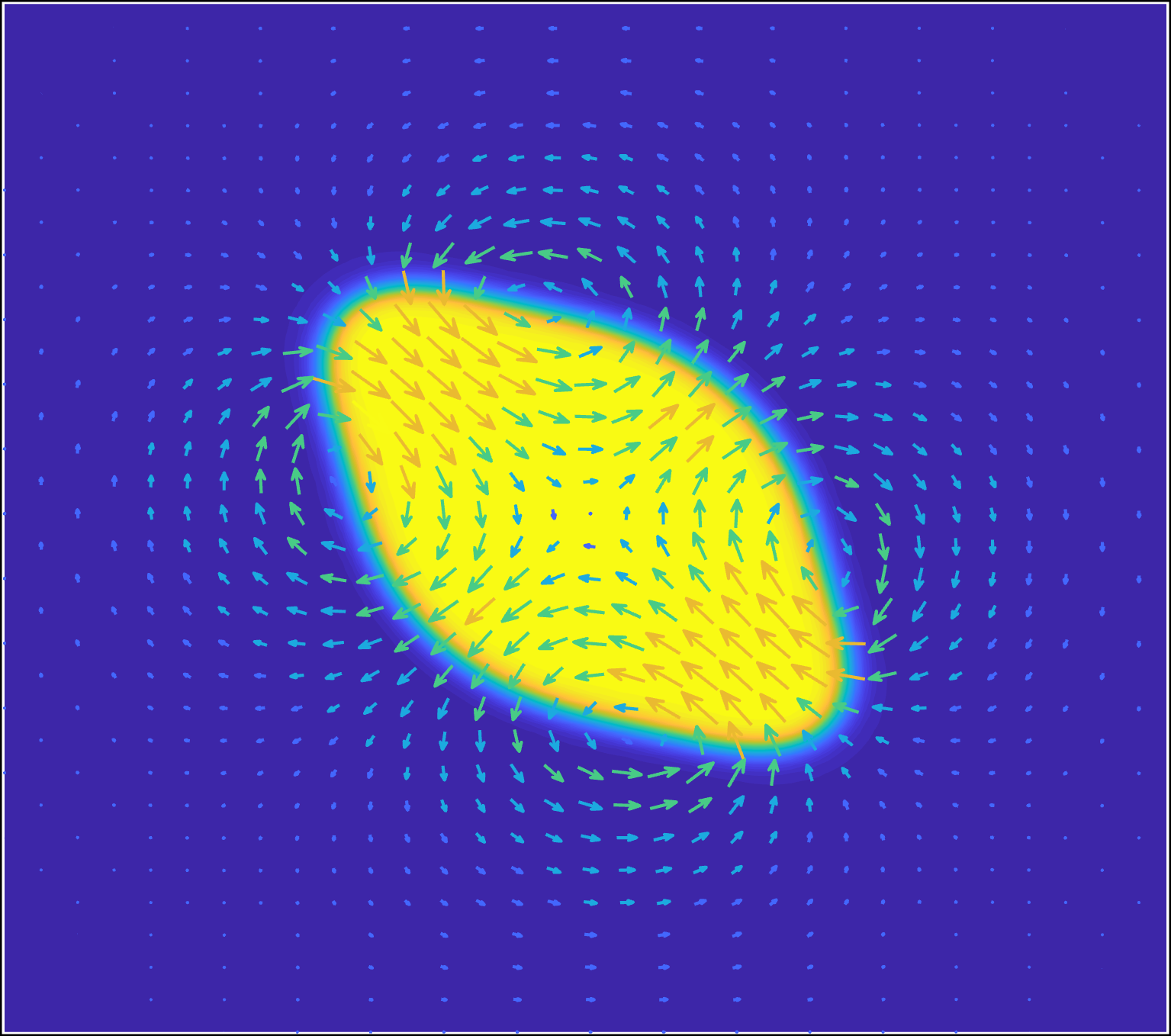}
		\includegraphics[width=0.18\textwidth]{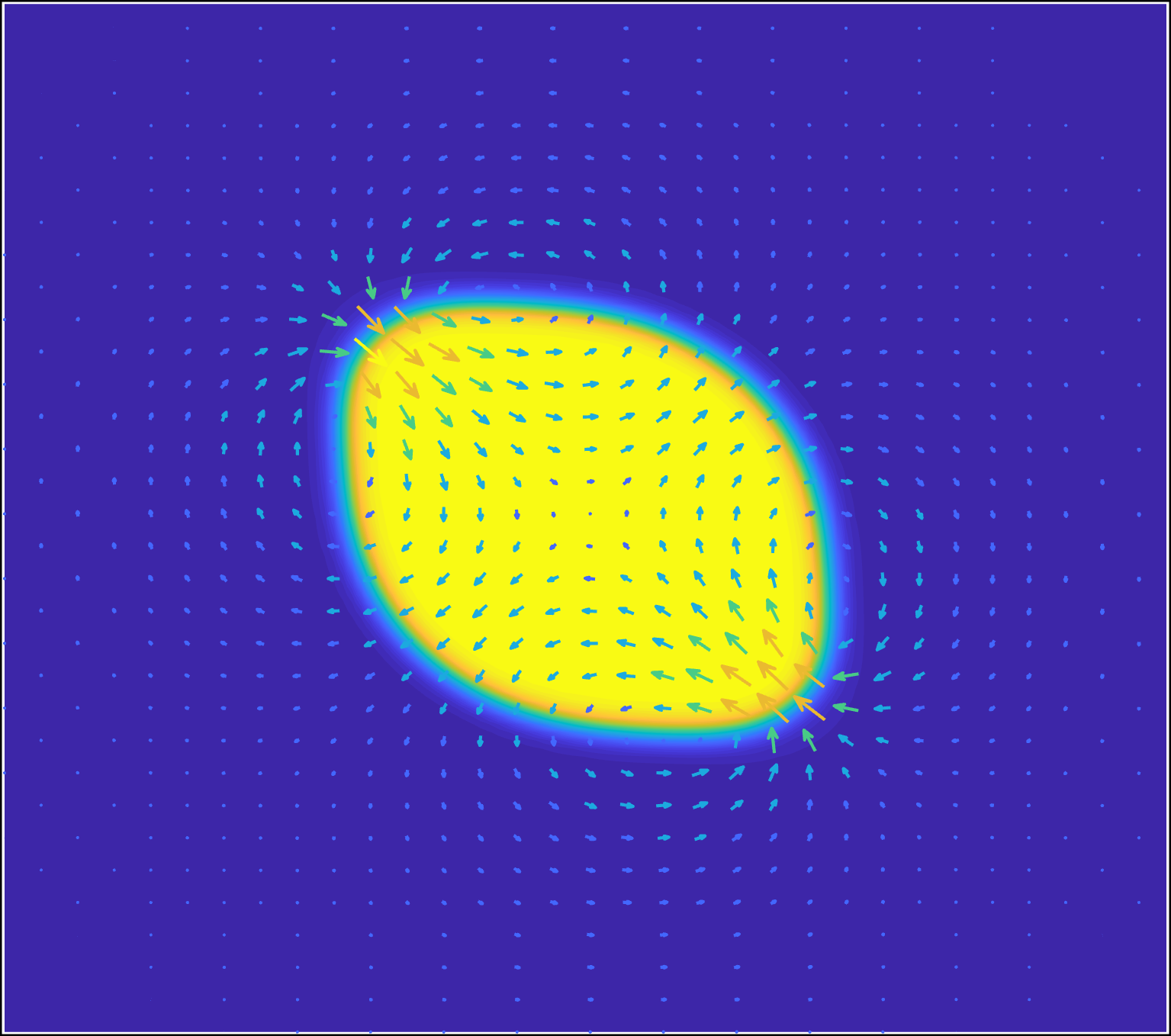}

		\includegraphics[width=0.18\textwidth]{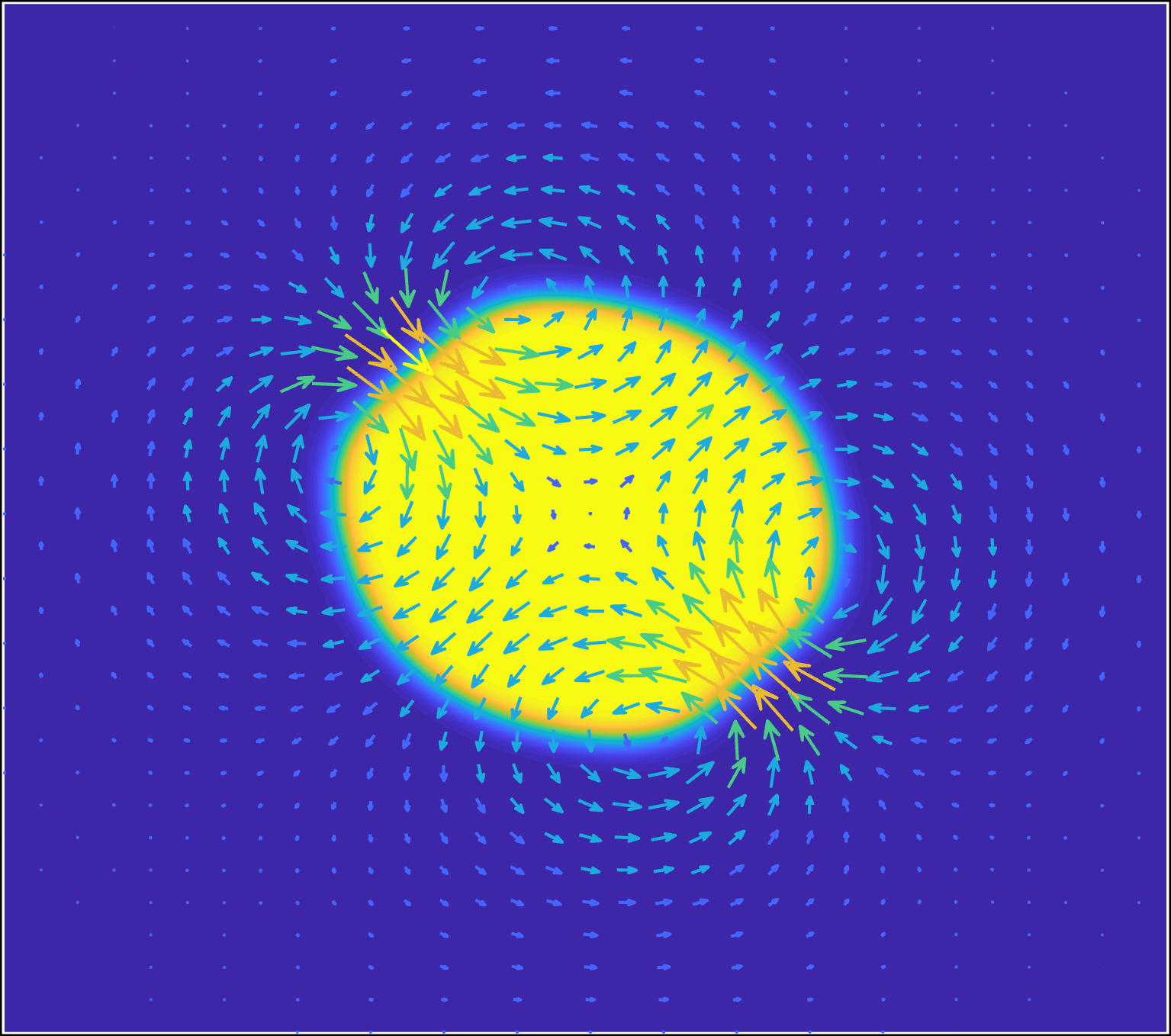}
		\includegraphics[width=0.18\textwidth]{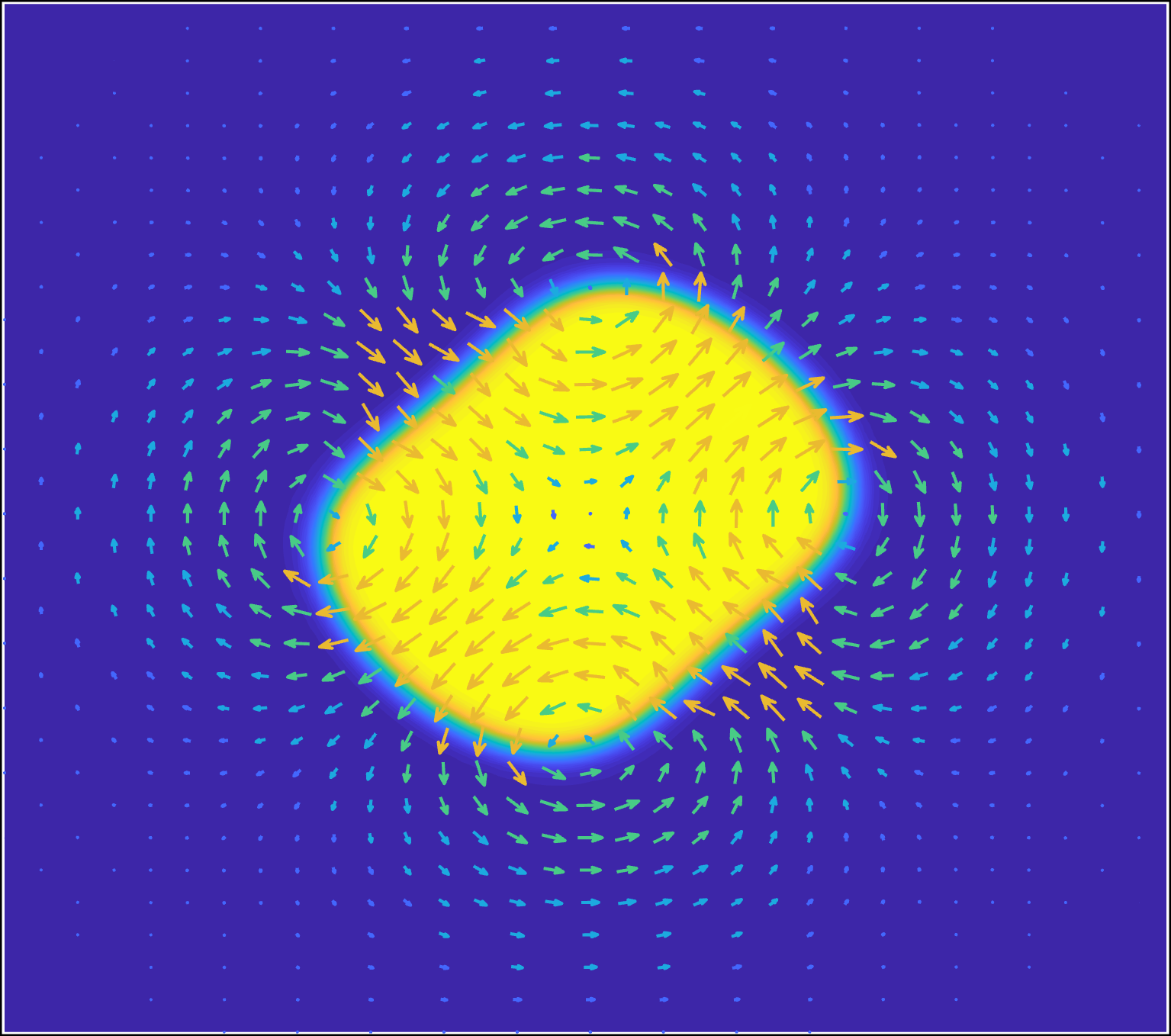}
		\includegraphics[width=0.18\textwidth]{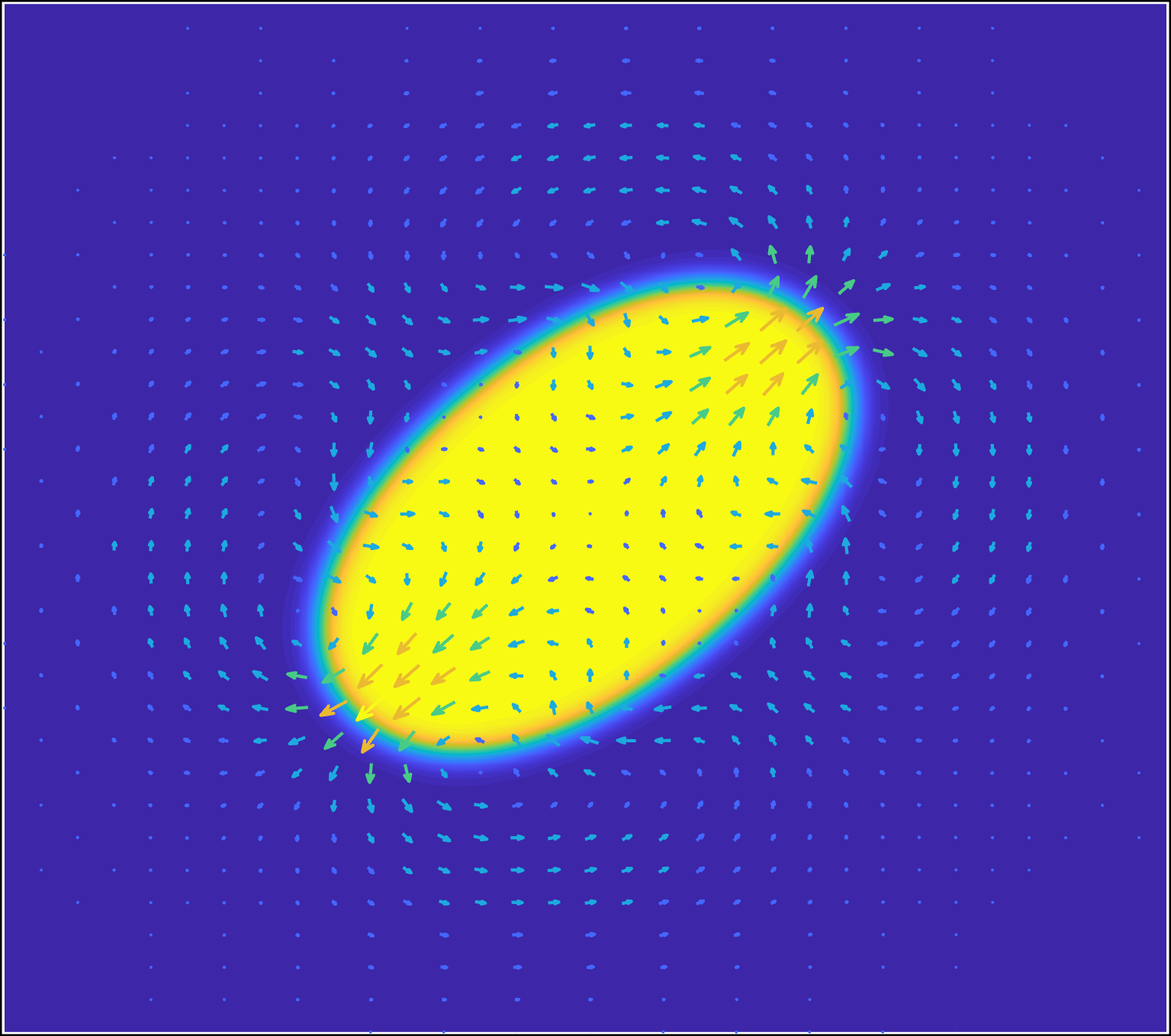}
		\includegraphics[width=0.18\textwidth]{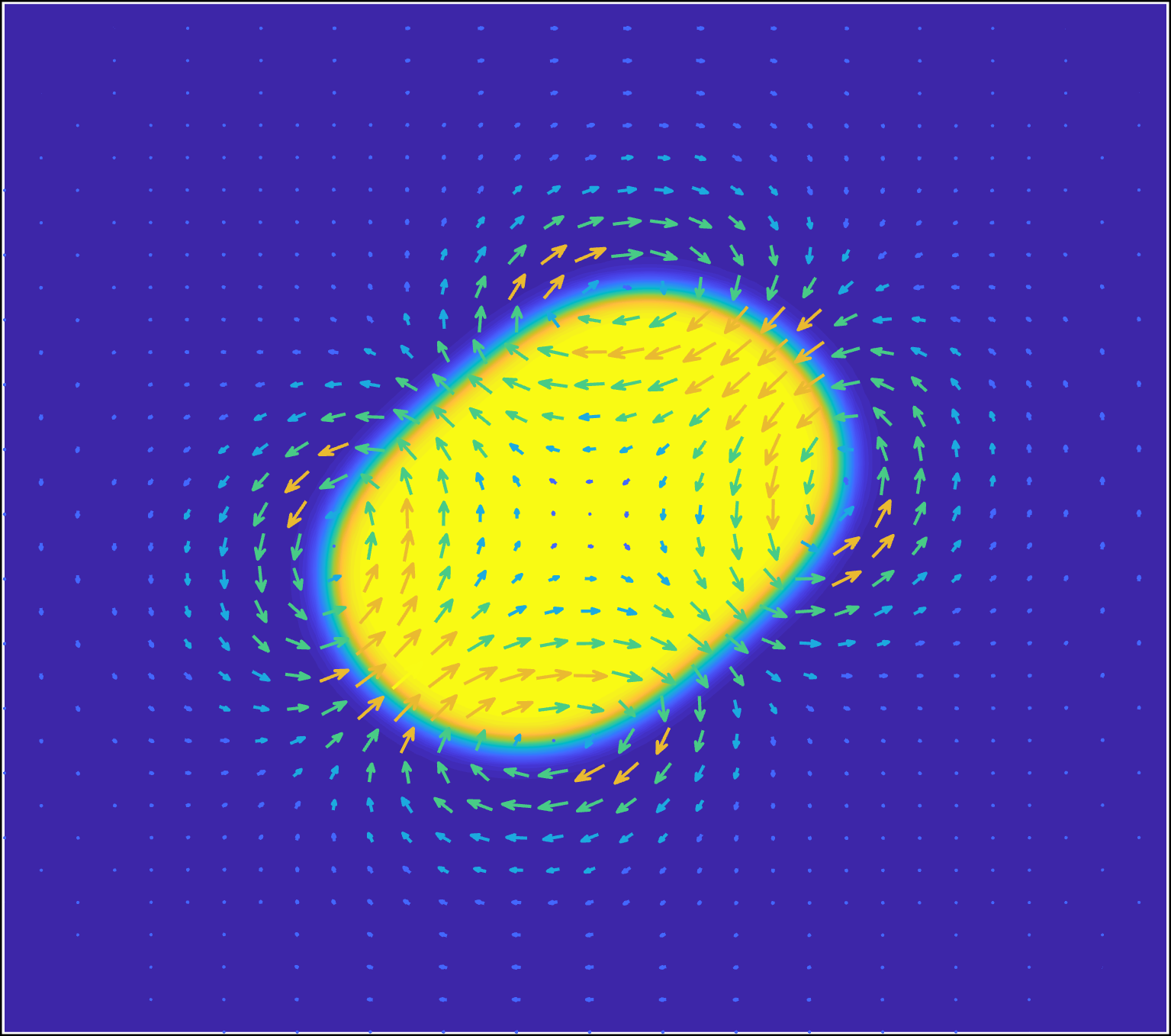}
		\includegraphics[width=0.18\textwidth]{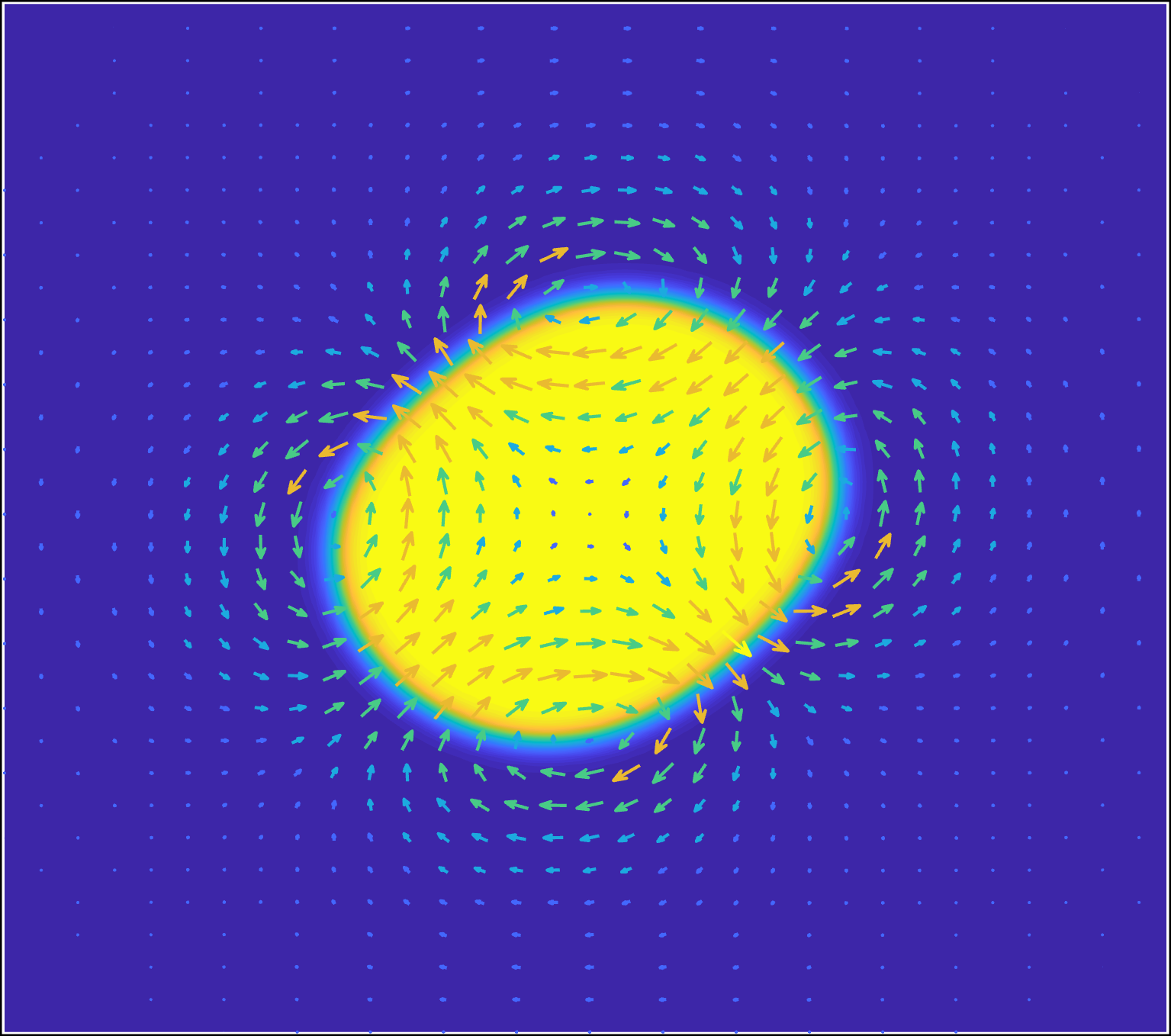}

		\includegraphics[width=0.18\textwidth]{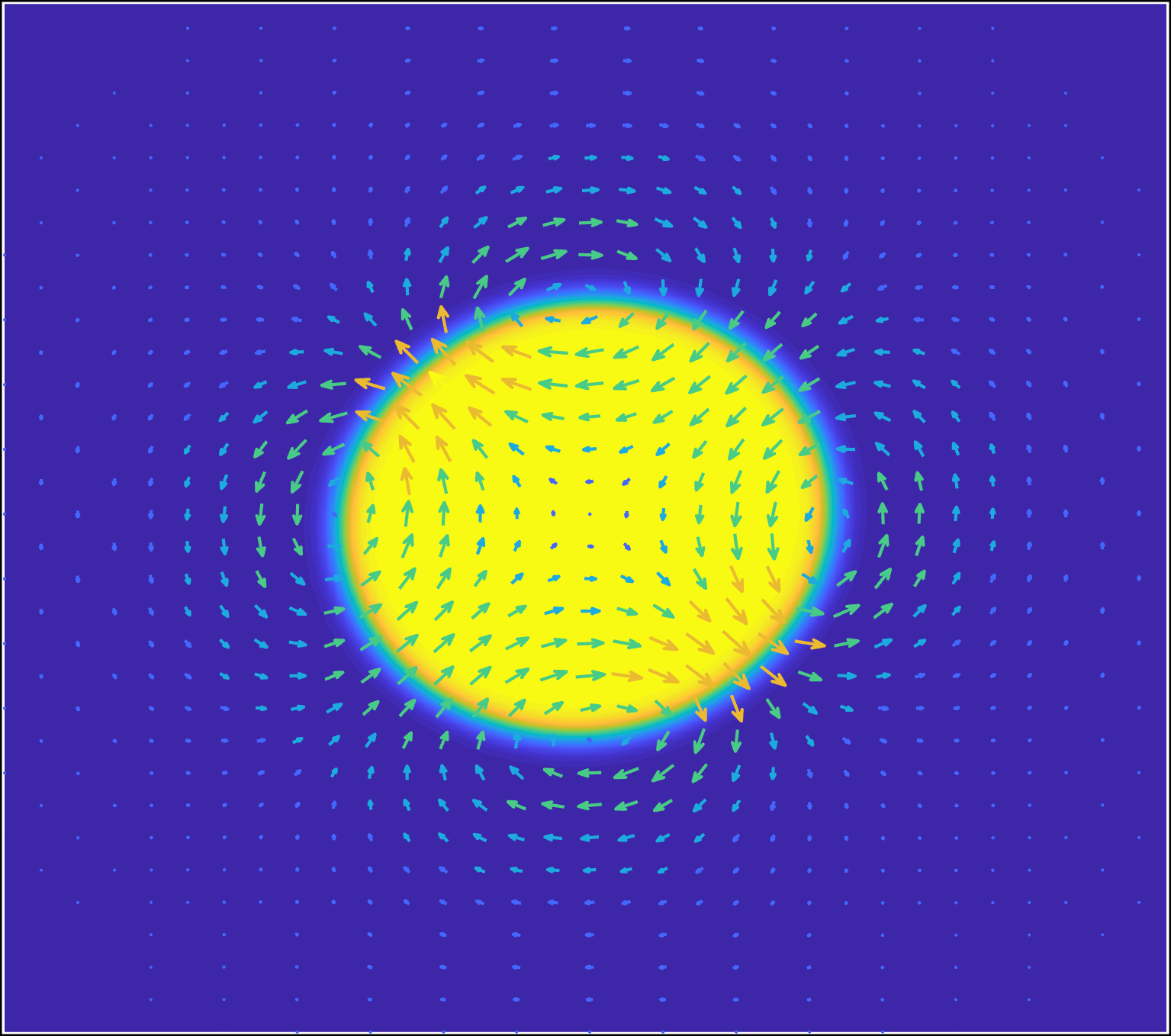}
		\includegraphics[width=0.18\textwidth]{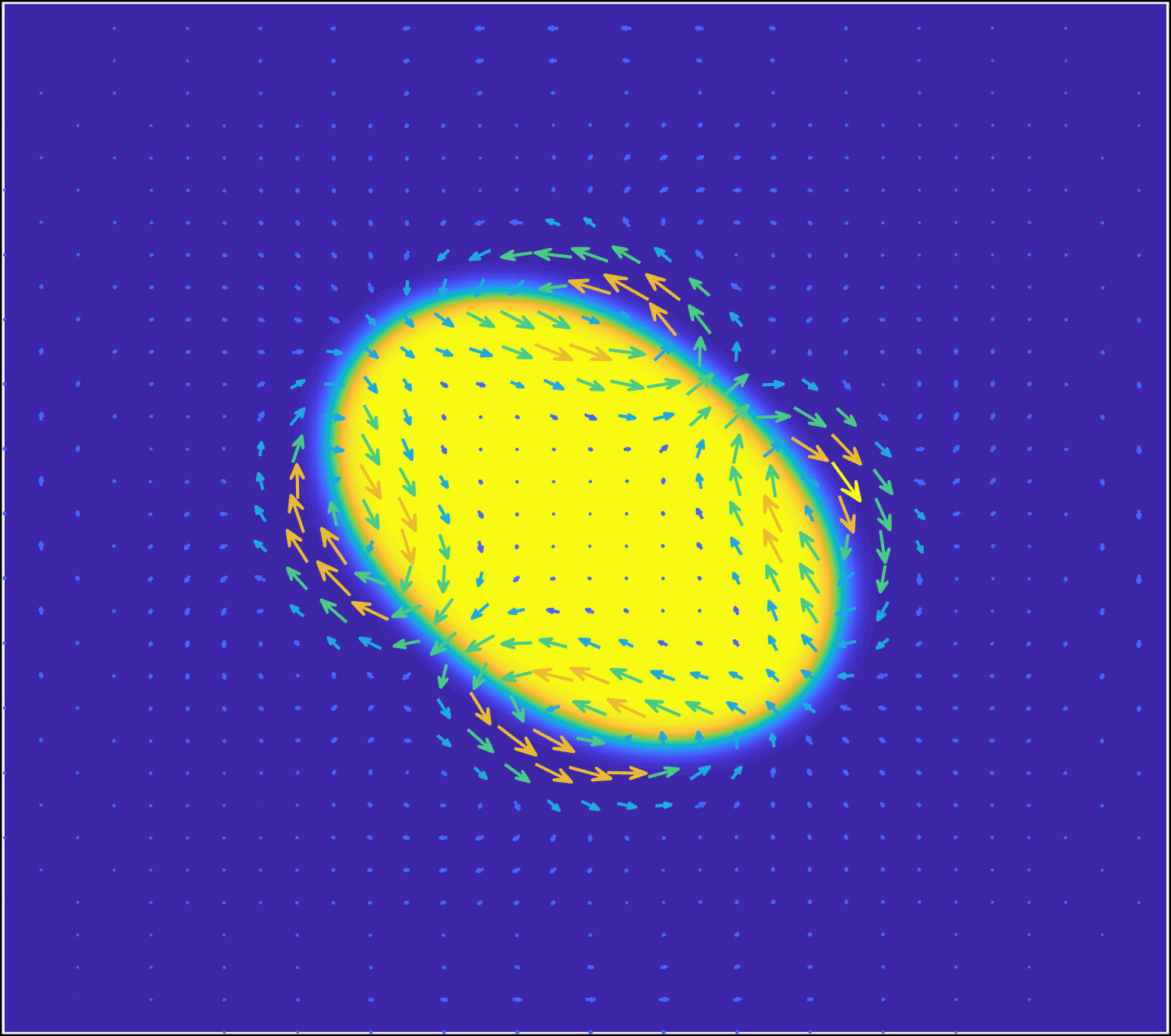}
		\includegraphics[width=0.18\textwidth]{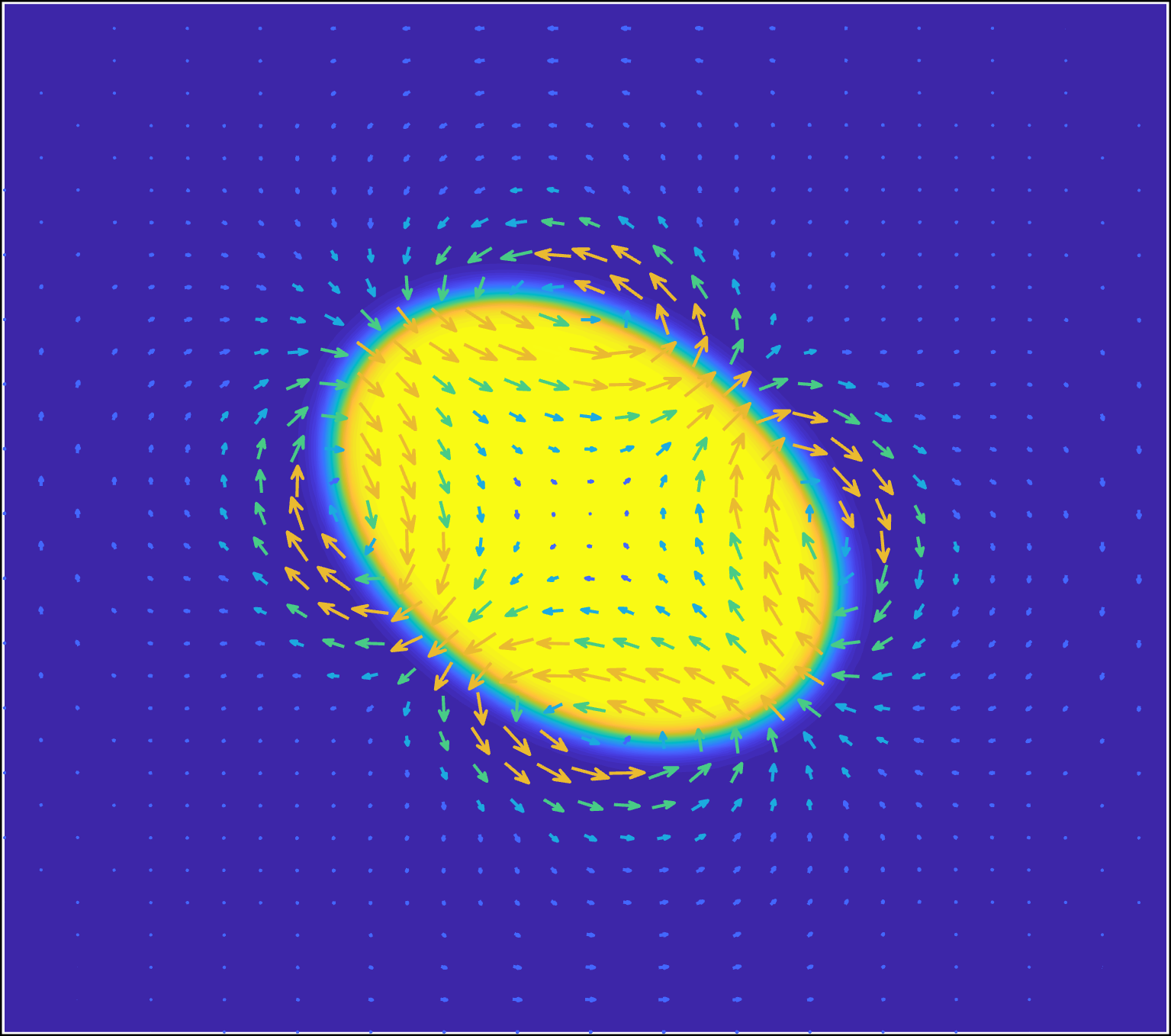}
		\includegraphics[width=0.18\textwidth]{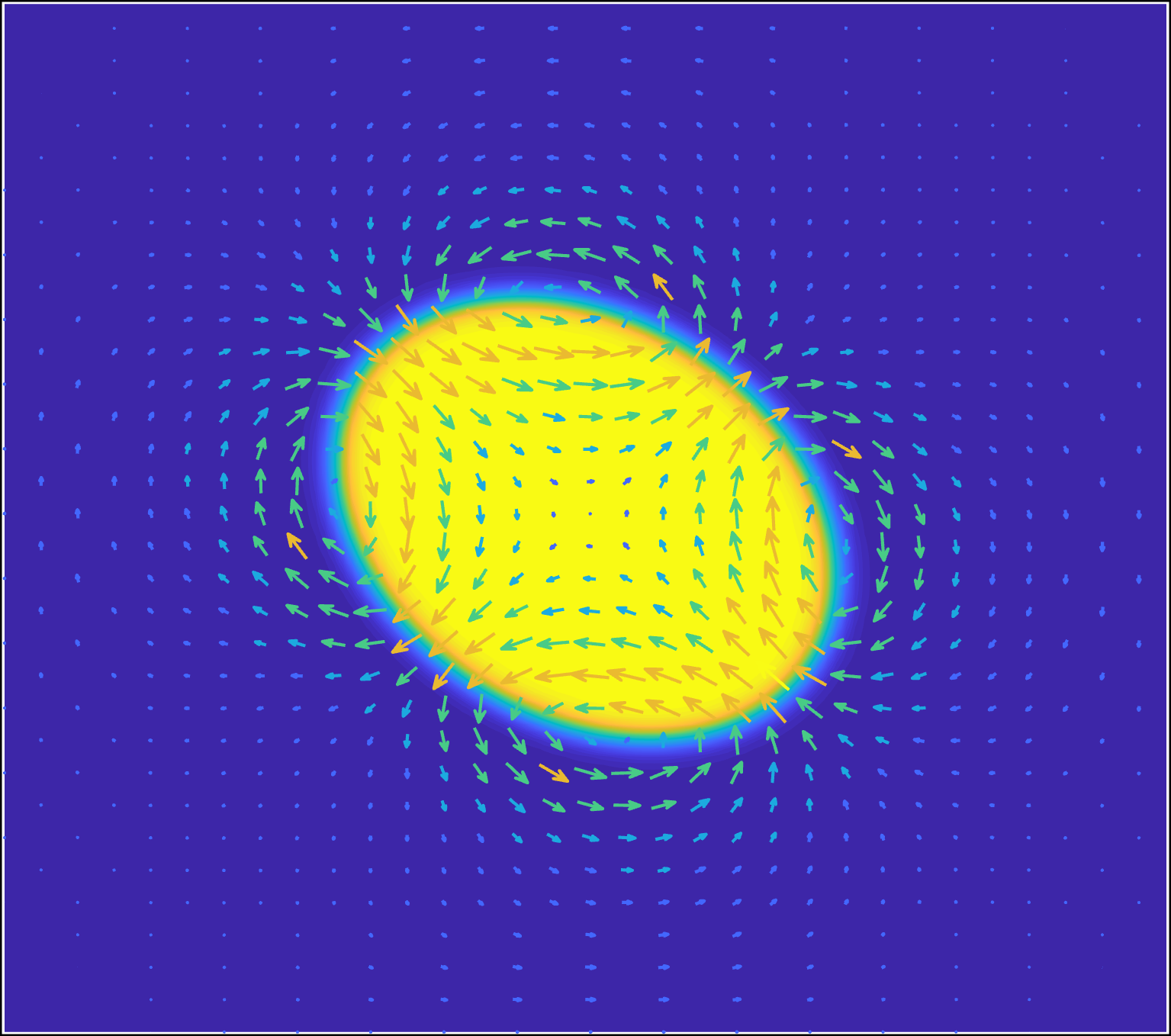}
		\includegraphics[width=0.18\textwidth]{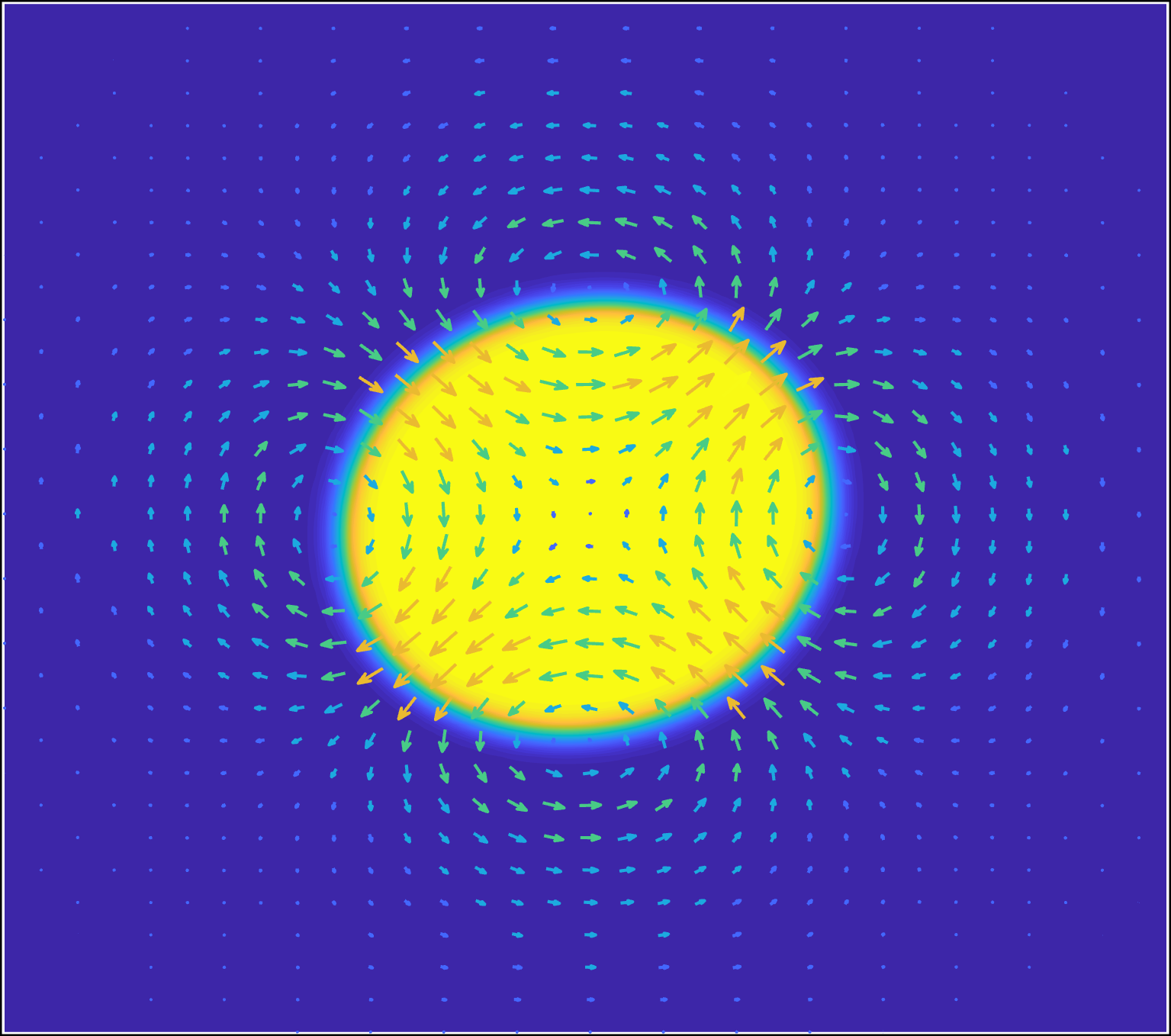}
	\end{center}
	\caption{Snapshots of the phase field $\phi$ and the corresponding velocity for the merged bubble with $\nu=0.001$ at selected time instances. The top row shows the evolution at $t = 0.001, 0.2, 0.4, 0.8, 1.0$; the middle row at $t = 1.2, 1.5, 2.0, 3.0, 3.2$; and the bottom row at $t = 3.4, 4.6, 4.8, 5.0, 10.0$.}\label{fig:chns-merged-bubble-2}
\end{figure}

\section{Conclusions}
We introduce the Skew Gradient Embedding (SGE) framework to systematically reformulate thermodynamically consistent partial differential equations (TCPDEs)-capturing both reversible and irreversible processes-into a generalized gradient flow structure. This reformulation enables the construction of thermodynamically consistent numerical schemes. The SGE framework provides a unified strategy for developing energy dissipation rate-preserving algorithms applicable to both conservative and dissipative systems. Efficient numerical schemes are constructed by explicitly discretizing the skew-symmetric (reversible) component, while stabilized schemes for the dissipative part are systematically derived using convex splitting techniques. For a broad class of multiphysics TCPDEs, the explicit treatment of the skew-symmetric term leads to natural decoupling of the governing equations, significantly improving computational efficiency without compromising stability or accuracy. Numerical experiments demonstrate the accuracy, robustness, and efficiency of the proposed methods. This systematic modeling and computational framework applies to a wide range of TCPDEs arising in electrodynamics, fluid mechanics, quantum mechanics, and statistical physics, offering a versatile toolset for algorithm design in computational science and engineering.

\section*{Acknowledgments}
Xuelong Gu's research is supported by an NSF award OIA-2242812. Qi Wang’s research is partially supported by NSF awards  DMS-2038080 and OIA-2242812, a DOE award DE-SC0025229, and an SC GEAR award.

\bibliographystyle{siamplain}
\bibliography{ref}

\end{document}